\documentclass[a4paper,12pt]{amsart}
\usepackage[italian,english]{babel}

\usepackage{geometry}
\usepackage{enumitem}
\usepackage{mathrsfs}
\usepackage{comment}
\usepackage[colorlinks, linkcolor=blue,anchorcolor=Periwinkle,
citecolor=blue,urlcolor=Emerald]{hyperref}

\usepackage{amsmath}
\usepackage{amsthm}
\usepackage[T1]{fontenc}
\usepackage[utf8]{inputenc}
\usepackage{amssymb}
\usepackage{csquotes}

\usepackage{mathrsfs}

\usepackage{graphicx}
\usepackage{tikz-cd}
\tikzset{dynkdot/.style={circle,draw,scale=.38}}

\newtheorem{prop}{Proposition}[section]
\newtheorem{lem}[prop]{Lemma}
\newtheorem{teo}[prop]{Theorem}
\newtheorem{cor}[prop]{Corollary}

\theoremstyle{definition}
\newtheorem{de}[prop]{Definition}

\newtheorem{rem}[prop]{Remark}
\newtheorem{ex}[prop]{Example}

\newtheorem{con}[prop]{Construction}

{\left\lbrace\begin{array}{@{}l@{}}}%
{\end{array}\right.}

\title[Solution of a problem in monoidal categorification]{Solution of a problem in monoidal categorification by additive categorification}


\author{Alessandro Contu}

\address{
Université Paris Cité, Sorbonne Université,
CNRS,
Institut de Mathématiques de Jussieu- Paris Rive Gauche,
F-75013 Paris, France
}
\email{alessandro.contu@imj-prg.fr}

\usepackage{comment}

\newcommand{\eps}{\varepsilon}

\begin{document}

\maketitle

\begin{abstract}
In 2021, Kashiwara--Kim--Oh--Park constructed cluster algebra structures on the
Grothendieck rings of certain monoidal subcategories of the category of finite-dimensional repre\-sentations of a quantum loop algebra, generalizing Hernandez--Leclerc's pioneering work from 2010.
They stated the problem of finding explicit quivers
for the seeds they used. We provide a solution by using Palu’s generalized mutation rule applied
to the cluster categories associated with certain algebras of global dimension at most 2, for example tensor products of path algebras of representation-finite quivers. Thus, our method is based on (and contributes to) the bridge, provided by cluster combinatorics,
between the representation theory of quantum groups and that of quivers with relations.
\end{abstract}

\tableofcontents

\section{Introduction}

Around the year 2000, Fomin and Zelevinsky introduced the study of cluster algebras, motivated by the aim of understanding Luzstig’s canonical bases of quantum groups and his closely related theory of total positivity. 
Cluster algebras are commutative algebras equipped with a set of distinguished generators, the \textit{cluster variables}, which are grouped into overlapping sets of fixed cardinality, the \textit{clusters}. Starting from an initial cluster, the others are constructed through a recursive procedure, whose elementary step is called \textit{mutation}.
The combinatorial information needed to mutate a given cluster is encoded in a quiver (or, equivalently, the so-called \textit{exchange matrix}) associated with it.  A cluster with the associated matrix is called a \textit{seed}.

Fomin--Zelevinsky's theory soon developed into an independent and successful research area, and cluster algebra structures were discovered in numerous contexts. 
In particular, in 2010, a link was established with the representation theory of quantum affine algebras: let $\mathfrak{g}$ be a finite-dimensional complex Lie algebra. In order to study the category $\mathscr{C}$ of finite-dimensional modules of type 1 of the quantum loop algebra $U_q(L\mathfrak{g})$, Hernandez and Leclerc exhibited \cite{Her_Lec_clust_alg_quat_aff} a cluster algebra structure on the Grothendieck rings of certain monoidal Serre subcategories $\mathscr{C}_l \subset \mathscr{C}$ defined for each level $l\in \mathbb{N}$. The two key ingredients in their approach were the so-called Kirillov--Reshetikin modules and a system of recurrence equations satisfied by their $q$-characters, the so-called $T$-system. In the language of cluster algebras, these correspond to certain cluster variables and certain exchange relations respectively. In particular, in the case of $\mathscr{C}_1$ in types $A$ and $D_4$, Hernandez--Leclerc proved that cluster monomials bijectively correspond to real simple modules and cluster variables to the tensor indecomposables ones among these. Their result was extended to all simply laced
finite root systems by Nakajima \cite{Nakajima_2011}.

Motivated by this example they introduced the notion of monoidal categorification of a cluster algebra, that is, the datum of a monoidal category whose Grothendieck ring is isomorphic to the cluster algebra in such a way that the above correspondences hold. In the following years, other subcategories of $\mathscr{C}$ were shown to provide monoidal categorifications: all the categories $\mathscr{C}^l_{\mathfrak{g}}\ (l\in \mathbb{N})$ (\cite{Qin_conjectures}), the category $\mathscr{C}^-_{\mathfrak{g}}$ \cite{HL_Clust_alg_approach_q_char}, and the categories $\mathscr{C}_\mathfrak{g}^{[a,b],\mathcal{D},\widehat{\underline{w}_0}}$ \cite{KKOP_mon_cat_quant_aff_II}, which include the previous ones as special cases.
Moreover, analogous results were obtained for the deformed versions of the Grothendieck rings, which were shown to have  quantum cluster algebra structures (see, for example,  \cite{Bittmann_quant_clust_alg_Cz-} and \cite{HFOO_propagation_positivity}).
Let us also mention that further classes of monoidal categorifications were constructed by Kang--Kashiwara--Kim--Oh \cite{KKKO_monoidal_cat_clust_alg} using representations of quiver Hecke algebras. This allowed them to show that in suitable quantum unipotent subgroups all cluster monomials belong to the canonical basis, thus confirming Fomin-Zelevinsky's main hope for these important examples. In their work, Kang--Kashiwara--Kim--Oh introduced the notion of a monoidal seed, that is, a collection of simple modules and an associated matrix (or, equivalently, a certain quiver) satisfying suitable compatibility conditions. Monoidal seeds are the monoidal avatar of the seeds of the associated cluster algebras. \\

This article is motivated by a problem raised in \cite{KKOP_mon_cat_quant_aff_II} concerning the categories $\mathscr{C}_\mathfrak{g}^{[a,b],\mathcal{D},\widehat{\underline{w}_0}}$. They are associated with the following data: an interval of integers $[a,b]$, a strong duality datum $\mathcal{D}$ arising from a simple Lie algebra $\mathsf{g}$ of simply-laced type, and an infinite sequence of indices obtained from a reduced expression $\underline{w}_0$ of the longest element of the Weyl group of $\mathsf{g}$.
In order to prove that the categories $\mathscr{C}_\mathfrak{g}^{[a,b],\mathcal{D},\widehat{\underline{w}_0}}$ are monoidal categorifications (under the assumption that the strong duality datum comes from a $\mathcal{Q}$-datum), they prove the existence of a whole family of monoidal seeds for each of these categories, starting from the initial monoidal seed used by Hernandez-Leclerc to prove the monoidal categorification given by $\mathscr{C}^-_{\mathfrak{g}}$. However, for any of these categories (and under the assumption that the interval $[a,b]$ is bounded on the right), they explicitly describe the quiver of only one monoidal seed. We denote these quivers by $Q^{[a,b]}(\underline{w}_0)$ (the quivers defined by Hernandez-Leclerc are special cases). At the end of \cite{KKOP_mon_cat_quant_aff_II}, they state the explicit determination of the exchange matrices of the other seeds as an open problem.\\

Our aim in the present paper is to exhibit natural additive categorifications corresponding to the categories $\mathscr{C}_\mathfrak{g}^{[a,b],\mathcal{D},\widehat{\underline{w}_0}}$ and to use this framework to solve the problem posed by Kim--Kashiwara--Oh--Park.

 Let us recall from \cite{Keller_clust_alg_quiv_rep_tring_cat} that an additive categorification of an (upper) cluster algebra $\mathcal{A}$ consists of a suitable small triangulated 2-Calabi--Yau category $\mathcal{C}$ and a \textit{decategorifcation map} from the set of objects of $\mathcal{C}$ to $\mathcal{A}$ which links the combinatorics of the rigid objects in $\mathcal{C}$ to those of the cluster monomials in $\mathcal{A}$. In particular, the decategorification map yields a bijection between (reachable, basic) cluster-tilting objects $T=T_1\oplus\dots\oplus T_n$ in $\mathcal{C}$ and the clusters $(u_1,\dots,u_n)$ in $\mathcal{A}$, with the indecomposable summands $T_i$ of $T$ being mapped to the cluster variables $u_i$. Moreover, the quiver of the seed corresponding to $u$ can be recovered as the \emph{endoquiver} of $T$ (i.e. the quiver of its endomorphism algebra). Thus, if we are given an additive categorification $\mathcal{C}$ and a monoidal categorification $\mathcal{M}$ of the same cluster algebra $\mathcal{A}$, we obtain a highly non trivial interplay between the triangulated structure of $\mathcal{C}$ and the monoidal structure of $\mathcal{M}$, which has already proved very promising in the recent papers \cite{baur_2024_arX_correspondenceadditivemonoidalcategorifications, duan_Schiff2023real, Fujita_singularitiesnormalizedrmatriceseinvariants_arX}, $cf$ also \cite{Fujita_Murakami_deformed_Cartan_generalized_preproj_I} and \cite{Fujita_Murakami_deformed_Cartan_generalized_preproj_II}.

In this paper, we apply this philosophy to the cluster algebras monoidally categorified by the categories $\mathscr{C}_\mathfrak{g}^{[a,b],\mathcal{D},\widehat{\underline{w}_0}}$. In particular, to solve the problem posed by \cite{KKOP_mon_cat_quant_aff_II}, we appeal to Palu's generalized mutation rule \cite{Palu02}, which allows to compute the endoquiver of any cluster-tilting object if we are given an initial cluster-tilting object whose endoquiver is known.

In more detail, we construct the additive avatar of $\mathscr{C}_\mathfrak{g}^{[a,b],\mathcal{D},\widehat{\underline{w}_0}}$ as the generalized cluster category $\mathcal{C}_A$ where $A$ is an algebra of global dimension $\leq 2$ given by 
 the quotient of the path algebra of the quiver $Q^{[a,b]}(\underline{w}_0)$ obtained by replacing certain arrows with relations. Our construction of $A$ ensures that the endoquiver of the image of the free module $A_A$ in $\mathcal{C}_A$ is exactly $Q^{[a,b]}(\underline{w}_0)$.

 Next, for each monoidal seed $\mathcal{S}$ appearing in Kashiwara--Kim--Oh--Park's problem, by iterated mutations we construct a corresponding cluster-tilting object in $\mathcal{C}_A$. We then determine its endoquiver using Palu's generalized mutation rule. Since mutations in the additive setting correspond to those in the monoidal setting, this endoquiver is indeed the one of the monoidal seed $\mathcal{S}$.

Let us briefly summarize the content of the sections of this paper:
\begin{itemize}
    \item in Section 2, we recall the basics of the representation theory of quantum loop algebras, introduce the categories $\mathscr{C}_\mathfrak{g}^{[a,b],\mathcal{D},\widehat{\underline{w}_0}}$  and discuss the properties of $i$-boxes, a combinatorial tool invented by Kashiwara-Kim-Oh-Park;
    \item in Section 3, we recall the combinatorics of $Q$-data, a generalization of Dynkin quivers introduced by Fujita-Oh \cite{Fuj_Oh_2021}, where the orientation is replaced with the datum of a height function on the Dynkin diagram. These allow us to define a class of strong duality data in the sense of \cite{KKOP_mon_cat_quant_aff_II};
    \item in Section 4, we recall the necessary background on monoidal categorification. Furthermore, we state the open problem of Kashiwara--Kim--Oh--Park \cite{KKOP_mon_cat_quant_aff_II};
    \item in section 5, we recall the necessary tools from additive categorification;
    \item in section 6, we present our solution of Kashiwara--Kim--Oh--Park's problem.
\end{itemize}

\section{Module categories of quantum loop algebras}
\label{sect_module_categories}
In this section, we recall the definition of the quantum loop algebra of a simple complex Lie algebra
and of some important monoidal subcategories of its module category. 

Let $q$ be an indeterminate. We denote by $\mathbb{K}$ the algebraic closure of the field $\mathbb{Q}(q)$. We write $\mathbb{K}^\times$ for the set of its invertible elements. Let $\mathfrak{g}$ be a finite-dimensional simple complex Lie algebra. Let $I_\mathfrak{g}=\{1, \ldots, n\}$ be the set of vertices of its Dynkin diagram,
$A$ the associated Cartan matrix and $D$ the integer diagonal $n\times n$-matrix with strictly positive diagonal entries $d_i$ with greatest common divisor $1$ such that $DA$ is symmetric. 

The quantum loop algebra associated to $\mathfrak{g}$ is the $\mathbb{K}$-algebra defined via the generators 
\[ \{K_i^{\pm}| i\in I_{\mathfrak{g}}\}\cup \{X_{i,s}, Y_{i,s}| i\in I_{\mathfrak{g}}, s\in \mathbb{Z} \}\cup \{H_{i,r}|i\in I_{\mathfrak{g}}, r \in \mathbb{Z}-\{0\}\}
\] 
subject to the relations given in \cite[Definition 2.1]{HFOO_propagation_positivity}. 
We denote it by $U_q(L\mathfrak{g})$.

\subsection{The category $\mathscr{C}_\mathfrak{g}^\mathbb{Z}$}\label{subsec_categoryC0}
We denote by $\mathscr{C}_\mathfrak{g}$ the category of finite-dimensional $U_q(L\mathfrak{g})$-modules of type 1, that is, 
the finite-dimensional $U_q(L\mathfrak{g})$-modules where the generators $K_1, \dots, K_n$ act semisimply, 
with eigenvalues which are integer powers of $q$.  
The theory of $q$-characters (see \cite{FrenkelReshetikhin98}) provides us with a parametrisation
$m \mapsto L(m)$ of the simple objects in $\mathscr{C}_\mathfrak{g}$ by the monomials $m$ 
in the variables $\{Y_{i,a} \ |\ i\in I,a\in \mathbb{K}^\times \}$.

For each vertex $i$ of the Dynkin diagram of $\mathfrak{g}$, we choose $\Tilde{\varepsilon}_i\in \{0,1\}$ such that
we have $\Tilde{\varepsilon}_i\equiv\Tilde{\varepsilon}_j +\text{min}(d_i,d_j) \text{ mod 2}$ whenever the vertices $i$ and
$j$ are linked by an edge. Define $\widehat{I}_\mathfrak{g}$ to be the set
$$\widehat{I}_\mathfrak{g}=\{(i,p)\in I\times\mathbb{Z}\ |\ p-\Tilde{\varepsilon}_i\in 2\mathbb{Z}\}.$$
Let $\mathscr{C}_\mathfrak{g}^\mathbb{Z}$ be the monoidal Serre subcategory of $\mathscr{C}_\mathfrak{g}$ whose simple objects are the
$L(m)$, where $m$ is a monomial in the $Y_{i,q^p}$ associated with the vertices $(i,p)\in \widehat{I}_\mathfrak{g}$.

In the following, we will focus on subcategories of the category $\mathscr{C}_\mathfrak{g}^\mathbb{Z}$ and use the 
simplified notation $Y_{i,p}=Y_{i,q^p}$.
For  $(i,p)\in \widehat{I}_\mathfrak{g}$ and an integer $k\geq 1$, we define the 
\emph{Kirillov--Reshetikhin module $W_{k,p}^i\in \mathscr{C}_\mathfrak{g}^\mathbb{Z}$} as
\[W_{k,p}^i=L(Y_{i,p}Y_{i,p+2d_i}\dots Y_{i,p+2d_i(k-1)}).\]
For $k=0$, the module $W_{k,p}^i$ is defined to be the trivial module.

\begin{de}
Let $M$ and $N$ be simple objects of $\mathscr{C}_\mathfrak{g}$. 
\begin{enumerate}
    \item $M$ and $N$ \emph{commute} if there is an isomorphism $M\otimes N \cong N\otimes M$.
    \item $M$ and $N$ \emph{strongly commute} if $M\otimes N$ is simple.
    \item M is \emph{real} if it strongly commutes with itself.
\end{enumerate}
\end{de}

\begin{rem}\mbox{}
\begin{itemize}
    \item[a)]  Since the Grothendieck ring 
    $\mathcal{K}(\mathscr{C}_\mathfrak{g})$ is commutative with a basis formed by the classes of the simple objects, 
    if two simple objects $M$ and $N$ strongly commute, then they commute.
    \item[b)] Let $m$ be a positive integer. As shown in \cite{Hernandez_simple_tensor_products}, if $(M_i)_{1\leq i\leq m}$ is a family of simple objects of $\mathscr{C}_\mathfrak{g}$, then they pairwise strongly commute if and only if $M_{\sigma(1)}\otimes \dots \otimes M_{\sigma(m)}$ is simple for any permutation $\sigma$ of the indices $\{1,\dots,m\}$.
\end{itemize}
\end{rem}

\subsection{Duality datum and PBW pairs.} \label{subsect_Duality_datum}
Let $C=(c_{\imath \jmath})_{\imath,\jmath\in J}$ be a generalized Cartan matrix with index set $J$. Beware that
$C$ is usually {\em not} the Cartan matrix of $\mathfrak{g}$. A {\em duality datum} for $\mathfrak{g}$ associated to $C$ is a collection of simple modules $\mathcal{D}=\{L(m_\jmath)\}_{\jmath\in J}$ of $\mathscr{C}_\mathfrak{g}^\mathbb{Z}$ satisfying certain properties (see \cite[Section 4]{kashiwara2021pbw}). 
We refer to \cite[Def. 4.6]{kashiwara2021pbw} for the notion of a {\em strong duality datum}, which is defined when $C$
is of simply laced type.

From now on, we suppose that $C$ is of simply laced type.  Let the symbol $\mathsf{g}$ (not to be confused with $\mathfrak{g}$) 
denote a finite-dimensional simple Lie algebra with Cartan matrix $C$. 
Let $I_{\mathsf{g}}$ be the set of  vertices of its Dynkin diagram. We identify the elements $\imath$ of $J$ with these vertices and
write $\imath \sim \jmath$ if they correspond to adjacent vertices.
Let $w_0$ be the longest element of the Weyl group of $\textsf{g}$ and denote its length by $l_0$. Define the involution $(-)^*: I_\mathsf{g}\rightarrow I_{\mathsf{g}}$
by $w_0(\alpha_{\imath})=-\alpha_{\imath^*}$.

Choose a reduced expression $\underline{w}_0=s_{\imath_1}\ \dots\ s_{\imath_{l_0}}$ of $w_0$. We extend the sequence
$\imath_1, \ldots, \imath_{l_0}$ to an infinite sequence $\widehat{\underline{w}}_0=(\imath_k)_{k\in\mathbb{Z}}$ by requiring
that we have 
\[
\imath_{k+{l_0}}=(\imath_k)^*
\]
for all $k\in \mathbb{Z}$. If $\mathcal{D}$ is a strong duality datum associated to the Cartan matrix of $\mathsf{g}$, we call $(\mathcal{D},\underline{w}_0)$ a PBW-pair. For each PBW-pair $(\mathcal{D},\underline{w}_0)$, Kashiwara--Kim--Oh--Park have introduced the family of \textit{affine cuspidal modules} $(S_{k}^{\mathcal{D}, \widehat{\underline{w}}_0})_{k\in\mathbb{Z}}$, 
cf.~\cite[Def. 4.5]{KKOP_mon_cat_quant_aff_II}. For each (possibly infinite) integer interval $K\subseteq \mathbb{Z}$, they defined 
$\mathscr{C}_{\mathfrak{g}}^{K,\mathcal{D},\underline{w}_0}$
to be the monoidal Serre subcategory of $\mathscr{C}_\mathfrak{g}^\mathbb{Z}$ generated by the affine cuspidal modules 
$S_{k}^{\mathcal{D},\widehat{\underline{w}}_0}$, for $k\in K$, cf.~\cite[$\S$ 6.3]{kashiwara2021pbw}.
For our purposes, we will not need the precise definition of these modules, but we recall how to characterize them
in a particular case in Section \ref{section_Q_data}.

\subsection{Combinatorics of $i$-boxes}
In this subsection, following \cite{KKOP_mon_cat_quant_aff_II}, we define and discuss the properties of $i$-boxes.  
These are combinatorial objects which will serve to parametrize certain simple modules, cf.~Definition~\ref{def_aff_det_mod}.
We fix a PBW-pair $(\mathcal{D},\underline{w}_0)$. Recall that we have fixed a sequence $\widehat{\underline{w}}_0=(\imath_k)_{k\in\mathbb{Z}}$.
When dealing with the indices of $\widehat{\underline{w}}_0$, for $s\in \mathbb{Z}$ and $\jmath\in I_\mathsf{g}$,  we will use the following notations: 

\begin{align}
s^+& =\text{min}\{t\;|\;s<t,\ \imath_t=\imath_s\}, \quad\quad   s^-=\text{max}\{t\;|\;s>t,\ \imath_t=\imath_s\},\\
s(\jmath)^+&=\text{min}\{t\;|\;s\leq t,\ \imath_t=\jmath\}, \quad\quad  s(\jmath)^-=\text{max}\{t\;|\;s\geq t,\ \imath_t=\jmath\}.
\end{align}

For $a\leq b\in\mathbb{Z}\sqcup\{\pm\infty\}$, as in \cite{KKOP_mon_cat_quant_aff_II},  we define the integer interval $[a,b]$ by
\[ [a,b]=\{k\in\mathbb{Z} \;|\; a\leq k\leq b\}.\]
If $a$ and $b$ are integers, the \textit{length} of  the interval $[a,b]$ is $l=b-a+1$. Otherwise, we say that the interval
has infinite length.

\begin{de}[{\cite[Section 4]{KKOP_mon_cat_quant_aff_II}}] \mbox{}
\begin{itemize} 
    \item[-] An {\em $i$-box} is a finite integer interval $[a,b]$ such that $\imath_a=\imath_b$. 
    \item[-] For an $i$-box $[a,b]$, we define its {\em index} as $\jmath=\imath_a=\imath_b$. When we want to emphasize that an $i$-box is of index $\jmath$, we use the notation $[a,b]_\jmath$.
    \item[-] The {\em $i$-cardinality} of an $i$-box $[a,b]_\jmath$ is the number of times that the index $\jmath$ appears in the sub-interval of $\widehat{\underline{w}}_0$ corresponding to $[a,b]$.
    \item[-] For a finite interval $[a,b]$, we define $[a,b\}$ and $\{a,b]$ as the largest $i$-boxes contained in $[a,b]$ with index $\imath_a$ and $\imath_b$ respectively. In other terms, we have
\[ [a,b\}=[a,b(\imath_a)^-] \text{  and  } \{a,b]=[a(\imath_b)^+,b]. \]
\end{itemize}
\end{de}

In the following, all $i$-boxes will be defined with respect to the sequence 
$\widehat{\underline{w}}_0=(\imath_k)_{k\in\mathbb{Z}}$ fixed above.

\begin{de}[{\cite[Def. 5.1]{KKOP_mon_cat_quant_aff_II}}]
Let $[a,b]$ be a finite interval of length $l$.
A \textit{chain} of $i$-boxes of \textit{range} $[a,b]$ is a sequence of $i$-boxes $\mathfrak{C}=(\mathfrak{c}_k)_{1\leq k\leq l}$ contained in $[a,b]$ and satisfying the following conditions for any $1\leq s\leq l$:
\begin{itemize}
    \item[(i)] The union $\bigcup_{1\leq k\leq s} \mathfrak{c}_k $ is an interval of length $s$;
    \item [(ii)] The $i$-box $\mathfrak{c}_s$ is the largest $i$-box of its own index contained in the interval $\bigcup_{1\leq k\leq s} \mathfrak{c}_k $.
\end{itemize}
Similarly, for an interval $[a,b]$ of infinite length $l=\infty$, a \textit{chain} of $i$-boxes of \textit{range} $[a,b]$ is a sequence
of $i$-boxes $\mathfrak{C}=(\mathfrak{c}_k)_{1\leq k< l}$ contained in $[a, b]$ and satisfying (i) and (ii) for any $1\leq s<\infty$.
By abuse of terminology, we call $l$ the \textit{length of the chain}.
Clearly, for any  $1\leq s \leq l$, the sequence $(\mathfrak{c}_k)$, where $1\leq k\leq s$ for $s<\infty$ and $1\leq k <s$ for $s=\infty$, is a chain of $i$-boxes. We refer to it as a \textit{subchain} of $\mathfrak{C}$ of length $s$.
\end{de}

\begin{rem} \label{rem_exp_op} 
Let $[a,b]$ be a finite interval of length $l$. It follows from the previous definition that we have a bijection between
the chains of $i$-boxes  in $[a,b]$ and the pairs consisting of a singleton 
$\{c\}\subseteq [a,b]$ and a sequence $E\in \{L,R\}^{l-1}$. Namely, for a given chain $\mathfrak{C}=(\mathfrak{c}_k)_{1\leq k\leq l}$,
we put $\{c\}=\mathfrak{c}_1$ and, for $1\leq k<l$, define $E_k$ to be $L$ respectively $R$ if adding $\mathfrak{c}_{k+1}$ increases the range of the interval $\bigcup_{1\leq k\leq s} \mathfrak{c}_k $ by one unit on the left respectively on the right. We call
$L$ resp.~$R$ the \emph{left} resp.~\emph{right expansion operator} and refer to the sequence $E$ together with
the $i$-box $[c]$ as the \emph{rooted sequence
of expansion operators} associated with the chain of $i$-boxes $\mathfrak{C}$. We leave it to the reader to formulate
the natural extension of this bijection to the case of an infinite interval $[a,b]$.
\end{rem}

For $a\in \mathbb{Z}\cup \{-\infty\}$ and  $b\in\mathbb{Z},\ b\geq a$,  
we denote by $\mathfrak{C}_{-}^{[a,b]}$ the chain of $i$-boxes 
\[
\mathfrak{C}_{-}^{[a,b]}=(\mathfrak{c}^-_k)_{1\leq k \leq b-a+1}=([s,b\})_{b\geq s\geq a } \ .
\] 
It is associated to the pair $(b;(E_k)_{1\leq k\leq b-a})$, where $E_k=L$ for any $k$.

\begin{rem}
    Let $[a,b]$ be a finite interval of length $l$. Let $\mathfrak{C}$ be a chain of $i$-boxes of range $l$, associated to the pair 
    $([c],(E_k)_{1\leq k\leq l-1})$. Then the integer $c\in [a,b]$ is determined by the sequence $(E_k)_{1\leq k\leq l-1}$. In fact, 
    if we denote by $\mathcal{R}$ the set $\{ k\in [1,l-1]\ |\ E_k=R\}$, then $c$ is equal to $b-|\mathcal{R}|$.
\end{rem}

Next, we recall a combinatorial operation on chains of $i$-boxes from \cite[\S 5.2]{KKOP_mon_cat_quant_aff_II}.
\begin{de}
\label{de_box_move}
Let $\mathfrak{C}=(\mathfrak{c}_k)$ be a chain of $i$-boxes of length $l\leq \infty$ corresponding to a pair $(c,(E_k)_{1\leq k<l})$. 
\begin{itemize}
\item[(i)] For $1 \leq s < l$, the $i$-box $\mathfrak{c}_s$ is defined to be \textit{movable} if $s=1$ or  $s\geq 2$ and $E_{s-1}\neq E_s$.
\item[(ii)] For a movable $i$-box $\mathfrak{c}_s$, the \textit{box move} at $s$, denoted by $\nu_s$, is the operation sending $\mathfrak{C}$ 
to the chain $\nu_s\mathfrak{C}$, whose associated pair $(c',E')$ is defined as follows:
\[
c'=\begin{cases}
c+1 & \text{if } s=1, E_1=R,\\
c-1 & \text{if } s=1, E_1=L,\\
c & \text{if } s>1,
\end{cases} \ \text{   and   }\
E'_k=\begin{cases}
    R & \text{if } E_k=L,\ k\in\{s-1,s\},\\
    L & \text{if } E_k=L,\ k\in\{s-1,s\},\\
    E_k & \text{if } k\notin\{s-1,s\}.
\end{cases}
\]
\item[(iii)] We call a finite composition of box moves a \textit{chain transformation}.
\end{itemize}
\end{de}

The next lemma shows that box moves are involutions. 
\begin{lem}
Let $\mathfrak{C}=(\mathfrak{c}_k)$ be a chain of $i$-boxes of length $l\leq \infty$ and $1\leq s<l$ an integer such that $\mathfrak{c}_s$ is movable. Then the $i$-box $\mathfrak{c'}_s$ of $\nu_s\mathfrak{C}$ is movable and $\nu_s\nu_s\mathfrak{C}=\mathfrak{C}$.
\end{lem}
\begin{proof}
    This can be checked directly using the definition of box move.
\end{proof}

\begin{ex}
Let $\mathsf{g}$ be of type $A_3$ and choose the reduced expression $\underline{w}_0=s_1s_3s_2s_3s_1s_2$. The associated 
sequence $\widehat{\underline{w_0}}$ is 
\[
 \widehat{\underline{w_0}}=\dots \underbrace{\ 1,\ 3,\ 2,\ 3,\ 1,\ 2,\ 3,\ 1,\ 2,\ 1,\ 3,\ 2\,}_{[1,12]}, \  
 \dots .\]
 Consider the chain of $i$-boxes $\mathfrak{C}=(\mathfrak{c}_k)_{1\leq k\leq 6}$ of range $[1,6]$ defined by 
  \begin{align*}
\mathfrak{c}_1&=[6]_2, &  \mathfrak{c}_2&=[5]_1, &\mathfrak{c}_3&=[4]_3, \\
\mathfrak{c}_4&=[3,6]_2, &  \mathfrak{c}_5&=[2,4]_3, &\mathfrak{c}_6&=[1,5]_1.
\end{align*} 
It corresponds to the pair $(6;(L,L,L,L,L))$. According to the above definition, the only movable $i$-box in the chain $\mathfrak{C}$ is $\mathfrak{c}_1$. The box move at 1 sends $\mathfrak{C}$ to the chain $\nu_1\mathfrak{C}=(\mathfrak{c}_k')_{1\leq k\leq 6}$, with 

\begin{align*}
\mathfrak{c'}_1&=[5]_1, &  \mathfrak{c'}_2&=[6]_2, &\mathfrak{c'}_3&=[4]_3, \\
\mathfrak{c'}_4&=[3,6]_2, &  \mathfrak{c'}_5&=[2,4]_3, &\mathfrak{c'}_6&=[1,5]_1.
\end{align*} 

It corresponds to the pair $(5;(R,L,L,L,L))$. The chain $\nu_1\mathfrak{C}$ has two movable $i$-boxes, namely $\mathfrak{c'}_1$ and $\mathfrak{c'}_2$. Performing the box move at 1 returns the original chain $\mathfrak{C}$, while the box move at 2 returns the chain with associated pair $(5;(L,R,L,L,L))$.
\end{ex}

\begin{rem}
\label{rem_can_chain_trans}
Proceeding as in the above example (that is, replacing a left operator with a right operator via the box move at 1 and shifting it in the proper position via box moves at the other indices), any chain of $i$-boxes of finite range $[a,b]$ can be obtained through a chain transformation from the chain of $i$-boxes $\mathfrak{C}_{-}^{[a,b]}$.
Let us make this more precise: let $\mathfrak{C}=(\mathfrak{c}_k)$ be a chain of $i$-boxes on the interval $[a,b]$. 
Let $k_1 < k_2 < \ldots$ be the elements of the set
 \[
 \mathcal{R}=\{ k\in [1,l-1]\ |\ E_k=R\}.
 \]
We set $\mathfrak{C}_0=\mathfrak{C}_{-}^{[a,b]}$ and, for any $1\leq s\leq |\mathcal{R}|$, we define $\mathfrak{C}_s$ as the chain of 
 $i$-boxes of range $[a,b]$ with associated sequence of expansion operators
\[ E_k=\begin{cases}
     R,\text{    if } k\in \mathcal{R}, k\geq k_{|\mathcal{R}|-s+1}\\
     L, \text{    otherwise.}
 \end{cases}
 \]
Then we have, for any $1\leq s\leq |\mathcal{R}|$, the chain transformation 
 
 \begin{equation}
 \label{eq_basic_chain_trans}
     \nu_{k_{|\mathcal{R}|-s+1}}\circ \nu_{k_{|\mathcal{R}|-s}}\circ \dots \circ \nu_2\circ\nu_1: \mathfrak{C}_{s-1}\mapsto \mathfrak{C}_{s}. 
 \end{equation}
 
 By composition, we obtain a chain transformation 
 \begin{equation}
 \label{eq_can_chain_trans}
 \mathfrak{C}_{-}^{[a,b]} \mapsto \mathfrak{C}.
 \end{equation}
 We call a chain transformation of the form (\ref{eq_basic_chain_trans}) \emph{basic} and one of the form (\ref{eq_can_chain_trans}) \emph{canonical}. Notice that, by composing a canonical chain transformation with the inverse of a canonical chain transformation, 
we obtain a chain transformation between any pair of chains of $i$-boxes with the same finite range.
\end{rem}

\begin{lem}
Let $\mathfrak{C}=(\mathfrak{c}_k)$ be a chain of $i$-boxes of length $l$ such that the $i$-box $\mathfrak{c}_s$ is movable, with $2\leq s< l-1$. The following are equivalent:
\begin{itemize}
    \item[(i)] The $i$-box $\mathfrak{c}_{s+1}$ is equal to the union $\bigcup_{1\leq k\leq s+1}\mathfrak{c}_k$.
    \item[(ii)] The $i$-boxes $\mathfrak{c}_s$ and $\mathfrak{c}_{s+1}$ have the same index.
\end{itemize}
\end{lem}

\begin{proof}
Let $\mathfrak{c}_s=[a_s,b_s]$, $\mathfrak{c}_{s+1}=[a_{s+1},b_{s+1}]$ and $\cup_{1\leq k\leq s+1}\mathfrak{c}_k=[\Tilde{a}_s,\Tilde{b}_s]$. By the definition of chain of $i$-boxes, we either have $a_{s+1}=\Tilde{a}_s$ or $b_{s+1}=\Tilde{b}_s$. Moreover, since the $i$-box $\mathfrak{c}_s$ is movable, these cases correspond to $b_{s}=\Tilde{b}_s$ or $a_{s}=\Tilde{a}_s$ respectively. 
In the first case, the equality $\mathfrak{c}_{s+1}=\cup_{1\leq k\leq s+1}\mathfrak{c}_k$ holds if and only if $b_{s+1}=\Tilde{b}_s$, 
which is equivalent to $\mathfrak{c}_{s}$ and $\mathfrak{c}_{s+1}$ having the same index $\imath_{\Tilde{b}_s}$. The second case is analogous. 
\end{proof}

The following result clarifies the effects of the box move on the $i$-boxes of a chain.
For an $i$-box $[a,b]$, we define
\begin{align*}
    \tau [a,b] &= [a^+,b^+] \\
    \tau^- [a,b] &= [a^-,b^-] \\
    \tau^{(\pm i)} [a,b] &= \underbrace{\tau^\pm\dots \tau^\pm}_{i\text{-times}} [a,b],\ \ i\in \mathbb{N}.
\end{align*}

\begin{prop}[{\cite[Prop.5.6, Prop.5.7]{KKOP_mon_cat_quant_aff_II}}]
\label{prop_box_move}
Let $\mathfrak{C}=(\mathfrak{c}_k)$ be a chain of $i$-boxes of length $l$, with expansion operators $(E_k)_{1\leq k<l}$. 
Assume that $\mathfrak{c}_s$ is movable.
\begin{itemize}
    \item[(i)] If the $i$-box $\mathfrak{c}_{s+1}$ has the same index as $\mathfrak{c}_s$, then
    
   \[ (\nu_s(\mathfrak{C}))_k=\begin{cases}
     \tau \mathfrak{c}_s & \text{if } k=s \text{ and } E_k=R,\\
     \tau^- \mathfrak{c}_s & \text{if } k=s \text{ and } E_k=L,\\
     \mathfrak{c}_k & \text{if } k\neq s.
    \end{cases}\]
    
    \item[(ii)] If the $i$-boxes $\mathfrak{c}_{s+1}$ and $\mathfrak{c}_s$ have different indices, then 
    \[ (\nu_s(\mathfrak{C}))_k=\begin{cases}
     \mathfrak{c}_{s+1} & \text{if } k=s,\\
     \mathfrak{c}_s & \text{if } k=s+1,\\
     \mathfrak{c}_k & \text{if } k\neq s, s+1.
    \end{cases}\]
\end{itemize}
\end{prop}

\begin{rem}
In the setting of Prop.~\ref{prop_box_move}, if $\mathfrak{c}_{s+1}=[a_{s+1},b_{s+1}]$ and $\mathfrak{c}_s$ have the same index, 
then we have two possible scenarios:
\[(1) \begin{cases}
    \mathfrak{c}_s=[a_{s+1},b_{s+1}^-],\\
    E_s=R\\
    (\nu_s(\mathfrak{C}))_s=[a_{s+1}^+,b_{s+1}];
\end{cases}\ \ \ \ \ \ (2) 
\begin{cases}
 \mathfrak{c}_s=[a_{s+1}^+,b_{s+1}],\\
    E_s=L\\
    (\nu_s(\mathfrak{C}))_s=[a_{s+1},b_{s+1}^-].
\end{cases}\]
In particular, the box move shifts $\mathfrak{c}_s$ inside $\mathfrak{c}_{s+1}$ from flush right to flush left in
scenario $(1)$ and from flush left to flush right in scenario $(2)$, preserving its index in both cases.
\end{rem}

Notice that the definitions of $i$-box and of chain of $i$-boxes require only the datum of the infinite sequence $\widehat{\underline{w}}_0$. We use the strong duality datum $\mathcal{D}$ to associate to each $i$-box a module in the category $\mathscr{C}_\mathfrak{g}^\mathbb{Z}$ as follows:

\begin{de}[{\cite[Def.~4.14]{KKOP_mon_cat_quant_aff_II}}] \label{def_aff_det_mod}
Let $\mathfrak{c}=[a,b]$ be an $i$-box. The \textit{affine determinant module} 
$M(\mathfrak{c})=M^{\mathcal{D},\widehat{\underline{w}}_0}(\mathfrak{c})$ is defined as the
 head of the module
\[ 
    S^{\mathcal{D},\widehat{\underline{w}}_0}_b\otimes S^{\mathcal{D},\widehat{\underline{w}}_0}_{b^-}\otimes
    S^{\mathcal{D},\widehat{\underline{w}}_0}_{(b^-)^-} \otimes \dots \otimes 
    S^{\mathcal{D},\widehat{\underline{w}}_0}_{a^+}\otimes S^{\mathcal{D},\widehat{\underline{w}}_0}_a\in \mathscr{C}^{[a,b],\mathcal{D},\underline{w}_0}_\mathfrak{g}.
\]
\end{de}

The next proposition shows that the combinatorics of $i$-boxes reflect the behaviour of the affine determinant
modules with respect to tensor products. 
If $\mathfrak{C}$ is a chain of $i$-boxes, we write $M(\mathfrak{C})$ for the family of modules $M(\mathfrak{c})$, where
$\mathfrak{c}$ belongs to $\mathfrak{C}$.

\begin{prop}[{\cite[Thm.~5.5]{KKOP_mon_cat_quant_aff_II}}]
\label{prop_properties_aff_det_mod}
For any $i$-box $\mathfrak{c}$, the module $M(\mathfrak{c})$ is simple and real.
Moreover, for each chain of $i$-boxes $\mathfrak{C}$, the modules in the family $M(\mathfrak{C})$ strongly commute.
\end{prop}

The following result is a generalization of the $T$-system for Kirillov–Reshetikhin modules. 
It can be interpreted as a monoidal avatar of certain exchange relations in a suitable cluster algebra.

\begin{prop}[{\cite[Thm.~4.25]{KKOP_mon_cat_quant_aff_II}}] \label{prop_T-system}
For each $i$-box $[a,b]$, there is a short exact sequence
\begin{equation}\label{eq_T_su}
    0\rightarrow \bigotimes_{\imath\sim \jmath} M[a(\jmath)^+,b(\jmath)^-]\rightarrow M[a^+,b]\otimes M[a,b^-]\rightarrow 
    M[a,b]\otimes M[a^+,b^-]\rightarrow 0.
\end{equation}
\end{prop}

The next result is an application of Prop.~\ref{prop_T-system} that describes the relation between the affine determinant modules associated to a movable $i$-box and the $i$-box obtained from its box move.
We keep the notation of Prop.(\ref{prop_box_move})

\begin{cor}[{\cite[Prop.5.7]{KKOP_mon_cat_quant_aff_II}}]
Assume that the $i$-boxes $\mathfrak{c}_{s+1}$ and $\mathfrak{c}_s$ have the same index $\imath$.
\begin{itemize}
    \item[(i)] If $T_s=L$, then there is a short exact sequence
    \[0\rightarrow \bigotimes_{\imath \sim \jmath} M([a_{s+1}(\jmath)^+,b_{s+1}(\jmath)^-]) \rightarrow M((\nu_s(\mathfrak{C}))_s)\otimes M(\mathfrak{c}_s) \rightarrow M(\mathfrak{c}_{s+1})\otimes M([a_{s+1}^+,b_{s+1}^-])\rightarrow 0. 
    \]
    \item[(ii)] If $T_s=R$, then there is a short exact sequence
    \[0\rightarrow \bigotimes_{\imath \sim \jmath} M([a_{s+1}(\jmath)^+,b_{s+1}(j)^-]) \rightarrow M(\mathfrak{c}_s)\otimes M((\nu_s(\mathfrak{C}))_s)  \rightarrow M(\mathfrak{c}_{s+1})\otimes M([a_{s+1}^+,b_{s+1}^-])\rightarrow 0. 
    \]
\end{itemize}
\end{cor}

\section{Combinatorics of $Q$-data}\label{section_Q_data}
In this section, following \cite{Fuj_Oh_2021}, we recall the definition and the properties of $Q$-data.

Let $\mathfrak{g}$ be a finite-dimensional complex Lie algebra. 
Using  Table \ref{ta_unfold}, we associate to $\mathfrak{g}$ a pair $(\Delta, \sigma)$, 
where $\Delta$ is a simply-laced Dynkin diagram and $\sigma$ is a graph automorphism of $\Delta$. In the table, the symbol $\sigma= \text{id}$
denotes the identity map, while $\sigma = \vee $ and $\Tilde{\vee}$  are given by the blue arrows in Figure \ref{Fig_Q_data}. 
Let $\check{\mathfrak{g}}$ be the simply-laced Lie algebra associated to $\Delta$, and let $I_{\check{\mathfrak{g}}}$ be 
the set of vertices of  $\Delta$. We use the symbols $\imath, \jmath$ for the elements of  $I_{\check{\mathfrak{g}}}$ 
and the symbols $i,j$ for the elements of $I_\mathfrak{g}$.  In the following, we denote by $w_0$ the longest element of the 
Weyl group $W_{\check{\mathfrak{g}}}$ of $\check{\mathfrak{g}}$, and by $l_0$ its length.

\begin{table}
\begin{tabular}{||c c c||} 
 \hline
 $\mathfrak{g}$ & $\Delta$ & $\sigma$ \\ 
 \hline\hline
 $A_n$ & $A_n$ & id \\ 
 \hline
 $D_n$ & $D_n$ & id \\
 \hline
 $E_n$ & $E_n$ & id  \\
 \hline
 $B_n$ & $A_{2n-1}$ & $\vee$  \\
 \hline
 $C_n$ & $D_{n+1}$ & $\vee$  \\
 \hline 
 $F_4$ & $E_6$ & $\vee$ \\
 \hline  
 $G_2$ & $D_4$ & $\Tilde{\vee}$ \\
 \hline
\end{tabular}
\vspace*{0.3cm}
\caption{unfoldings}
\label{ta_unfold}
\end{table}

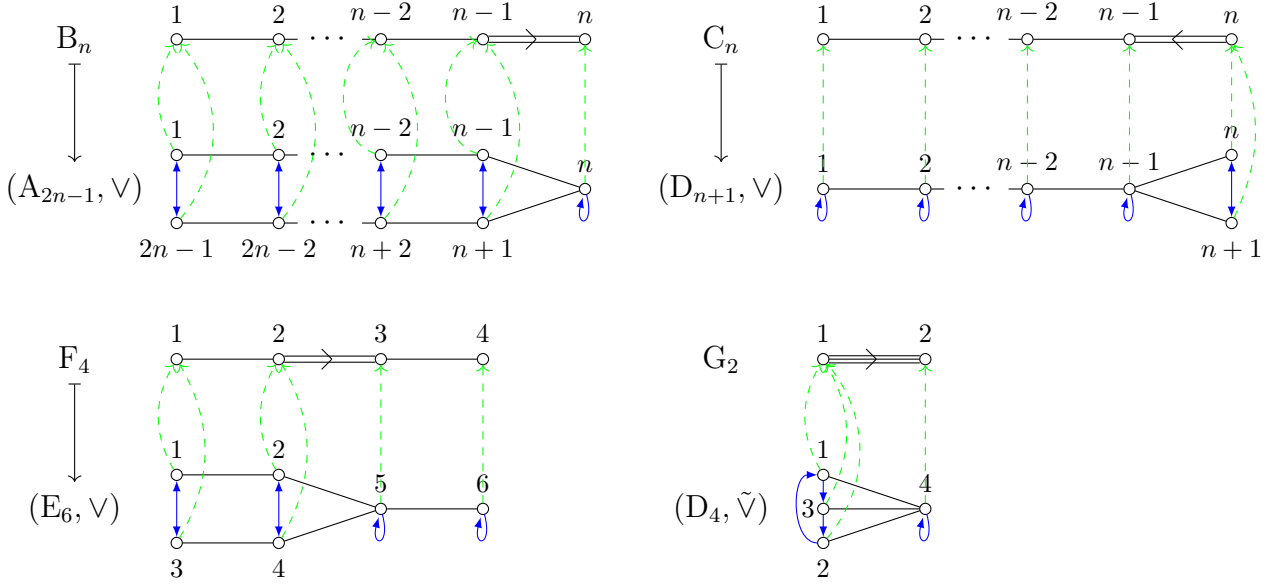
\begin{figure}
    \centering
\begin{tikzpicture}[xscale=0.9,yscale=0.9]
    \node (Bnlabel) at (-0.5,0.7) {$\text{B}_n$};
    \node[dynkdot, label={above: \footnotesize$1$}] (B1) at (1,0.7) {};
    \node[dynkdot, label={above: \footnotesize$2$}] (B2) at (2.5,0.7) {};
    \node (Bdots) at (3.25,0.7) {$\cdots$};
    \node[dynkdot, label={above: \footnotesize$n-2$}] (Bn-2) at (4,0.7) {};
    \node[dynkdot, label={above: \footnotesize$n-1$}] (Bn-1) at (5.5,0.7) {};
    \node[dynkdot, label={above: \footnotesize$n$}] (Bn) at (7,0.7) {};

    \node (A2n-1label) at (-0.5,-1.5) {$(\text{A}_{2n-1},\vee)$};
    \node[dynkdot, label={above: \footnotesize$1$}] (A1) at (1,-1) {};
    \node[dynkdot, label={above: \footnotesize$2$}] (A2) at (2.5,-1) {};
    \node (Adots1) at (3.25,-1) {$\cdots$};
    \node[dynkdot, label={above: \footnotesize$n-2$}] (An-2) at (4,-1) {};
    \node[dynkdot, label={above: \footnotesize$n-1$}] (An-1) at (5.5,-1) {};
    \node[dynkdot, label={above: \footnotesize$n$}] (An) at (7,-1.5) {};
    
    \node[dynkdot, label={below: \footnotesize$2n-1$}] (A2n-1) at (1,-2) {};
    \node[dynkdot, label={below: \footnotesize$2n-2$}] (A2n-2) at (2.5,-2) {};
    \node (Adots2) at (3.25,-2) {$\cdots$};
    \node[dynkdot, label={below: \footnotesize$n+2$}] (An+2) at (4,-2) {};
    \node[dynkdot, label={below: \footnotesize$n+1$}] (An+1) at (5.5,-2) {};

    \draw[|->] (Bnlabel) -- (A2n-1label);
    
    \draw (B1) -- (B2);
    \draw (B2) -- (Bdots);
    \draw (Bdots) -- (Bn-2);
    \draw (Bn-2) -- (Bn-1);


  \node (ghost1) at (5.42,0.745) {};
  \node (ghost2) at (7.08,0.745) {};
  \draw (ghost1) -- (ghost2);

  \node (ghost3) at (5.42,0.655) {};
  \node (ghost4) at (7.08,0.655) {};
  \draw (ghost3) -- (ghost4);

  \draw (6.13,0.85) -- (6.28,0.7);
  \draw (6.13,0.55) -- (6.28,0.7);

    \draw (A1) -- (A2);
    \draw (A2) -- (Adots1);
    \draw (An-2) -- (An-1);
    \draw (An-1) -- (An);
    \draw (An) -- (An+1);
    \draw (An+1) -- (An+2);
    \draw (An+2) -- (Adots2);
    \draw (Adots2) -- (A2n-2);
    \draw (A2n-2) -- (A2n-1);

    \draw[<->,>=latex, blue] (A1) -- (A2n-1);
    \draw[<->,>=latex, blue] (A2) -- (A2n-2);
    \draw[<->,>=latex, blue] (An-2) -- (An+2);
    \draw[<->,>=latex, blue] (An-1) -- (An+1);
    \path[->, >=latex, blue] (An) edge [loop below] (An);

    \draw[->, dashed, green] (A2n-1) to[bend right=40] (B1);
    \draw[->, dashed, green] (A1) to[bend left] (B1);
    \draw[->, dashed, green] (A2n-2) to[bend right=40] (B2);
    \draw[->, dashed, green] (A2) to[bend left] (B2);
    \draw[->, dashed, green] (An+2) to[bend right=40] (Bn-2);
    \draw[->, dashed, green] (An-2) to[bend left = 75] (Bn-2);
    \draw[->, dashed, green] (An+1) to[bend right=40] (Bn-1);
    \draw[->, dashed, green] (An-1) to[bend left = 75] (Bn-1);
    \draw[->, dashed, green] (An) to (Bn);


\node (Cnlabel) at (9,0.7) {$\text{C}_n$};
    \node[dynkdot, label={above: \footnotesize$1$}] (C1) at (10.5,0.7) {};
    \node[dynkdot, label={above: \footnotesize$2$}] (C2) at (12,0.7) {};
    \node (Cdots) at (12.75,0.7) {$\cdots$};
    \node[dynkdot, label={above: \footnotesize$n-2$}] (Cn-2) at (13.5,0.7) {};
    \node[dynkdot, label={above: \footnotesize$n-1$}] (Cn-1) at (15,0.7) {};
    \node[dynkdot, label={above: \footnotesize$n$}] (Cn) at (16.5,0.7) {};

    \draw (C1) -- (C2);
    \draw (C2) -- (Cdots);
    \draw (Cdots) -- (Cn-2);
    \draw (Cn-2) -- (Cn-1);

\node (ghostC1) at (14.92,0.745) {};
  \node (ghostC2) at (16.58,0.745) {};
  \draw (ghostC1) -- (ghostC2);

  \node (ghostC3) at (14.92,0.655) {};
  \node (ghostC4) at (16.58,0.655) {};
  \draw (ghostC3) -- (ghostC4);
 
  \draw (15.63,0.7) -- (15.78,0.83);
  \draw (15.63,0.7) -- (15.78,0.55);

\node (Dn+1label) at (9,-1.5) {$(\text{D}_{n+1}, \vee)$};
    \node[dynkdot, label={above: \footnotesize$1$}] (D1) at (10.5,-1.5) {};
    \node[dynkdot, label={above: \footnotesize$2$}] (D2) at (12,-1.5) {};
    \node (Ddots) at (12.75,-1.5) {$\cdots$};
    \node[dynkdot, label={above: \footnotesize$n-2$}] (Dn-2) at (13.5,-1.5) {};
    \node[dynkdot, label={above: \footnotesize$n-1$}] (Dn-1) at (15,-1.5) {};
    \node[dynkdot, label={above: \footnotesize$n$}] (Dn) at (16.5,-1) {};
    \node[dynkdot, label={below: \footnotesize$n+1$}] (Dn+1) at (16.5,-2) {};

    \draw (D1) -- (D2);
    \draw (D2) -- (Ddots);
    \draw (Ddots) -- (Dn-2);
    \draw (Dn-2) -- (Dn-1);
    \draw (Dn-1) -- (Dn);
    \draw (Dn-1) -- (Dn+1); 

    \draw[<->, >=latex, blue] (Dn) -- (Dn+1);

    \path[->, >=latex, blue] (D1) edge [loop below] (D1);
    \path[->, >=latex, blue] (D2) edge [loop below] (D2);
    \path[->, >=latex, blue] (Dn-2) edge [loop below] (Dn-2);
    \path[->, >=latex, blue] (Dn-1) edge [loop below] (Dn-1);

    \draw[|->] (Cnlabel) -- (Dn+1label);

    \draw[->, green, dashed] (D1) -- (C1);
    \draw[->, green, dashed] (D2) -- (C2);
    \draw[->, green, dashed] (Dn-2) -- (Cn-2);
    \draw[->, green, dashed] (Dn-1) -- (Cn-1);
    \draw[->, green, dashed] (Dn) -- (Cn);
    \draw[->, green, dashed] (Dn+1) to[bend right=25] (Cn);
    
\node (Flabel) at (-0.5,-4) {$\text{F}_4$};
    \node[dynkdot, label={above: \footnotesize$1$}] (F1) at (1,-4) {};
    \node[dynkdot, label={above: \footnotesize$2$}] (F2) at (2.5,-4) {};
    \node[dynkdot, label={above: \footnotesize$3$}] (F3) at (4,-4) {};
    \node[dynkdot, label={above: \footnotesize$4$}] (F4) at (5.5,-4) {};

   \node (ghostF1) at (2.42,-3.955) {};
   \node (ghostF2) at (4.08,-3.955) {};
  \draw (ghostF1) -- (ghostF2);

  \node (ghostF3) at (2.42,-4.045) {};
  \node (ghostF4) at (4.08,-4.045) {};
  \draw (ghostF3) -- (ghostF4);
  
  \draw (3.13,-3.85) -- (3.28,-4);
  \draw (3.13,-4.15) -- (3.28,-4);

  \draw (F1) -- (F2);
  \draw (F3) -- (F4);

  \node (Elabel) at (-0.5,-6.2) {$(\text{E}_6,\vee)$};
    \node[dynkdot, label={above: \footnotesize$1$}] (E1) at (1,-5.7) {};
    \node[dynkdot, label={above: \footnotesize$2$}] (E2) at (2.5,-5.7) {};
    \node[dynkdot, label={below: \footnotesize$3$}] (E3) at (1,-6.7) {};
    \node[dynkdot, label={below: \footnotesize$4$}] (E4) at (2.5,-6.7) {};
    \node[dynkdot, label={above: \footnotesize$5$}] (E5) at (4,-6.2) {};
    \node[dynkdot, label={above: \footnotesize$6$}] (E6) at (5.5,-6.2) {};
    
   \draw (E1) -- (E2);
  \draw (E3) -- (E4);
  \draw (E2) -- (E5);
  \draw (E4) -- (E5);
  \draw (E5) -- (E6);

 \path[->, >=latex, blue] (E5) edge [loop below] (E5);
 \path[->, >=latex, blue] (E6) edge [loop below] (E6);
 \draw[<->,>=latex, blue] (E1) -- (E3) ;
 \draw[<->,>=latex, blue] (E2) -- (E4) ;

 \draw[|->] (Flabel) -- (Elabel);

 \draw[->, green, dashed] (E1) to[bend left] (F1);
 \draw[->, green, dashed] (E3) to[bend right] (F1);
 \draw[->, green, dashed] (E2) to[bend left] (F2);
 \draw[->, green, dashed] (E4) to[bend right] (F2);
 \draw[->, green, dashed] (E5) to (F3);
 \draw[->, green, dashed] (E6) to (F4);


\node (Glabel) at (9,-4) {$\text{G}_2$};
    \node[dynkdot, label={above: \footnotesize$1$}] (G1) at (10.5,-4) {};
    \node[dynkdot, label={above: \footnotesize$2$}] (G2) at (12,-4) {};

    \node (ghostC1) at (10.42,-3.95) {};
  \node (ghostC2) at (12.08,-3.95) {};
  \draw (ghostC1) -- (ghostC2);

  \node (ghostC3) at (10.42,-4.055) {};
  \node (ghostC4) at (12.08,-4.055) {};
  \draw (ghostC3) -- (ghostC4);
 
  \draw (11.13,-3.85) -- (11.28,-4);
  \draw (11.13,-4.15) -- (11.28,-4);

  \draw (G1) -- (G2);
  
\node (Glabel) at (9,-6.2) {$(\text{D}_4,\Tilde{\vee})$};
  \node[dynkdot, label={above: \footnotesize$1$}] (DG1) at (10.5,-5.7) {};
  \node[dynkdot] (DG3) at (10.5,-6.2) {};
  \node (DG3label) at (10.28,-6.2) {\footnotesize$3$};
  \node[dynkdot, label={below: \footnotesize$2$}] (DG4) at (10.5,-6.7) {};
  \node[dynkdot, label={above: \footnotesize$4$}] (DG2) at (12,-6.2) {};

  \draw (DG1) -- (DG2);
  \draw (DG3) -- (DG2);
  \draw (DG4) -- (DG2);

  \draw[->,>=latex,blue] (DG1) -- (DG3);
  \draw[->,>=latex,blue] (DG3) -- (DG4);
  \draw[->,>=latex,blue] (DG4) edge[bend left = 90] (DG1);

  \path[->, >=latex, blue] (DG2) edge [loop below] (DG2);

  \draw[->,dashed, green] (DG1) to[bend left] (G1); 
  \draw[->,dashed, green] (DG3) to[bend right] (G1); 
  \draw[->,dashed, green] (DG4) to[bend right=40] (G1); 
  \draw[->,dashed, green] (DG2) to (G2); 
    
\end{tikzpicture}
\caption{Unfoldings for non-simply-laced $\mathfrak{g}$} 
\label{Fig_Q_data}
\end{figure}

Notice that the green arrows in Figure \ref{Fig_Q_data} provide a bijection between the orbits of the action of $\sigma$ 
on the set $I_{\check{\mathfrak{g}}}$ and the set $I_\mathfrak{g}$ introduced at the beginning of section~\ref{sect_module_categories}.
We denote by $\pi: I_{\check{\mathfrak{g}}} \rightarrow I_{\mathfrak{g}}$ the corresponding projection.

A {\em height function} on $(\Delta,\sigma)$ is a function $\varepsilon: I_{\check{\mathfrak{g}}}\rightarrow \mathbb{Z}$ satisfying the following conditions for any $\imath, \jmath \in I_{\check{\mathfrak{g}}}$ such that $\imath\sim \jmath$:
\begin{enumerate}
    \item[(i)] If $d_{\pi{(\imath)}}=d_{\pi{(\jmath)}}$, then $|\varepsilon_{\imath}-\varepsilon_{\jmath}|=d_{\pi{(\imath)}}$.
    \item[(ii)] If $1=d_{\pi{(\imath)}}<d_{\pi{(\jmath)}}=\text{ord}(\sigma)$, there is a unique $\Tilde{\jmath}$ in the $\sigma$-orbit of $\jmath$ 
    such that $|\varepsilon_{\imath}-\varepsilon_{\tilde{\jmath}}|=1$ and $\varepsilon_{\sigma^k(\Tilde{\jmath})}=\varepsilon_{\Tilde{\jmath}}+2k$ for any $0\leq k< \text{ord}(\sigma)$.
\end{enumerate}

If $\varepsilon$ is a height function for $(\Delta, \sigma)$, we call the triple
$\mathcal{Q}=(\Delta,\sigma,\varepsilon)$ a {\em $Q$-datum} for $\mathfrak{g}$.

\begin{de}[{\cite[Section 2]{Fuj_Oh_2021}}]
Let $\varepsilon$ be a height function for $(\Delta, \sigma)$ and let $\mathcal{Q}$ be the associated $Q$-datum.
\begin{itemize}
    \item A vertex $\imath\in I_{\check{\mathfrak{g}}}$ is a \textit{sink} of $\mathcal{Q}$ if $\varepsilon_\imath<\varepsilon_\jmath$ for any $\jmath\sim \imath$. It is a \textit{source} of $\mathcal{Q}$ if $\varepsilon_{\imath}-2d_{\pi(\imath)}>\varepsilon_\jmath-2d_{\pi(\jmath)}$ for any $\jmath\sim \imath$.
    \item If $\imath$ is a sink of $\mathcal{Q}$, we denote by $s_\imath\varepsilon$ the height function on $(\Delta,\sigma)$ defined by 
    \[(s_\imath\varepsilon)_\jmath=\varepsilon_\jmath+2d_{\pi(\imath)}\delta_{\imath\jmath},\]
    and by $s_\imath\mathcal{Q}$ the $Q$-datum $s_\imath\mathcal{Q}=(\Delta,\sigma,s_\imath\varepsilon)$ for $\mathfrak{g}$.
    \item A reduced expression $\underline{w}=s_{\imath_1}\dots s_{\imath_r}$ of an element of $W_{\check{\mathfrak{g}}}$ is $\mathcal{Q}$-\textit{adapted} if, for any $1\leq k\leq r$, the vertex $\imath_k$ is a sink of  $s_{\imath_{k-1}}s_{\imath_{k-2}}\dots s_{\imath_1}\mathcal{Q}$. When the context is clear, we also say that the reduced expression $\underline{w}$ is $\varepsilon$-\emph{adapted}, meaning that it is $\mathcal{Q}$-adapted.
\end{itemize}
\end{de}

\begin{de}[{\cite[Def. 6.9]{KKOP_mon_cat_quant_aff_II}}]
An {\em admissible sequence} for $(\Delta,\sigma)$ is a sequence of pairs $(\imath_k,p_k)$, $k\in\mathbb{Z}$, where
$\imath_k\in I_{\check{\mathfrak{g}}}$ and $p_k\in \mathbb{Z}$, that satisfies the following conditions:
\begin{enumerate}
    \item[(i)] $p_{s^+}=p_s+2d_{\pi(\imath_s)}$ for any $s\in \mathbb{Z}$;
    \item[(ii)] for $s,t\in\mathbb{Z}$, if $\imath_s\sim \imath_t$ and $t^-<s<t<s^+$, then $p_t=p_s+\text{min}(d_{\pi(\imath_s)},d_{\pi(\imath_t)})$;
    \item[(iii)] for any $k\in \mathbb{Z}$, we have $s_{\imath_k}\dots s_{\imath_{k+{l_0}-1}}=w_0$.
\end{enumerate}
Here, for $s\in\mathbb{Z}$, we denote by $s^+$ (resp. $s^-$) the smallest index $k>s$ such that $\imath_k=\imath_s$ (resp. the greatest index $k<s$ such that $\imath_k=\imath_s$).
\end{de}

The next proposition describes how to construct an admissible sequence from the datum of a height function and a reduced expression of $w_0$, and vice versa. Recall that $(\Delta,\sigma)$ is an unfolding of the finite-dimensional simple complex Lie algebra $\mathfrak{g}$.
\begin{prop}[{\cite[Prop. 6.11]{KKOP_mon_cat_quant_aff_II}}]
    There is a bijection between the set of pairs $(\varepsilon,\underline{w}_0)$, where $\varepsilon$ is a height function on $(\Delta,\sigma)$ and $\underline{w}_0$ is an $\varepsilon$-adapted reduced expression of $w_0$, and the set of admissible sequences on $(\Delta,\sigma)$. It is given as follows:
    \begin{itemize}
        \item If $\varepsilon$ is a height function on $(\Delta, \sigma)$ and $\underline{w}_0=s_{\jmath_1}\dots s_{\jmath_{l_0}}$ is an 
        $\varepsilon$-adapted expression, the associated admissible sequence $(\imath_k,p_k)$, $k\in \mathbb{Z}$, is defined by
    \[
\imath_k=(\widehat{\underline{w}}_0)_k \text{  and  } p_k= 
\begin{cases}
    (s_{\imath_{k-1}}\dots s_{\imath_1}\varepsilon)_{\imath_k}, & k\geq 1,\\
    (s_{\imath_{k-1}}^{-1}\dots s_{\imath_1}^{-1}\varepsilon)_{\imath_k}, & k\leq 0.
\end{cases}
    \]

    \item If $(\imath_k,p_k)$, $k \in \mathbb{Z}$, is an admissible sequence, the corresponding pair 
    $(\varepsilon,\underline{w}_0)$ is defined by
    \[ \underline{w}_0=s_{\imath_1}\dots s_{\imath_{l_0}} \text{  and  } \varepsilon_\imath=p_{a_\imath} \text{ for each } \imath\in I_{\check{\mathfrak{g}}},\]
    where $a_\imath=\text{min}\{k\in\mathbb{N}\backslash \{0\} | \imath_k=\imath\}$.
    \end{itemize}
\end{prop}

In the following, we only consider  height functions satisfying the extra assumption
\[ \varepsilon_\imath = \Tilde{\varepsilon}_{\Bar{\imath}}\ \ (\text{mod } 2), \ \ \forall \imath\in I_{\check{\mathfrak{g}}}, \]
where $\Tilde{\varepsilon}$ is the parity function introduced in \ref{subsec_categoryC0}.

As shown in \cite[Thm. 6.12]{KKOP_mon_cat_quant_aff_II}, each $Q$-datum $\mathcal{Q}$ of $\mathfrak{g}$ naturally provides  a strong duality datum $\mathcal{D}_\mathcal{Q}$ of $\mathfrak{g}$.
For PBW-pairs $(\mathcal{D}_\mathcal{Q},\underline{w}_0)$, where $\underline{w}_0$ is $\mathcal{Q}$-adapted, the following proposition gives an explicit characterization of the affine cuspidal modules.

\begin{prop}[{\cite[Thm. 6.13]{KKOP_mon_cat_quant_aff_II}}]
Let $\varepsilon$ be a height function on $(\Delta,\sigma)$ and $\mathcal{Q}=(\Delta,\sigma,\eps)$ the corresponding $Q$-datum.
Let $\underline{w}_0$ be a $\mathcal{Q}$-adapted reduced expression for $w_0$.
Let $(\imath_k,p_k)$,  $k\in\mathbb{Z}$, be the corresponding admissible sequence.
Then, for any $k\in\mathbb{Z}$, we have an isomorphism
\[ S^{\mathcal{D}_\mathcal{Q},\widehat{\underline{w}}_0}_k \cong L(Y_{\pi(\imath_k),p_k}). \]
Therefore, for any $i$-box $[c,d]_\imath$ of $i$-cardinality $m$, we have
\[ M^{\mathcal{D}_\mathcal{Q},\widehat{\underline{w}}_0}[c,d] \cong W^{\pi(\imath)}_{m,p_c}\, , \]
and, for any (possibly infinite) interval $[a,b]$, the category  $\mathscr{C}_{\mathfrak{g}}^{[a,b],\mathcal{D}_\mathcal{Q},\underline{w}_0}$ is the smallest monoidal Serre subcategory of $\mathscr{C}_\mathfrak{g}^{0}$ containing the fundamental modules $L(Y_{\pi(\imath_k),p_k})$, for all $k$ in $[a,b]$.
\end{prop}

\section{Monoidal Categorification}
\label{sec_mo_cat}
In this section, following \cite{KKOP_mon_cat_quant_aff_II}, we recall the notion of monoidal categorification of cluster algebras 
and that each Grothendieck ring $\mathcal{K}(\mathscr{C}_{\mathfrak{g}}^{\mathcal{D}_\mathcal{Q},\underline{w}_0,[a,b]})$ is an occurance of such phenomenon. In particular, we recall the seed of the cluster algebra structure on the Grothendieck ring $\mathcal{K}(\mathscr{C}^-_{\mathfrak{g}})$
 introduced by Hernandez and Leclerc in \cite{HL_Clust_alg_approach_q_char} and explain a construction which allows to 
 explicitly describe the quivers of a large family of other seeds. 

\subsection{Reminder on monoidal categorification}
Let $\mathcal{C}$ be a monoidal Serre subcategory of $\mathscr{C}_\mathfrak{g}$.
\begin{de}
 A \textit{monoidal seed} in $\mathcal{C}$ is a triple 
\[ \mathcal{S}=(\{M_i\}_{i\in K}, B, K),\]
where
\begin{itemize}
    \item $K$ is a countable index set with a decomposition $K=K^{\text{ex}}\coprod K^{\text{fr}}$. We refer to $K^{\text{ex}}$ as the subset of {\em exchangeable} indices, and to $K^{\text{fr}}$ as the subset of {\em frozen} indices; 
    \item $B=(b_{ij})_{K\times K}$ is an \emph{exchange matrix}, that is, a skew-symmetric matrix with the property that, for each $j\in K^\text{ex}$, there exist only finitely many $i\in K$ such that $b_{ij}\neq 0$;
    \item $\{M_i\}_{i\in K}$ is a strongly tensor-commuting family of real simple objects in $\mathcal{C}$.
\end{itemize}
\end{de}

\begin{rem}
    The above definition of exchange matrix differs from the one given in \cite{KKOP_mon_cat_quant_aff_II}, where the format of the matrix is $K\times K^{\text{ex}}$. This change is motivated by the setting of additive categorification of the next section, where, in order to apply Palu's generalized mutation rule to compute the matrices $B(\mathfrak{C})$ (see Thm. \ref{teo_main_KKOP}), we need square matrices. However, if we restrict the matrix $B(\mathfrak{C})$ to the subset $K\times K^{\text{ex}}$, we obtain the corresponding matrix of \cite{KKOP_mon_cat_quant_aff_II}.
\end{rem}

\begin{de}
\label{de_admissible_monoidal_seed}
    A monoidal seed $\mathcal{S}$ in $\mathcal{C}$ is \textit{admissible} if the following two conditions are satisfied:
    \begin{itemize}
        \item[(i)] For any $k\in K^\text{ex}$, there exists a real simple 
        object $M_k'\in \mathcal{C}$ occurring in a short exact sequence
        \[0\rightarrow \bigotimes_{b_{ik}>0} M_i^{b_{ik}} \rightarrow M_k\otimes M_k'\rightarrow \bigotimes_{b_{ik}<0} M_i^{-b_{ik}}\rightarrow 0.\]
        \item[(ii)] For any $k\in K^{\text{ex}}$, the object $M_k'$ strongly commutes with $M_i$, for any $i\in K\backslash \{k\}$.
    \end{itemize}
    If $\mathcal{S}$ is admissible, we define its \emph{mutation} in direction $k$ as the monoidal seed
    \[\mu_k(\mathcal{S})=(\{M_i\}_{i\neq k}\cup \{M_k'\},\mu_k(B), K).\]
    We say that $\mathcal{S}$ is \textit{completely admissible} if it admits successive mutations in all possible directions.
\end{de}

\begin{de}[{\cite[Def. 2.1]{Her_Lec_clust_alg_quat_aff}}]  \label{def_monoidal_cat}
A \textit{monoidal categorification in $\mathscr{C}_\mathfrak{g}$} of a cluster algebra $\mathcal{A}$ is given by 
\begin{itemize}
    \item[(a)] a monoidal Serre subcategory $\mathcal{C}$ of $\mathscr{C}_\mathfrak{g}$ and
    \item[(b)] a ring isomorphism $\varphi: \mathcal{K}(\mathcal{C}) \xrightarrow{\sim} \mathcal{A}$ such that there is a completely admissible monoidal seed $\mathcal{S}=(\{M_i\}_{i\in K}, B, K)$ whose image $(B, (\varphi([M_i]))_{i\in K})$ under $\varphi$
    is a seed of $\mathcal{A}$ and, for $i\in K^{\mathrm{fr}}$, the elements $\varphi([M_i])$ are the frozen cluster variables.
\end{itemize}
\end{de}

We recall that, using the theory of $R$-matrices, to each pair of simple modules $M$ and $N$ in $\mathscr{C}_\mathfrak{g}$,
Kashiwara–Kim–Oh–Park have associated  an integer denoted by $\Lambda(M,N)$ 
(see \cite[Def. 3.6]{KKOP_mon_cat_quant_aff} for the precise definition). They have shown in Corollary~3.17 of
[loc.~cit.] that if $M$ and $N$ are real modules which strongly commute, then we have $\Lambda(M,N)=-\Lambda(N,M)$.
Relying on these invariants, they have introduced a stronger notion of categorification as follows.

\begin{de}[{\cite[Def. 6.5]{KKOP_mon_cat_quant_aff}}] 
An admissible monoidal seed $\mathcal{S}=(\{M_i\}_{i\in K}, B, K)$ is 
$\Lambda$-\emph{admissible} if, for any $k\in K^{\text{ex}}$, the module $M_k'$ of Def. \ref{de_admissible_monoidal_seed} satisfies
\[(\Lambda(M_k,M_k')+\Lambda(M_k',M_k))/2=1.\] It is \emph{completely $\Lambda$-admissible} if it is completely admissible and
any monoidal seed obtained from it through iterated mutations is $\Lambda$-admissible.
\end{de}

For a monoidal seed $\mathcal{S}=(\{M_i\}_{i\in K}, B, K)$, 
we denote by $\Lambda^\mathcal{S}=(\Lambda^\mathcal{S}_{ij})_{i,j\in K}$ the skew-symmetric matrix defined by
\[ \lambda^\mathcal{S}_{ij} = \Lambda(M_i,M_j).\]

\begin{de}[{\cite[Def. 6.7]{KKOP_mon_cat_quant_aff}}]
A {\em $\Lambda$-monoidal categorification in $\mathscr{C}_\mathfrak{g}$} of a cluster algebra $\mathcal{A}$ is 
a monoidal categorification of $\mathcal{A}$ in $\mathscr{C}_\mathfrak{g}$ as in definition~\ref{def_monoidal_cat}
 where $\mathcal{S}=(\{M_i\}_{i\in K}, B, K)$ is a completely $\Lambda$-admissible monoidal seed and we have
 \[
\sum_{k\in K} \lambda^\mathcal{S}_{ik} b_{kj} = -2\delta_{i,j}
\]
for all $i\in K$ and $j\in K^{\mathrm{ex}}$.    
\end{de}

\subsection{Admissible seeds associated to chains of $i$-boxes}
We fix a simple finite-dimensional complex Lie algebra $\mathfrak{g}$, a $Q$-datum $\mathcal{Q}$ of $\mathfrak{g}$ 
and a reduced expression $\underline{w}_0$ for the longest element of the Weyl group of the unfolded
Lie algebra $\check{\mathfrak{g}}$, see section~\ref{section_Q_data}. For the rest of this section, 
we work with $i$-boxes defined with respect to the sequence $\widehat{\underline{w}}_0=(\imath_k)_{k\in\mathbb{Z}}$, 
and with affine determinant modules defined with respect to the PBW-pair $(\mathcal{D}_{\mathcal{Q}},\underline{w}_0)$.
For a chain of $i$-boxes $\mathfrak{C}=(\mathfrak{c}_k)$ of range $[a,b]$ and length $l$, we use the following
notations:
\begin{align*}
K(\mathfrak{C})&=\begin{cases}
    \{1,\dots, l\}, &\text{ if } l<\infty;\\
    \mathbb{N}_{\geq 1}, &\text{ if } l=\infty;
\end{cases}\\
K(\mathfrak{C})^{\text{fr}}&=\{s \in K(\mathfrak{C})\ |\ \mathfrak{c}_s=[a(\imath)^+,b(\imath)^-] \text{ for some } \imath \in I_{\check{\mathfrak{g}}} \};\\
K(\mathfrak{C})^{\text{ex}}&=K(\mathfrak{C})\backslash K(\mathfrak{C})^{\text{fr}}.
\end{align*}

Note that $K(\mathfrak{C})^{\text{fr}}$ is the set of indices of those $i$-boxes which are maximal with respect to the inclusion of $i$-boxes with the same index. In particular, if $l$ is finite, then $K(\mathfrak{C})^{\text{fr}}$ has the same number of elements
as $I_{\check{\mathfrak{g}}}$ and it is empty otherwise.

Recall that, by Prop. \ref{prop_properties_aff_det_mod}, for any chain $\mathfrak{C}$ of $i$-boxes the modules in
$M(\mathfrak{C})$ form a strongly commuting family of real simple modules. The next theorem is the fundamental result of \cite{KKOP_mon_cat_quant_aff_II}.

\begin{teo}[{\cite[Thm.~8.1]{KKOP_mon_cat_quant_aff_II}}]
\label{teo_main_KKOP}
Let $[a,b]$ be a possibly infinite interval. For  each chain of $i$-boxes $\mathfrak{C}$ of range $[a,b]$, there is an exchange matrix $B(\mathfrak{C})$ such that the monoidal seed 
\begin{equation}
\label{S(C)}
\mathcal{S}(\mathfrak{C})=(M(\mathfrak{C}), B(\mathfrak{C}), K(\mathfrak{C}) ) 
\end{equation}
in $\mathcal{C}=\mathscr{C}_\mathfrak{g}^{[a,b],\mathcal{D}_\mathcal{Q},\underline{w}_0}$
is completely $\Lambda$-admissible and defines a $\Lambda$-monoidal categorification of the
cluster algebra $\mathcal{A}$ associated with $B(\mathfrak{C})$.
\end{teo}

At the end of \cite{KKOP_mon_cat_quant_aff_II}, Kashiwara--Kim--Oh--Park state the following problem for any chain of $i$-boxes $\mathfrak{C}$

\begin{equation*} \text{(KKOP Problem) determine explicitly the exchange matrix } B(\mathfrak{C}).
\end{equation*}

\begin{rem}
    Notice that the statement of the problem in \cite[\S~8]{KKOP_mon_cat_quant_aff_II} is formulated for general PBW-pairs, while in this work we are providing a solution for PBW-pairs associated to $\mathcal{Q}$-data. However, in the general setting, Theorem \ref{teo_main_KKOP} is a conjecture \cite[Conj.~8.13]{KKOP_mon_cat_quant_aff_II}. If the conjecture is proven to be true, our method equally applies in the general case.
\end{rem}

For any chain of $i$-boxes $\mathfrak{C}$, the exchange matrix $B(\mathfrak{C})$ restricted to the subset $K(\mathfrak{C})\times K^{\text{ex}}(\mathfrak{C})$ is unique, since it is determined by the cluster 
$([M(\mathfrak{c}_k)])_k$, cf.~\cite{CerulliKellerLabardiniPlamondon13}. We denote by $Q(\mathfrak{C})$ the associated quiver. 
In the following, we will occasionally identify its set of vertices with the set of $i$-boxes $\{\mathfrak{c}_k\}$ (they are both indexed by $K(\mathfrak{C})$). Under this identification, the frozen vertices correspond to the $i$-boxes which are maximal with respect to the inclusion of $i$-boxes with the same index.

For a specific type of chain of 
$i$-boxes, the corresponding exchange matrix is already known, and this is the starting point
for the proof of the above Theorem~\ref{teo_main_KKOP} in \cite{KKOP_mon_cat_quant_aff_II}. 

Let $Q(\underline{w}_0)$ be the quiver with vertex set $\mathbb{Z}$, and with arrows
\[
\{s\rightarrow s^- | s\in\mathbb{Z}\} \cup \{s\rightarrow t | s^-<t^-<s<t,\ \imath_s\sim\imath_t \},
 \]
 where $s^-$ and $t^-$ are computed using the sequence $\widehat{\underline{w}}_0$.

We choose to adopt the following conventions to graphically represent the quiver $Q(\underline{w}_0)$: we arrange the vertices in lines indexed by the set $I_{\check{\mathfrak{g}}}$, such that the vertex $a$ is placed in the line $i_a$ and, if $a<b$, then $a$ is placed to the left of $b$.
Moreover, by the periodicity of the sequence $\widehat{\underline{w}}_0$, in order to know $Q(\underline{w}_0)$ it suffices to draw the full subquiver $Q^{[-l_0+1,l_0]}(\underline{w}_0)$.

\begin{ex}
If $\mathfrak{g}$ is of type $A_3$ and we fix the reduced expression $\underline{w}_0=s_1 s_2s_3s_1s_2s_1$, then the quiver $Q(\underline{w}_0)$ can be represented by the following diagram:
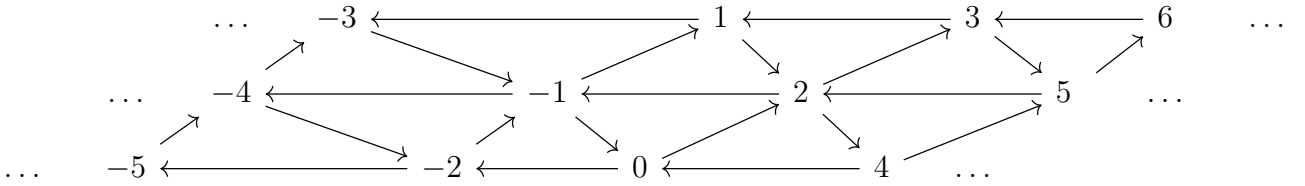
\begin{equation*}
    \begin{tikzcd}[sep=small]
      &  & \dots & -3 \arrow[drr] & & & & 1  \arrow[dr] \arrow[llll] & & & 3  \arrow[lll] \arrow[dr] &  & 6 \arrow[ll]& \dots \\
        &\dots&  -4 \arrow[drr] \arrow[ur]  & & & -1 \arrow[urr] \arrow[dr]  \arrow[lll]& 
        & & 2 \arrow[lll] \arrow[urr] \arrow[dr] &  & &5 \arrow[ur]\  \arrow[lll] & \dots\\
       \dots&  -5  \arrow[ur]& & & -2  \arrow[ur]\arrow[lll] & & 0  \arrow[ll] \arrow[urr]& & & 4 \arrow[urr] \arrow[lll]& \dots &  &
    \end{tikzcd}
\end{equation*}
\end{ex}

For $a\in \mathbb{Z}\cup \{-\infty\}$ and  $b\in\mathbb{Z},\ b\geq a$,  we denote by $Q^{[a,b]}(\underline{w}_0)$ the ice 
quiver (i.e. quiver with a distinguished subset of `frozen' vertices) obtained from the full subquiver of $Q(\underline{w}_0)$ 
on the vertex set $[a,b]$ by freezing the vertices $s\in [a,b]$ such that 
\[
s=\text{min}\{t\in [a,b] \ |\ \imath_t=\imath  \text{ for some } \imath \in I_{\check{\mathfrak{g}}}. \} 
\] 
We denote by $B(\underline{w}_0)$ the associated matrix. 
Notice that there is a bijection 
\[
Q^{[a,b]}(\underline{w}_0)_0 \xrightarrow{\sim} K(\mathfrak{C}_{-}^{[a,b]}),\ s\mapsto 1 + (b-s).  
\]
It sends the set of frozen vertices to the set $K^{\text{fr}}(\mathfrak{C}_{-}^{[a,b]})$.

When $\underline{w}_0$ is $\mathcal{Q}$-adapted, the following result is basically \cite[Thm. 5.1]{HL_Clust_alg_approach_q_char}. The general case is included in the proof of \cite[Prop. 8.11]{KKOP_mon_cat_quant_aff_II}
\begin{prop} 
Let $[a,b]$ be an interval, where $a\in \mathbb{Z}\cup \{-\infty\}$ and $b\in\mathbb{Z}$.
For $\mathfrak{C}=\mathfrak{C}_{-}^{[a,b]}$ the matrix $B(\mathfrak{C})$ of Theorem~\ref{teo_main_KKOP} 
equals $B(\underline{w}_0)^{[a,b]}$.
\end{prop}

The next result describes the effect of a box move in the associated monoidal seeds.

\begin{lem}[\cite{KKOP_mon_cat_quant_aff_II}]
\label{lem_box_move_mutation}
    Let $\mathfrak{C}=(\mathfrak{c}_k)_{1\leq k\leq l}$ be a chain of $i$-boxes of finite range. Let $s\in K^{\mathrm{ex}}$ such that $\mathfrak{c}_{s}$ is movable.
    
    \begin{itemize}
        \item[(i)] If the $i$-box $\mathfrak{c}_{s+1}=[a_{s+1},b_{s+1}]$ has the same index as $\mathfrak{c}_s$, then
        \[\mathcal{S}(\nu_s(\mathfrak{C}))= (\{M(\mathfrak{c}_k)\}_{k\in K(\mathfrak{C})\backslash \{s\}}\sqcup \{M(\mathfrak{c_s'})\}, \mu_s(B(\mathfrak{C})), K(\mathfrak{C})),\]
        where  \[ M(\mathfrak{c}_s')=\begin{cases}
     [a_{s+1},b_{s+1}^-] & \text{if } \mathfrak{c}_s=[a_{s+1}^+,b_{s+1}]\\
     [a_{s+1}^+,b_{s+1}] & \text{if } \mathfrak{c}_s=[a_{s+1},b_{s+1}^-].
    \end{cases}\]

        \item[(ii)] If the $i$-boxes $\mathfrak{c}_{s+1}$ and $\mathfrak{c}_s$ have different indices, then 
    \[\mathcal{S}(\nu_s(\mathfrak{C}))= (\{M(\mathfrak{c'}_k)\}_{k\in K(\mathfrak{C})}, B'(\mathfrak{C}), K(\mathfrak{C})),\]
    where
    \[ M(\mathfrak{c}_k') = M(\mathfrak{c}_{\sigma_s(k)}), \ \ \ B'(\mathfrak{C})_{ij}=B(\mathfrak{C})_{\sigma_s(i),\sigma_s(j)} \]
    and $\sigma_s$ denotes transposition of $s$ and $s+1$.
    
    \end{itemize}
\end{lem}

\begin{proof}
Case $(i)$ is \cite[Lem. 7.17]{KKOP_mon_cat_quant_aff_II}. In case $(ii)$, notice that the chain $\nu_s(\mathfrak{C})$ differs from the chain $\mathfrak{C}$ only by the transposition of the indices of the $i$-boxes $\mathfrak{c}_s$ and $\mathfrak{c}_{s+1}$. Therefore, the monoidal seed $\mathcal{S}(\nu_s(\mathfrak{C}))$ can be obtained from $\mathcal{S}(\mathfrak{C})$ by the corresponding permutation of its indices.
\end{proof}

We now discuss an algorithmic procedure to obtain the exchange matrix for an arbitrary chain of $i$-boxes with finite length.
In the following, we identify chains of $i$-boxes with the associated rooted sequences of exchange
operators as in Remark~\ref{rem_exp_op}.

\begin{lem}[Order of mutations]
\label{lem_ord_mut}
Let $[a,b]$ be a finite interval of length $l\geq 2$. Let $\mathfrak{C}$ be the chain of $i$-boxes of range $[a,b]$ with associated pair 
\[([b-1];(\underbrace{L,\dots,L}_{(l-2)\text{-times}},R)).\] 
Then the quiver $Q(\mathfrak{C})$ is obtained from the quiver $Q^{[a,b]}(\underline{w}_0)$ by mutating, in order, 
at the vertices associated to the $i$-boxes $\mathfrak{c}^-_{i_1},\dots,\mathfrak{c}^-_{i_m}$, where $\mathfrak{c}^-_{i_j}$ is the 
$i$-box in $\mathfrak{C}_{-}^{[a,b]} $ with index $\imath_b$ and $i$-cardinality $j$ and $(m+1)$ is the 
$i$-cardinality of the maximal $i$-box of index $\imath_b$.
\end{lem}

\begin{proof}
Consider the basic chain transformation
\begin{equation}
\label{eq_basic_chain_in_lemma_order_mutations}
    \mathfrak{C}_{-}^{[a,b]} \mapsto ([b-1];(\underbrace{L,\dots,L}_{(l-2)\text{-times}},R)).
\end{equation}  
It involves $l-1$ box moves. For $0 \leq k\leq l-1$, denote by $\mathfrak{C}_k=(\mathfrak{c}_t^k)_t$ the chain obtained after the first $k$ moves. For $k\geq 1$, the right expansion operator is in the $k$th position. 
Moreover, for any $0 \leq k\leq l-1$, the $i$-box $\mathfrak{c}_{k+1}^{k}$ has $b$ as its right border and is the only $i$-box among 
the $\mathfrak{c}_s^{k}$, $1\leq s\leq k+1$, containing $b$. Indeed, this follows from the facts that 
\begin{itemize}
\item[1)] $b$ is the right border of the range of the chain of $i$-boxes, and 
\item[2)] for $k\geq 1$, the $i$-box $\mathfrak{c}^k_{k+1}$ is obtained from the unique right expansion operator in the sequence.
\end{itemize}
Therefore, it follows from Lemma \ref{lem_box_move_mutation} that the $k$th box move entails a quiver mutation at the vertex associated to the $i$-box $\mathfrak{c}_{k}^{k-1}$ 
if and only if also $\mathfrak{c}_{k+1}^{k-1}=\mathfrak{c}_{k+1}^{-}$ is of 
index $\imath_b$. In particular, the basic chain transformations (\ref{eq_basic_chain_in_lemma_order_mutations}) entails exactly $m$ mutations and these occur at vertices associated to $i$-boxes of index $\imath_b$. 

Notice that, for any $0\leq k\leq l-1$, there exists an integer $1\leq j\leq m$ such that $\mathfrak{c}_{k+1}^k=\mathfrak{c}_{i_j}^-$. Moreover, the $i$-cardinality of $\mathfrak{c}_{1}^0=\{b\}$ is equal to 1 and, by Lemma \ref{lem_box_move_mutation}, for any $1\leq k\leq l-1$, the $i$-cardinality of $\mathfrak{c}_{k+1}^{k}$ increases by 1 with respect to the $i$-cardinality of $\mathfrak{c}_{k}^{k-1}$ if and only if the $k$th box move entails a quiver mutation (otherwise, these two $i$-boxes are equal). Therefore, for any $1\leq j\leq m$, the $j$th mutation occurs at the vertex associated to the $i$-box $\mathfrak{c}_{i_j}^-$.
\end{proof}

\begin{con}
\label{con_sequence_mutations_Q(C)}
 Let $[a,b]$ be a finite interval of length $l$. Let $\mathfrak{C}$ be a chain of $i$-boxes of range $[a,b]$ associated with a 
 pair $([c],(E_k)_{1\leq k\leq l-1})$. We recursively construct the quiver $Q(\mathfrak{C})$
starting from the quiver $Q^{[a,b]}(\underline{w}_0)$ as follows:  Let $i_1 < i_2 < \ldots$ be the elements of the set
 \[
 \mathcal{R}=\{ k\in [1,l-1]\ |\ E_k=R\}.
 \]
 We set $\mathfrak{C}_0=\mathfrak{C}_{-}^{[a,b]}$ and, for any $1\leq s\leq |\mathcal{R}|$, as in Remark \ref{rem_can_chain_trans},
 we define   the chain of $i$-boxes $\mathfrak{C}_s$ of range $[a,b]$ associated to the sequence of expansion operators
\[ E_k=\begin{cases}
     R,\text{    if } k\in \mathcal{R}, k\geq i_{|\mathcal{R}|-s+1}\\
     L, \text{    otherwise.}
 \end{cases}
 \]
Notice that $\mathfrak{C}_{|\mathcal{R}|}=\mathfrak{C}$ and that, for any $1\leq s\leq |\mathcal{R}|$, the chain of $i$-boxes 
 $\mathfrak{C}_s$ associated with 
 \[
 ([b-s]; \underbrace{L,\dots,L}_{(i_{|\mathcal{R}|-s+1}-1)}, R, E_{i_{|\mathcal{R}|-s+1}+1},\dots, E_{l-1} )
 \]
 can be obtained from the chain $\mathfrak{C}_{(s-1)}$ associated with
 \[
 ([b-s+1]; \underbrace{L,\dots,L}_{(i_{|\mathcal{R}|-s+1}-1)}, L, E_{i_{|\mathcal{R}|-s+1}+1},\dots, E_{l-1} )
 \] 
through the basic chain transformation defined in the proof of Lemma~\ref{lem_ord_mut} (the expansion operators $E_k$ with $k\geq i_{|\mathcal{R}|-s+1}+1$ play no role). Thus, using the mutation sequence associated with the chain transformation, we obtain
$Q(\mathfrak{C}_s)$ from $Q(\mathfrak{C}_{s-1})$. The quiver $Q(\mathfrak{C})$ is reached when $s=|\mathcal{R}|$.
\end{con}

\begin{lem}[Number of mutations]
\label{lem_numb_mut}
 Let $[a,b]$ be a finite interval.
Let $\mathfrak{C}$ be a chain of \mbox{$i$-boxes} of range $[a,b]$. Let $[c,d]$ be an $i$-box of $i$-cardinality $k$ 
and index $\imath$ in $\mathfrak{C}$ and $[e,b(\imath)^-]$ the corresponding $i$-box of $i$-cardinality $k$ and index $\imath$
in $\mathfrak{C}^{[a,b]}_-$. Denote by $v$ the corresponding vertex of 
$Q(\mathfrak{C}^{[a,b]}_-)=Q^{[a,b]}(\widehat{\underline{w}}_0)$. Then the number of
mutations at $v$ in the mutation sequence associated with the canonical chain transformation
\[  
\mathfrak{C}_{-}^{[a,b]} \mapsto \mathfrak{C}
\]
equals the integer $m$ defined by
\[ 
[c,d]=\tau^{-m} [e,b(\imath)^-]. 
\]
\end{lem}

\begin{proof}
The canonical chain transformation from $\mathfrak{C}^{[a,b]}_-$ to $\mathfrak{C}$ does not involve flush
right shifts of $i$-boxes. Thus, the box moves associated to mutations are exactly the left flush shifts
moving an $i$-box to the next $i$-box to the left with the same index and $i$-cardinality. Therefore, the number of mutations at the 
vertex $v$ is equal to the number of times that the corresponding starting $i$-box  $[e,b(\imath)^-]$
has to be shifted to obtain $[c,d]$.
\end{proof}

\begin{de}
Let $\mathfrak{C}$ be a chain of $i$-boxes and let $([c],(E_k)_{1\leq k\leq l-1})$ be the associated rooted sequence of expansion operators. For any $1\leq k\leq l-1$, the \emph{index} $i(E_k)$ of the expansion operator $E_k$ is the index of the $i$-box $\mathfrak{c}_{k+1}$.
\end{de}

\begin{rem}
Let $\mathfrak{C}$ and $\mathfrak{C}'$ be two chain of $i$-boxes related by a box move at $s\geq 2$. If we denote the associated sequences of expansion operators by $(E_k)_k$ and $(E'_k)_k$, then we have
\[
i(E_s)=i(E'_{s-1})\quad \text{and}\quad i(E_{s-1})=i(E_{s}'). 
\]\
In particular, each basic chain transformation is associated with a right expansion operator of a fixed index.
\end{rem}

In the next couple of lemmas, we retain the notation introduced in Construction \ref{con_sequence_mutations_Q(C)}.

\begin{lem}
\label{lem_condition_box_move_basic_chain_trans}
For $1\leq t\leq |\mathcal{R}|$, let $\mathfrak{c}$ be an $i$-box of $\mathfrak{C}_{t-1}$ of index $i(E_{i_{|\mathcal{R}|-t+1}})$ and let $k$ be its $i$-cardinality. Let $m(t)$ be the integer
\[m(t)=|\{ j\in [i_{|\mathcal{R}|-t+1}+1,l-1]\ |\ E_j=R, i(E_j)=\imath\}|.\]
Then the basic chain transformation $\mathfrak{C}_{t-1}\mapsto\mathfrak{C}_{t}$ sends $\mathfrak{c}$ to $\tau^-\mathfrak{c}$ if and only if 
there is an integer $1\leq i(m(t)+1+k,\imath)\leq l$ such that 
\begin{itemize}
    \item the $i$-box $\mathfrak{c}^{-}_{i(m(t)+1+k,\imath)}$ of $\mathfrak{C}_{-}^{[a,b]}$ has index $\imath$ and $i$-cardinality $m(t)+1+k$;
    \item $i_{|\mathcal{R}|-t+1}+t\geq i(m(t)+1+k,\imath)$.
\end{itemize}
\end{lem}

\begin{proof}
The basic chain transformation associated to the expansion operator $E_{i_{|\mathcal{R}|-t+1}}=R$ has the form
 \[
 ([b-t+1]_\imath; \underbrace{L,\dots,L}_{i_{|\mathcal{R}|-t+1}-1}, L, E_{i_{\mathcal{R}-t+1}+1},\dots, E_{l-1} )\mapsto 
 ([b-t]; \underbrace{L,\dots,L}_{i_{|\mathcal{R}|-t+1}-1}, R, E_{i_{\mathcal{R}-t+1}+1},\dots, E_{l-1} ).
 \]
By Lemma \ref{lem_ord_mut}, we have that the basic chain transformation entails the box move $\mathfrak{c}\mapsto\tau^-\mathfrak{c}$ if and only if 
 \begin{equation}
\label{eq_disequality_lemma_tecnico_entail_mutation}
    |\{ \imath_s \ |\  b-(t-1)-(i_{|\mathcal{R}|-t+1})\leq s\leq b-(t-1),\ \imath_s=\imath\}|\geq k+1.
\end{equation}
Moreover, since the sequence $(E_k)_{k\geq i_{{|\mathcal{R}|-t+1}}+1}$ contains exactly $m(t)$ expansion operators of index $\imath$, we have also
  \[|\{ \imath_s \ |\  b-(t-1)+1\leq s\leq b,\ \imath_s=\imath\}|= m(t).\]
Therefore, the inequality (\ref{eq_disequality_lemma_tecnico_entail_mutation}) is satisfied if and only if there is an integer $t'$ such that
\[b-(t-1)-(i_{|\mathcal{R}|-t+1})\leq t'\leq b\] and the $\mathfrak{c}^-_{b-t'+1}$ has index $\imath$ and $i$-cardinality $m(t)+k+1$.
Setting $i(m(t)+1+k,\imath)=b-t'+1$, we have 
\[i_{|\mathcal{R}|-t+1} + (t-1)\geq b-t'=i(m(t)+1+k,\imath)-1,\]
as required.
\end{proof}

\begin{lem}
\label{lem_last_box_move_determine_other}
For $1\leq t\leq |\mathcal{R}|$, let $\mathfrak{c}$ be an $i$-box of $\mathfrak{C}_{t-1}$ of index $\imath$ and $i$-cardinality $k$.  Suppose that the basic chain transformation $\mathfrak{C}_{t-1}\mapsto\mathfrak{C}_{t}$ sends $\mathfrak{c}$ to $\tau^-\mathfrak{c}$. Then, for any integer $t'$ such that $1\leq t'<t$ and $i(E_{i_{|\mathcal{R}|-t'+1}})=\imath$, the basic chain transformation $\mathfrak{C}_{t'-1}\mapsto\mathfrak{C}_{t'}$ sends $\mathfrak{c'}$ to $ \tau^-\mathfrak{c}'$, where $\mathfrak{c}'$ is the $i$-box of $\mathfrak{C}_{t'-1}$ of index $\imath$ and $i$-cardinality $k$.
\end{lem}

\begin{proof}
By Lemma \ref{lem_condition_box_move_basic_chain_trans}, we have that $\mathfrak{C}_-^{[a,b]}$ contains an $i$-box $\mathfrak{c}_{i(m(t)+1+k,\imath)}^-$ of index $\imath$ and $i$-cardinality $m(t)+1+k$ such that $i_{|\mathcal{R}|-t+1}+t\geq i(m(t)+1+k,\imath)$. Since $m(t)>m(t')$, there exists a positive integer $i(m(t')+1+k,\imath)\leq i(m(t)+1+k,\imath)$ such that the $i$-box $\mathfrak{c}_{i(m(t')+1+k,\imath)}^-$ has index $\imath$ and $i$-cardinality $m(t')+1+k$. Moreover, $i_{|\mathcal{R}|-t'+1}-i_{|\mathcal{R}|-t+1}\geq t-t'$. Therefore, we have that 
\[i_{|\mathcal{R}|-t'+1}+t'\geq i_{|\mathcal{R}|-t+1}+t\geq i(m(t)+1+k,\imath)\geq i(m(t')+1+k,\imath),\]
which, by Lemma \ref{lem_condition_box_move_basic_chain_trans}, gives the desired result.
\end{proof}

\begin{prop}
\label{prop_condition_m_flush_right}
Let $[a,b]$ be a finite interval of length $l$. Let $\mathfrak{C}$ be a chain of $i$-boxes of range $[a,b]$ and let $([c],(E_j)_{1\leq j\leq l-1})$ be the associated pair.
Let $\mathfrak{c}$ be a $i$-box $\mathfrak{C}$ of index $\imath$ and $i$-cardinality $k$ such that  $\mathfrak{c}=\tau^{-m'}\mathfrak{c}^-$, where $\mathfrak{c}^-$ is in $\mathfrak{C}_-^{[a,b]}$ and $m'\geq 0$. Let $m$ be a non-negative integer such that $(E_j)_{1\leq j\leq -l-1}$ contains at least $m$ expansion operators of index $\imath$.
Write $s$ for the largest integer such that 
\[|\{ j\in [s,l-1]\ |\ E_j=R, i(E_j)=\imath\}|\geq m,\]
and write $r$ for the non-negative integer
\[r=|\{j\in [s+1,l-1]\ |\ E_j=R\}|.\]
Then $m'\geq m$ if and only if there is an integer $1\leq i(m+k,\imath)\leq l$ such that 
\begin{itemize}
    \item the $i$-box $\mathfrak{c}^-_{i(m+k,\imath)}$ of $\mathfrak{C}_{-}^{[a,b]}$ has index $\imath$ and $i$-cardinality $m+k$;
    \item $s+r\geq i(m+k,\imath)-1$.
\end{itemize} 
\end{prop}

\begin{proof}
By Lemma \ref{lem_last_box_move_determine_other}, $m\geq m'$ if and only if the right expansion operator $E_s=\mathcal{R}$ entails a box move of the $i$-box of index $\imath$ and $i$-cardinality $k$.
The basic chain transformation associated to the expansion operator $E_s=R$ has the form
 \[
 ([b-r]_\imath; \underbrace{L,\dots,L}_{s-1}, L, E_{s+1},\dots, E_{l-1} )\mapsto 
 ([b-r-1]; \underbrace{L,\dots,L}_{s-1}, R, E_{s+1},\dots, E_{l-1} ).
 \] As in Construction \ref{con_sequence_mutations_Q(C)}, let $i_1<i_2<\dots$ be the elements of the set $\mathcal{R}=\{ j\in [1,l-1]\ |\ E_j=R\}.$ Then $E_s=E_{i_{|\mathcal{R}|-(r+1)+1}}$ and, with the notation of Lemma \ref{lem_condition_box_move_basic_chain_trans}, $m(r+1)+1=m$. Therefore, the desired equivalence follows by applying Lemma \ref{lem_condition_box_move_basic_chain_trans}.
\end{proof}

The following result discusses the form of the exchange matrices of monoidal seeds associated to sub-chains of $i$-boxes of a given chain.

\begin{lem}[{\cite[Prop. 7.14, Lem. 7.16]{KKOP_mon_cat_quant_aff_II}}]
\label{lem_uniqueness_exchange_matrix}
Let $\mathfrak{C}$ be a chain of $i$-boxes with finite range $[a,b]$ and consider a subchain $\mathfrak{C}'$ of length $s$.
Let $\mathcal{S}(\mathfrak{C})$ and $\mathcal{S}(\mathfrak{C'})$ be the  associated admissible monoidal seeds as in 
Theorem~\ref{teo_main_KKOP}. Then, 
\[B(\mathfrak{C})_{|K(\mathfrak{C}')\times K^{\text{ex}}(\mathfrak{C}')} = B(\mathfrak{C}')_{|K(\mathfrak{C}')\times K^{\text{ex}}(\mathfrak{C}')}\text{  and  } B(\mathfrak{C})_{|K(\mathfrak{C})\backslash K (\mathfrak{C}')\times K^{\text{ex}}(\mathfrak{C}')}=0.\]

\end{lem}

\begin{rem}[Chain of $i$-boxes with infinite length]
Let $[a,b]$ be an infinite interval.
Following the proof of \cite[Prop. 8.11]{KKOP_mon_cat_quant_aff_II}, in order to find the exchange matrix of the admissible seed associated to a chain of $i$-boxes $\mathfrak{C}=(\mathfrak{c}_k)_{k\geq 1}$ with range $[a,b]$, we consider the sequence $(\mathfrak{C}_s)_{s\geq 1}$ of chains 
of $i$-boxes defined by $\mathfrak{C}_s=(\mathfrak{c}_k)_{1\leq k\leq s}$.

By Lemma \ref{lem_uniqueness_exchange_matrix}, the coefficients of the exchange matrices stabilize, that is, for any $s\geq 1$, 
if $(i,j)\in K(\mathfrak{C}_s)\times K^{\text{ex}}(\mathfrak{C}_s)$, we have $B(\mathfrak{C}_t)_{ij}=B(\mathfrak{C}_s)_{ij}$ for any $t\geq s$.
It follows that, defining the $K(\mathfrak{C})\times K(\mathfrak{C})$-matrix $B(\mathfrak{C})=(b_{i,j})$ by
\[  b_{ij}=B(\mathfrak{C}_s)_{ij},\ \ \ (i,j)\in K(\mathfrak{C}_s)\times K^{\text{ex}}(\mathfrak{C}_s),\]
we have 
\[ B(\mathfrak{C})_{|K(\mathfrak{C}_s)\times K^{\text{ex}}(\mathfrak{C}_s)} = B(\mathfrak{C}_s) \ \ \ \text{for any } s\geq 1.\] 
In order to show that the seed $\mathcal{S}(\mathfrak{C})$ is completely admissible, one needs to verify that, for an arbitrary 
integer $N\geq 1$ and a sequence of integers $i_1,i_2,\dots, i_N$, the seed $\mathcal{S}(\mathfrak{C})$ admits the successive mutations 
$\mu_1(\mathcal{S}(\mathfrak{C})),\mu_2\mu_1(\mathcal{S}(\mathfrak{C}),\dots,\mu_N\dots\mu_2\mu_1(\mathcal{S}(\mathfrak{C}))$. Since the above is true after replacing $\mathfrak{C}$ with any $\mathfrak{C}_s$, where $s$ is large enough for all the $i_j$ to belong to
$K^{\text{ex}}(\mathfrak{C}_s)$,  by the construction of the matrix $B(\mathfrak{C})$, complete admissibility
also holds for the seed $\mathcal{S}(\mathfrak{C})$.
\end{rem}

\section{Reminder on cluster-tilting theory}
\label{subsec_reminder_clust_tilt}
Triangulated categories which are $2$-Calabi--Yau
and contain cluster-tilting objects are routinely used to additively categorify cluster 
algebras (without coefficients), cf. \cite{Reiten_clust_cat} for a survey.
In this subsection, we recall their definition and their properties. 

Let $k$ be an algebraically closed field and denote by $D=\text{Hom}_k(-,k)$ the duality over $k$. Let $\mathcal{C}$ be a triangulated $k$-category with suspension functor $\Sigma$.

We assume that

\begin{enumerate}
    \item The category $\mathcal{C}$ is Krull--Schmidt,  \emph{$\mathrm{Hom}$-finite} (that is, its spaces of morphisms are finite-dimensional) and that its idempotents split;
    \item the category $\mathcal{C}$ is  2-\emph{Calabi--Yau}, \emph{i.e.} there are bifunctorial isomorphisms
\[\text{Hom}_\mathcal{C}(X,Y)\cong D(\text{Hom}_\mathcal{C}(Y,\Sigma^2X))\ \ \ X,Y\in \mathcal{C}.\]
\end{enumerate}
For an object $X$ of $\mathcal{C}$, we denote by $\text{add}(X)$ the full subcategory of $\mathcal{C}$ whose objects are the direct factors of finite direct sums of copies of $X$.

\begin{de}(cluster-tilting object)
An object $T$ of $\mathcal{C}$ is \emph{(basic) cluster-tilting} if it is the sum of pairwise non-isomorphic indecomposables and the objects $X$ such that $\text{Ext}^1_{\mathcal{C}}(T,X)$ vanishes are the objects of $\text{add}(T)$.
\end{de}
Cluster-tilting objects in $\mathcal{C}$ play the role of clusters in a cluster algebra.
In the following, we assume that $T$ is a cluster tilting object of $\mathcal{C}$.
The following theorem allows to lift mutation of clusters to mutation of cluster-tilting objects.

\begin{teo}[\cite{Iyama_Yoshino_mut_tringcat_rigi_Cohen}]
Let $T_i$ be an  indecomposable direct factor of $T\cong T_i\oplus T_0$. Up to isomorphism, there is a unique indecomposable object $T_i^*$ not isomorphic to $T_i$ such that the object $T_i^*\oplus T_0$ obtained from $T$ by replacing $T_i$ with $T_i^*$ is cluster-tilting. 
Moreover, there are non split triangles,
unique up to isomorphism,

\begin{equation}
\label{exchange_triangles_IYAMA}
    T_i\xrightarrow{u} B\xrightarrow{v} T_i^* \xrightarrow{} \Sigma T_i \text{   and   } T_i^*\xrightarrow{u'} B'\xrightarrow{v'} T_i \xrightarrow{} \Sigma T_i^*
\end{equation} 
such that $u$ and $u'$ are minimal left $\text{add}(T_0)$-approximations and $v$ and $v'$ are minimal
right $\text{add}(T_0)$-approximations.
\end{teo}

We call $\mu_i(T)=T_i^*\oplus T_0$ the \emph{mutation} of $T$ at $T_i$. We refer to $(T_i,T^*_i)$ as an \emph{exchange pair} and to the triangles (\ref{exchange_triangles_IYAMA}) as the \emph{exchange triangles} of $T_i$ and $T_i^*$.

\begin{lem}[{\cite[Lemma 7.5]{Keller_clust_alg_quiv_rep_tring_cat}}]
\label{lem_exchange_triang_given_arrows}
The Gabriel quiver of the endomorphism algebra of $T$ does not have a loop at the vertex $i$ corresponding to the indecomposable $T_i$ if and only if 
\[
\dim \mathrm{Ext}_\mathcal{C}^1(T_i,T_i^*)=1.
\]
In this case, in the exchange triangles (\ref{exchange_triangles_IYAMA}), we have 
\[ 
B=\bigoplus_{i\rightarrow j} T_j \text{   and   } B'=\bigoplus_{j\rightarrow i} T_j,
\]
where the sums are taken over the sets of arrows with source respectively target $i$
in the Gabriel quiver of the endomorphism algebra of $T$.
\end{lem}

The next result states that, under certain assumptions, mutating at an indecomposable direct summand of a cluster-tilting object has the effect of a combinatorial mutation at the level of the quiver of its endomorphism algebra.

\begin{teo}[\cite{BIRS_Clust_struct_2CY_unipot_groups}]
If the quivers $Q$ and $Q'$ of the endomorphism algebras of $T$ and $\mu_i(T)$ do not have loops nor 2-cycles, then $Q'$ is the mutation of $Q$ at the vertex $i$.
\end{teo}

The cluster tilting object $T$ allows us to associate a triangle to each object of $\mathcal{C}$ as follows.

\begin{prop}[{\cite[\S 2]{Keller_Reiten}}]
\label{keller_reiten-approx}
   Let $X$ be an object of $\mathcal{C}$. Then there is a triangle 
\begin{equation}
   \label{approximation_triangles_Keller_Reiten}
       \Sigma^{-1} X\rightarrow T'' \rightarrow T' \rightarrow X
   \end{equation}
   such that $T''$ and $T'$ are objects of $\text{add}(T)$. 
\end{prop}

We refer to a triangle of the form (\ref{approximation_triangles_Keller_Reiten}) as an 
{\em approximation triangle} of $X$ by $T$. The  \emph{index} of $X$ is the difference $[T']-[T'']$ in the split Grothendieck
group $K_0(\mathrm{add}(T))$.
Approximation triangles find an application in the following theorem. 

\begin{teo}[Palu's generalized mutation rule \cite{Palu02}]
\label{teo_Palu}Assume that $\mathcal{C}$ is algebraic and let $T=\bigoplus_i T_i$ and $T'=\bigoplus_i T_j'$ be two basic cluster tilting objects in $\mathcal{C}$. Consider, for each summand $T_j'$ of $T'$, the approximation triangle of $T_j'$ by $T$:
 \[\Sigma^{-1} T'_j\rightarrow \bigoplus_i T_i^{\beta_{ij} }\rightarrow \bigoplus_i  T_i^{\alpha_{ij}} \rightarrow T'_j.\]
Define the matrix $P=(p_{ij})_{i,j\in J}$ by $p_{ij}=\alpha_{ij}-\beta_{ij}$.
Then the matrix $P$ is invertible and the matrices associated to the quivers of the endomorphism algebras of the cluster tilting objects $T$ and $T'$ are linked by the following formula:
\[
B_{T'}=P^{-1} B_{T} P^{-t},
\]
where $P^{-t}$ denotes the inverse of the transpose of $P$.
\end{teo}

The above result is the main tool in our approach to the determination of the exchange matrix $B(\mathfrak{C})$ of Theorem \ref{teo_main_KKOP}. Our strategy consists in constructing a suitable category $\mathcal{C}$ with cluster tilting objects $T$ and $T'$ such that the quivers of their endomorphism algebras are 
$Q^{[a,b]}(\underline{w}_0)$ and $Q(\mathfrak{C})$ respectively. In this way, we can compute $B(\mathfrak{C})$ by applying Theorem \ref{teo_Palu}. Following Amiot's work (\cite{Amiot_clust_cat_glob_dim_2}), we proceed to define two classes of 2-Calabi-Yau categories $\mathcal{C}$ with cluster tilting objects.

To start with, we recall the notion of a quiver with potential following \cite{DWZ_quiv_pot_rep_I}.
Let $Q$ be a finite quiver. We denote by $\widehat{kQ}$ the completion of the path algebra $kQ$ at the ideal generated by the arrows of $Q$ and we denote by $\mathrm{HH}_0$ the \emph{continuous Hochschild homology} of $\widehat{kQ}$, $i.e.$ the quotient of $\widehat{kQ}$ by the closure of the subspace generated by all commutators.  The vector space $\mathrm{HH}_0$ admits a topological basis formed by
the cyclic permutation classes of cyclic paths.
Let $W$ be a \emph{potential} on $Q$, i.e. an element $W$ of $\mathrm{HH}_0$ 
whose expansion in 
the basis of the cycles does not involve cycles of length zero. The quiver with potential $(Q,W)$ is \emph{reduced} if the expansion of $W$ does not invovle 2-$cycles$. If $(Q,W)$ is not reduced, Theorem 4.6 in \cite{DWZ_quiv_pot_rep_I} associates to it, in an essentially unique way, a reduced quiver with potential, called the \emph{reduced component} of $(Q,W)$.    \\
In \cite{DWZ_quiv_pot_rep_I}, the mutation operation of the classical theory of cluster algebra is adapted to the context of quivers with potential. Let $i$ be a vertex of $Q$. Assume that there are not loops or 2-cycles incident to $i$. Up to cyclic permutation, we can also assume that no cycle in the potential $W$ starts or ends at $i$. Construct from $Q$ the quiver $\Tilde{\mu}_i(Q)$ as follows:

\begin{enumerate}
    \item keep all the arrows not incident to $i$;
    \item replace each arrow $\alpha$ incident to $i$ with an opposite arrow $\alpha^*$.
    \item For any couple of arrows $\alpha: j\rightarrow i$ and $\beta: i\rightarrow k$, add an arrow $[\beta\alpha]: j\rightarrow k$.
\end{enumerate}

Endow $\Tilde{\mu}_i(Q)$ with the potential $\Tilde{\mu}_i(W)$ given by
\[ \Tilde{\mu}_i(W) = W_r + W_t,\]
where $W_r$ is obtained from $W$ by replacing each subpath of the form $j \xrightarrow{\alpha} i \xrightarrow{\beta} k$ appearing in a cycle of $W$ with the arrow $[\beta\alpha]$, and $W_t$ is given by
\[
W_t = \sum_{j \xrightarrow{\alpha} i \xrightarrow{\beta} k } [\alpha\beta]\alpha^*\beta^*.
\]

The \emph{mutation} of $(Q,W)$ at the vertex $i$, denoted by $\mu_i(Q,W)$, is defined as the reduced component of $(\Tilde{\mu}_i(Q),\Tilde{\mu}_i(W))$. 

For an arrow $\alpha$ of $Q$, we define the \emph{cyclic derivative with respect to} $\alpha$ as the unique 
continuous linear map 
\[ \partial_\alpha: \mathrm{HH}_0\rightarrow \widehat{kQ}\]
which sends the class of a path $p$ to the sum
\[ \sum_{p=uav} vu \]
taken over all the decompositions of $p$ as a concatenation of paths $u$, $a$ and $v$, where $u$ and $v$ are of length $\geq 0$. 
The \emph{Jacobi ideal} $(\partial_\alpha W)_{\alpha \in Q_1}$ associated to the quiver with potential $(Q,W)$ is the closure of the ideal of $\widehat{kQ}$ generated by the cyclic derivatives of the potential $W$ with respect to all the arrows of $Q$. We define the \emph{Jacobi algebra} of $(Q,W)$ as the quotient
\[ J(Q,W) = \widehat{kQ}/(\partial_\alpha W)_{\alpha \in Q_1} .\]
We say that the quiver with potential $(Q,W)$ is

\begin{enumerate}
\item \emph{Jacobi finite} if the Jacobi algebra is finite-dimensional;
\item \emph{non-degenerate} if any sequence of iterated mutations of $(Q,W)$ does not produce 2-cycles;
\item \emph{rigid} if, up to cyclic permutation, any cycle of $Q$ is contained in the Jacobi ideal. 
\end{enumerate}

By \cite[Cor.~8.2]{DWZ_quiv_pot_rep_I}, any rigid quiver with potential is non-degenerate. 

Next, following \cite{Ginzburg_CY}, we consider a differential graded algebra associated to $(Q,W)$ 
(we refer to \cite{Keller_dg_cat} for basic notions on diffential graded algebras). Let $\Tilde{Q}$ be the graded quiver with the same set of vertices as $Q$ and with arrows
\begin{itemize}
    \item the arrows of $Q$ in degree 0;
    \item an arrow $\alpha^*: j\rightarrow i$ in degree $-1$ for each arrow $\alpha: i\rightarrow j$  of $Q$;
    \item a loop $t_i$ in degree $-2$ at each vertex $i$ of $Q$.
\end{itemize}
The \emph{Ginzburg algebra} $\widehat{\Gamma}(Q,W)$ is the differential graded algebra whose underlying graded algebra is the completed path algebra (formed in the category of graded algebras) of $\Tilde{Q}$ and whose differential $d$ is the unique continuous linear endomorphism of degree 1 satisfying the Leibniz rule (i.e.
$d(pq) = (dp)q + (-1)^apdq$ for all homogeneous elements $p$ of degree $a$ and all $q$) and acting on the generators as follows:

\begin{itemize}
    \item $d\alpha=0$ for any arrow $\alpha$ of $Q$;
    \item $d\alpha^*=\partial_\alpha W$ for any arrow $\alpha$ of $Q$;
    \item $d t_i= e_i (\sum_{\alpha \in Q_1} \alpha\alpha^* - \alpha^*\alpha) e_i$ for any vertex $i$ of $Q$, where $e_i$ denotes the idempotent at the vertex $i$.
\end{itemize}

Let us abbreviate $\widehat{\Gamma}(Q,W)$ by $\Gamma$. 
We denote by $D\Gamma$ the derived category of the dg algebra $\Gamma$ (its
objects are all dg right $\Gamma$-modules) and by $\mathrm{per}(\Gamma)$ the
{\em perfect derived category}, i.e. the thick subcategory of $D\Gamma$ generated
by the free right $\Gamma$-module of rank $1$. The {\em perfectly valued}
derived category $\mathrm{pvd}(\Gamma)$ is the full subcategory
of $D\Gamma$ whose objects are the dg $\Gamma$-modules $M$
whose underlying complex of vector spaces is perfect (i.e. the homology
of $M$ is of finite total dimension). The category
$\mathrm{pvd}(\Gamma)$ is contained in $\mathrm{per}(\Gamma)$, 
cf.~\cite{Keller_Yang_derived_eq_muta_pot}.

\begin{de}[{\cite[Def.~3.5]{Amiot_clust_cat_glob_dim_2}}] The  \emph{cluster category} associated
with $(Q,W)$ is the Verdier quotient
\[
\mathcal{C}_{Q,W} = \mathrm{per}(\Gamma)/\mathrm{pvd}(\Gamma).
\] 
\end{de}

\begin{teo}[{\cite[Thm.~3.6]{Amiot_clust_cat_glob_dim_2}}]
\label{teo_prop_(Q;W)_jacobi_finite}
Suppose that $(Q,W)$ is Jacobi finite.  Then
\begin{enumerate}
    \item the cluster category $\mathcal{C}_{Q,W}$ is
    $\mathrm{Hom}$-finite, Krull--Schmidt and $2$-Calabi--Yau;
    \item the image of $\Gamma$ in $\mathcal{C}_{Q,W}$ is a cluster-tilting object whose endomorphism algebra is isomorphic to the Jacobi algebra $J(Q,W)$.
\end{enumerate}
\end{teo}

We call the image  of $\Gamma$ in $\mathcal{C}_{Q,W}$ the \emph{canonical} cluster tilting-object of $\mathcal{C}_{Q,W}$.

\begin{teo}[\cite{Keller_Yang_derived_eq_muta_pot}]
\label{teo_mutation_T_compatible_mutation_Q}
Assume that the quiver $Q$ is without loops and 2-cycles, that the potential $Q$ is rigid and that the quiver with potential $(Q,W)$ is Jacobi-finite. Let $T$ be a cluster-tilting object  of $\mathcal{C}_{Q,W}$ obtained through iterated mutations from the canonical one and let $Q(T)$ the quiver of its endomorphism algebra. Then, for any indecomposable summand $T_i$ of $T$ $(i\in Q_0)$, the quiver of the endomorphism algebra of $\mu_i(T)$ coincides with $\mu_i(Q(T))$.
\end{teo}

\begin{proof}
    The assumption that the quiver with potential $(Q,W)$ is Jacobi-finite allows to apply Theorem \ref{teo_prop_(Q;W)_jacobi_finite}, while the other assumptions ensures that, at the level of the underlying quivers, any sequence of mutations of the quiver with potential $(Q,W)$ corresponds to the associated sequence of classical quiver mutations. Therefore, it follows from  \cite[Thm.~3.2]{Keller_Yang_derived_eq_muta_pot} that, if $T'$ is a cluster tilting object of $\mathcal{C}_{Q,W}$ and $i_1,\dots,i_N$ is a finite sequence of vertices of $Q_0$ such that $T'=\mu_{i_N}\dots \mu_{i_1}(T)$, then the endoquiver of $T'$ is $Q(T')=\mu_{i_N}\dots \mu_{i_1}(Q)$.
\end{proof}

The second class of $2$-Calabi-Yau categories with cluster-tilting objects defined by Amiot is associated to certain algebras of global dimension at most 2.
To start with, if $A$ is a finite-dimensional $k$-algebra of finite global dimension, we denote by $\mathcal{D}^b(A)$ the bounded derived category of the category of finite-dimensional right 
$A$-modules and by $\nu_A$ its Serre functor, which is given by the total derived functor of the tensor 
product $-\otimes_A DA$. We define the {\em orbit category}
\[
\mathcal{D}^b(A)/(\nu_A\Sigma^{-2})
\]
as in \cite{Keller_orbit_cat}: It has the same objects as $\mathcal{D}^b(A)$ 
and the space of morphisms between two objects $X$ and $Y$ is defined as
\[ \bigoplus_{i\in \mathbb{Z}}\mathrm{Hom}_{\mathcal{D}^b(A)}(X,\nu_A^i\Sigma^{-2i}(Y)).\]
In general, the orbit category is not triangulated, but we can embed it fully faithfully into a ‘smallest
triangulated overcategory’, its \emph{triangulated hull} (see \cite{Keller_orbit_cat} for details). 
In particular, each short exact sequence $0\rightarrow M'\rightarrow M \rightarrow M'' \rightarrow 0$ of 
right $A$-modules gives rise to a triangle $M'\rightarrow M \rightarrow M'' \rightarrow \Sigma M$ in the triangulated hull of the orbit category.

\begin{de}[\cite{Amiot_clust_cat_glob_dim_2}] Let $A$ be a finite dimensional $k$-algebra of finite global dimension. The \emph{cluster category} $\mathcal{C}_A$ is the 
triangulated hull of the orbit category $D^b(A)/(\nu_A\Sigma^{-2})$.
\end{de}

\begin{teo}[{\cite[Thm. 4.10]{Amiot_clust_cat_glob_dim_2}}]
Let $A$ be a finite dimensional $k$-algebra of global dimension at most 2.  If the functor $\mathrm{Tor}_2^A(-,DA)$ is nilpotent, then 
\begin{enumerate}
    \item the cluster category $\mathcal{C}_A$ is $\mathrm{Hom}$-$finite$ and $2$-Calabi-Yau;
    \item the image of $A$ in $\mathcal{C}_A$ is a cluster-tilting object.
\end{enumerate}
\end{teo}
Under the hypothesis of the theorem, we call the image of $A$ in $\mathcal{C}_A$ the \emph{canonical} cluster-tilting object of $\mathcal{C}_A$.\\

Next, we recall a link between these two classes of cluster categories. Suppose that a finite dimensional $k$-algebra $A$ 
is of global dimension at most $2$ and of the form $kQ'/I$, where $I$ is an ideal contained in the square of the ideal $S$ generated by the arrows of $Q'$. We define a quiver with potential $(Q_A,W_A)$ as follows. Let $R$ be a \emph{minimal set of relations} (\textit{cf.} \cite{Bongartz_alg_quad_forms}), \textit{i.e.} the union over the pairs of vertices $(i,j)$ of a set of representatives of a basis of 
\[e_j(I/(IS+SI))e_i.\]
Let $Q_A$ be the quiver obtained from $Q'$ by adding an arrow $\rho:j\rightarrow i$ for each minimal relation from $i$ to $j$. We endow  $Q_A$ with the potential given by
\[ W_A = \sum_{r\in R} r\rho_r.\]

\begin{teo}[{\cite[Thm. 6.12]{Keller_def_CY_completions}}]
\label{teo_link_cluster_categories}
Let $A$ be as above. 
\begin{enumerate}
    \item There is a canonical triangle equivalence from  $\mathcal{C}_A$ to the 
    cluster category $\mathcal{C}_{Q_A,W_A}$. 
    \item The canonical cluster-tilting object of $\mathcal{C}_A$ is sent by this equivalence to the 
    canonical cluster-tilting object of $\mathcal{C}_{Q_A,W_A}$.
    \end{enumerate}
\end{teo}

One of the advantages of working with cluster categories associated to algebras of global dimension at most 2 is that, in some cases, besides the image of $A$ itself, the images in $\mathcal{C}_A$ of other $A$-modules are cluster-tilting objects. Moreover, some exchange triangles and approximation triangles can be derived from short exact sequences of modules. Keeping in mind this point of view, we start with the following lemma.
Recall that there is an isomorphism of functors 
\[\tau_{\mathcal{D}}\Sigma \cong \nu_A,\] where  $\tau_{\mathcal{D}}$ denotes the Auslander--Reiten translation of the derived category of $A$.\\

The following Lemma sligthly extends \cite[Prop. 2.12]{Amiot_clust_cat_glob_dim_2} in the case of cluster categories associated to algebras of global dimension at most 2.

\begin{lem}
\label{lem_ext_clustcat_X_projdim1}
  Let $A$ be an algebra of global dimension at most 2 such that $\mathcal{C}_A$ is $\mathrm{Hom}$-$\mathrm{finite}$. Let $X$ and $Y$ be finite dimensional $A$-modules. Then if $\mathrm{pdim}X\leq 1$, there exists a $k$-linear isomorphism
  \[ \mathrm{Ext}_{\mathcal{C}_A}^1(X,Y) \cong  \mathrm{Ext}_A^1(X,Y)\oplus D\mathrm{Ext}_A^1(Y,X).\]
\end{lem}

\begin{proof}
Since the orbit category embeds fully faithfully into its triangulated hull, we have
\[ \mathrm{Ext}_{\mathcal{C}_A}^1(X,Y) = \bigoplus_{p\in\mathbb{Z}} \mathrm{Hom}_{\mathcal{D}^b(A)}(X, \Sigma^{p+1} \tau_{\mathcal{D}}^{-p} Y).\]
For $p\leq -2$, we have 
\[\tau^{-p}(Y) \in \mathcal{D}^{\geq 0}(A)  \text{    and    }  \Sigma^{p+1}\tau_{\mathcal{D}}^{-p}(Y) \in \mathcal{D}^{\geq -(p+1)}(A) \subset \mathcal{D}^{\geq 1}(A).\]Therefore, \[\mathrm{Hom}_{\mathcal{D}^b(A)}(X, \Sigma^{p+1} \tau^{-p} Y)=0.\]
For $p=-1$, 
\[\mathrm{Hom}_{\mathcal{D}^b(A)}(X,\tau_{\mathcal{D}} Y) =\mathrm{Hom}_{\mathcal{D}^b(A)}(X,\nu_A \Sigma^{-1} Y) = D\mathrm{Hom}_{\mathcal{D}^b(A)}(\Sigma^{-1}Y,\tau_{\mathcal{D}} X) =  D\mathrm{Ext}^1_{A}(Y,X).
\]
For $p=0$,  
\[\mathrm{Hom}_{\mathcal{D}^b(A)}(X,\Sigma Y) = \mathrm{Ext}_A^1(X,Y).\]
For $p>0$, we have 
\[\tau_{\mathcal{D}}^{-p}(Y) \in \mathcal{D}^{\leq 0}(A)  \text{    and    }  \Sigma^{p+1}\tau_{\mathcal{D}}^{-p}(Y) \in \mathcal{D}^{\leq -(p+1)}(A) \subset \mathcal{D}^{\leq -2}(A).\]
Therefore, since the projective dimension of $X$ is less or equal than $1$, we have
\[ \mathrm{Hom}_{\mathcal{D}^b(A)}(X, \Sigma^{p+1} \tau_{\mathcal{D}}^{-p} Y)=0. \]

\end{proof}

Recalling the definition of a (classical)  tilting module.

\begin{de}
Let $A$ be a finite-dimensional $k$-algebra. A finite dimensional $A$-module $T$ is a \emph{tilting module} if the following conditions are satisfied:
\begin{itemize}
    \item[(T1)] the projective dimension of $T$ is at most 1;
    \item[(T2)] $\mathrm{Ext}^1_A(T,T)=0$;
    \item[(T3)] the number of isomorphism classes of the indecomposable direct summands of $T$ is equal to the number of isomorphism classes of the simple $A$-modules.
\end{itemize}
\end{de}

For any finite-dimensional $k$-algebra $A$, the free module of rank one $A$ is a tilting module, referred to as \emph{canonical}.

\begin{prop}
\label{prop_tilting_becomes_clust_tilt}
Let $A$ be an algebra of global dimension at most 2 such that $\mathcal{C}_A$ is $Hom$-$finite$. Let $T$ be a basic tilting module of $A$. Then the image of $T$ in $\mathcal{C}_A$ is a cluster tilting object.
\end{prop}

\begin{proof}
Let $\Tilde{T}$ be the image of $T$ in $\mathcal{C}_A$. By \cite[Thm.~4.1]{AIR_tau_tilting}, $\Tilde{T}$ is a cluster-tilting object if and only if it is rigid. Therefore, the results follows from Lemma \ref{lem_ext_clustcat_X_projdim1}.
\end{proof}

We conclude this section  by recalling the so called \emph{Calabi--Yau reduction}. Assume that the category $\mathcal{C}$ is 2-Calabi--Yau, $\mathrm{Hom}$-finite, Krull--Schmidt and that it contains a cluster-tilting object $T=\oplus_{j\in J} T_i$ such that the $T_i$ are its indecomposable summands. For a subset $K$ of $J$, let $\mathcal{C}_K$ be the full subcategory of $\mathcal{C}$ formed by those objects $X$ such that $\mathrm{Ext}_\mathcal{C}(X,T_i)=0$ for any $i\in K$ and let $\langle T_i \rangle_{i\in K} $ be the ideal of $\mathcal{C}_K$ generated by the identities of the objects $T_i, i \in K $. 
As shown in \cite[\S ~4]{Iyama_Yoshino_mut_tringcat_rigi_Cohen}, the additive quotient $\mathcal{C}_K/\langle T_i \rangle_{i\in K} $ carries a canonical triangulated structure, is $\mathrm{Hom}$-finite, Krull--Schmidt, 2-Calabi--Yau and contains the image of $T$ as a cluster-tilting object.

\begin{teo}[{\cite[\S ~4]{Iyama_Yoshino_mut_tringcat_rigi_Cohen}}]
\label{teo_CY_reduction}
The projection functor
\[ \mathcal{C}_K \rightarrow \mathcal{C}_K/\langle T_i\rangle_{i\in K}
\]
induces a bijection between the cluster-tilting objects of $\mathcal{C}$ having, for any $i \in K$, $T_i$ as an indecomposable summand, and the cluster-tilting objects of $\mathcal{C}_K/\langle T_i\rangle_{i\in K}$. Moreover, it preserves exchange and approximation triangles.
\end{teo}

As a particular case of Calabi--Yau reduction, suppose that $\mathcal{C}=\mathcal{C}_A$ is the cluster category of a finite dimensional $k$-algebra $A$ of global dimension at most 2 of the form $kQ'/I$ as above and let $K$ be a subset of the vertices of $Q'$. For any $i\in K$, let $P_i$ the associated indecomposable projective $A$-module. Then it follows from \cite[\S~7]{Keller_def_CY_completions} that there is a triangle equivalence 

\[
\mathcal{C}_K/\langle P_i\rangle_{i\in K} \cong \mathcal{C}_{A / \langle e_i\ |\ i \in K\rangle}.
\]

\section{Cluster categories associated to $Q(\underline{w}_0)$}
In this section we translate the problem of the determination of the exchange matrices $B(\mathfrak{C})$ of Theorem \ref{teo_main_KKOP} in a setting of additive categorification and we provide a solution in Theorem~\ref{teo_sol_KKOP}.
In the following, we keep the notations $\mathfrak{g}$, $\mathcal{Q}$, $\underline{w}_0$ of Section \ref{sec_mo_cat}. 
\\

\subsection{A Jacobi algebra associated to the initial seed}

By the definition of the quiver $Q(\underline{w}_0)$, for any pair of vertices $s,t$ such that $\imath_s\sim\imath_t$ and $s^-<t^-<s<t$, there is an arrow $s\rightarrow t$. Moreover, denoting by $\Tilde{t}$ the smallest integer greater than $s^-$ such that $\imath_{\Tilde{t}}=\imath_t$, there is also an arrow $\Tilde{t}\rightarrow s$. Therefore, for any such pair $s,t$, there is in $Q(\underline{w}_0)$ an oriented cycle

\begin{equation}
\label{cycle_potential}
\begin{tikzcd}
   & s \arrow[dr] & \\
    \Tilde{t} \arrow[ur] &\dots \arrow[l] & t \arrow[l].
\end{tikzcd} 
\end{equation}

Let $[a,b]$ be a finite interval. We endow $Q^{[a,b]}(\underline{w}_0)$ with the potential $W^{[a,b]}(\underline{w}_0)$ given by the alternated sum of all the oriented cycles of type (\ref{cycle_potential}) appearing in $Q^{[a,b]}(\underline{w}_0)$ (notice that we can equivalently describe this potential as the alternated sum of the cordless cycles of $Q^{[a,b]}(\underline{w}_0)$). When $\underline{w}_0$ is $\mathcal{Q}$-adapted, this quiver with potential was first introduced in \cite{HL_Clust_alg_approach_q_char}. Let $J^{[a,b]}(\underline{w}_0)$ be the associated Jacobi algebra.

\begin{prop} \mbox{}
\label{prop_rigid_quiver}
\begin{enumerate}
    \item The Jacobi algebra $J^{[a,b]}(\underline{w}_0)$ is finite-dimensional.
    \item The quiver with potential $(Q^{[a,b]}(\underline{w}_0), W^{[a,b]}(\underline{w}_0))$ is rigid.
\end{enumerate}
    
\end{prop}

Before giving a proof, we state some preliminary results.
Recall the notation $\underline{w}_0=s_{\imath_1}\dots s_{\imath_{l_0}}$  and let $\underline{w}'_0=s_{\jmath_1}\dots s_{\jmath_{l_0}}$ be an other reduced expression of $w_0$. Following \cite[Remark.~8.5]{KKOP_mon_cat_quant_aff_II}, we say that

\begin{itemize}
    \item[(i)] $\underline{w}'_0$ is obtained from $\underline{w}_0$ by a \emph{commutation move} at an integer $1<k\leq l_0$ if 
    \[ \imath_k\not\sim \imath_{k-1},\  \imath_k=\jmath_{k-1},\ \imath_{k-1}=\jmath_{k} \text{ and }  \imath_s=\jmath_s \text{ for any } s\neq k, k-1. \]
    In this case, the quiver $Q(\underline{w}_0)$ is isomorphic to the quiver $Q(\underline{w}'_0)$ via the index change $k-1 + tl_0 \leftrightarrow k + tl_0,\ t\in\mathbb{Z}$.
    \item[(ii)] $\underline{w}'_0$ is obtained from $\underline{w}_0$ by a \emph{braid move} at an integer $1<k< l_0$ if
    \[ \imath_k\sim \imath_{k-1},\  \imath_k=\jmath_{k\pm 1},\ \imath_{k\pm 1}=\jmath_{k} \text{ and }  \imath_s=\jmath_s \text{ for any } s\neq k, k\pm 1. \]
    In this case, the quiver obtained from $Q(\underline{w}_0)$ by mutating at the vertices $k+1+tl_0,t\in\mathbb{Z}$ is isomorphic to the quiver $Q(\underline{w}'_0)$ (\cite[Lemma~8.9]{KKOP_mon_cat_quant_aff_II}).
\end{itemize}

It is a general fact from Lie theory that any two reduced expressions of $w_0$ can be obtained from each other by a composition of braid moves and commutation moves. For a finite-dimensional simple Lie algebra $\mathsf{g}$, the above terminology can be applied also to sequences of elements of $I_\mathsf{g}$ indexed by $\mathbb{Z}$ and satisfying the condition

\begin{equation}
\label{condition_cardinality_inf_seq}
|\{\imath_k\ |\ \imath_k=\imath, k\leq 0\}|=\infty = |\{\imath_k\ |\ \imath_k=\imath, k> 0\}|\ \ \imath\in I_\mathsf{g}.
\end{equation}

For such a sequence $\mathfrak{i}=(\imath_k)_{k\in \mathbb{Z}}$, we define the quiver $Q(\mathfrak{i})$ with vertex set $\mathbb{Z}$ and arrows

\begin{equation*}
\{s\rightarrow s^- | s\in\mathbb{Z}\} \cup \{s\rightarrow t | s^-<t^-<s<t,\ \imath_s\sim\imath_t \}.
\end{equation*}

 Notice that the quiver $Q(\underline{w}_0)$ is a particular case of this definition, with $\mathsf{g}=\check{\mathfrak{g}}$ and $\mathfrak{i}=\widehat{\underline{w}}_0$.
 For any possibly infinite interval $[a,b]$, we define the quiver with potential $(Q^{[a,b]}(\mathfrak{i}), W^{[a,b]}(\mathfrak{i}))$ as a generalization of  $(Q^{[a,b]}(\underline{w}_0), W^{[a,b]}(\underline{w}_0))$.

\begin{lem}
\label{lem_mut_quiv_pot_preserva}
Let $\mathsf{g}$ be a finite dimensional Lie algebra and let $\mathfrak{i}=(\imath_k)_{k\in\mathbb{Z}}$ be a sequence of elements of $I_\mathsf{g}$ satisfying $(\ref{condition_cardinality_inf_seq})$. Let $s\in\mathbb{Z}$ and assume that the sequence $\mathfrak{j}=(\jmath_k)_{k\in\mathbb{Z}}$ is obtained from $\mathfrak{i}$ by a braid move at $s$.
Then there is an isomorphism of quivers with potential 
\[\mu_{s+1}(Q(\mathfrak{i}), W(\mathfrak{i})) \cong (Q(\mathfrak{j}), W(\mathfrak{j})).\]
\end{lem}

\begin{proof}
In this proof, when we consider a cycle of a potential, we keep track implicitly of the attached sign.\\
It follows from the definition of the quiver $Q(\mathfrak{i})$ that the full subquiver $Q_{s+1}(\mathfrak{i})$ on the vertices adjacent to $s+1$ has the form

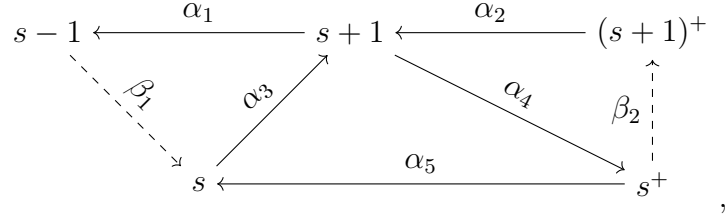
\[
\begin{tikzpicture}
    \node (s+1) at (0,0) {$s+1$};
    \node (s+1^+) at (4,0) {$(s+1)^+$};
    \node (s-1) at (-4,0) {$s-1$};
    \node (s) at (-2,-2) {$s$};
    \node (s^+) at (4,-2) {$s^+$};

    \draw[->] (s+1^+) -- (s+1) node [pos=.5, above] {$\alpha_2$};
    \draw[->] (s+1) -- (s^+) node [pos=.5, above, sloped] {$\alpha_4$};
    \draw[->] (s^+) -- (s) node [pos=.5, above] {$\alpha_5$};
    \draw[->] (s) --  (s+1) node [pos=.5, above, sloped] {$\alpha_3$};
    \draw[->] (s+1) -- (s-1) node [pos=.5, above] {$\alpha_1$};
    \draw[->, dashed] (s^+) -- (s+1^+) node [pos=.5, left] {$\beta_2$};
    \draw[->, dashed] (s-1) -- (s) node [pos=.5, above, sloped] {$\beta_1$};
\end{tikzpicture},
\]
where the dotted arrows $\beta_1$ and $\beta_2$ are present if and only if $(s-1)^-<s^-$ and $s^+< (s+1)^+$. We assume for simplicity that neither $\beta_1$ nor $\beta_2$ are in the picture (the other cases can be treated similarly). Moreover, adding the labels of the cycles of the potential $W(\mathfrak{i})$ which involve the arrows of $Q_{s+1}(\mathfrak{i})$, we get

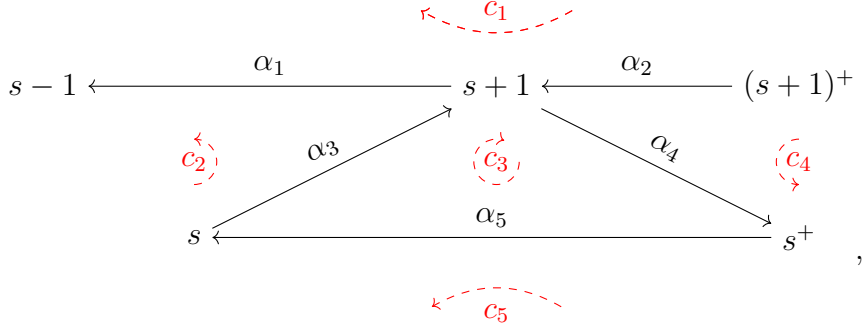
\[
\begin{tikzpicture}
    \node (s+1) at (0,0) {$s+1$};
    \node (s+1^+) at (4,0) {$(s+1)^+$};
    \node (s-1) at (-6,0) {$s-1$};
    \node (s) at (-4,-2) {$s$};
    \node (s^+) at (4,-2) {$s^+$};

    \node[red] (c_1) at (0,1) {$c_1$};
    \node[red] (c_2) at (-4,-1) {$c_2$};
    \node[red] (c_3) at (0,-1) {$c_3$};
    \node[red] (c_4) at (4,-1) {$c_4$};
    \node[red] (c_5) at (0,-3) {$c_5$};

    \node (ghost_c5_1) at (-1,-3) {};
    \node (ghost_c5_2) at (1,-3) {};
    \node (comma) at (4.8,-2.3) {,};

\draw[red, dashed, ->] (0.3,-1) arc (0:-290:0.3);
\draw[->, red, dashed] (ghost_c5_2) to [bend right]  (ghost_c5_1);
\draw[->, red, dashed] (1,1) to [bend left]  (-1,1);
\draw[->, red, dashed] (1,1) to [bend left]  (-1,1);
\draw[red, dashed, ->] (4,-0.7) arc (90:270:0.3);
\draw[red, dashed, ->] (-4,-1.3) arc (-90:90:0.3);

    \draw[->] (s+1^+) -- (s+1) node [pos=.5, above] {$\alpha_2$};
    \draw[->] (s+1) -- (s^+) node [pos=.5, above, sloped] {$\alpha_4$};
    \draw[->] (s^+) -- (s) node [pos=.5, above] {$\alpha_5$};
    \draw[->] (s) --  (s+1) node [pos=.5, above, sloped] {$\alpha_3$};
    \draw[->] (s+1) -- (s-1) node [pos=.5, above] {$\alpha_1$};
    
\end{tikzpicture}
\]
where 
\begin{itemize}
    \item $c_1$ is the sum of the cycles of $W(\mathfrak{i})$ containing $\alpha_1\alpha_2$ as a subpath. Its terms are indexed by the $\imath\in I_\mathsf{g}$ such that $\imath\sim \imath_{s+1}$ ;
    \item $c_2$ is the cycle of $W(\mathfrak{i})$ containing $\alpha_1\alpha_3$ as a subpath;
    \item $c_3$ is the cycle $\alpha_5\alpha_4\alpha_3$;
    \item $c_4$ is the cycle of $W(\mathfrak{i})$ containing $\alpha_4\alpha_2$ as a subpath;
    \item $c_5$ is the sum of the cycles of $W(\mathfrak{i})$ different from $c_3$ and containing $\alpha_5$ as a subpath. Its terms are indexed by the $\imath\in I_\mathsf{g}$ such that $\imath\sim \imath_{s}$ ;
\end{itemize}

As recalled in section \ref{subsec_reminder_clust_tilt}, the first step of the mutation of the quiver with potential $(Q(\mathfrak{i}), W(\mathfrak{i}))$ at $s+1$ requires to compute the quiver with potential $(\Tilde{\mu}_{s+1}Q(\mathfrak{i}), \Tilde{\mu}_{s+1} W(\mathfrak{i}))$. The quiver $\Tilde{\mu}_{s+1}Q(\mathfrak{i})$ appears as

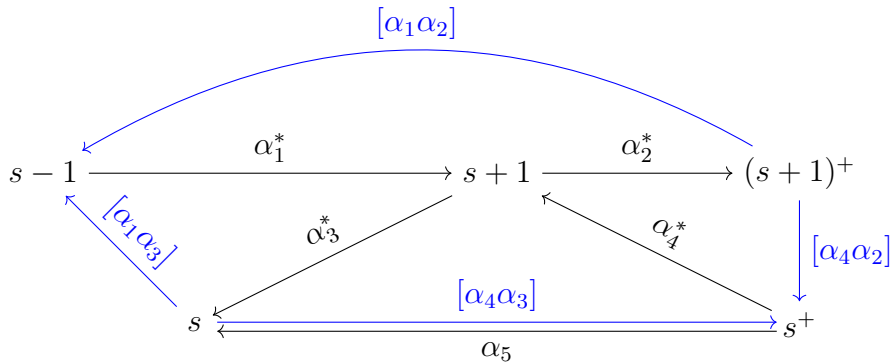
\[
\begin{tikzpicture}
    \node (s+1) at (0,0) {$s+1$};
    \node (s+1^+) at (4,0) {$(s+1)^+$};
    \node (s-1) at (-6,0) {$s-1$};
    \node (s) at (-4,-2) {$s$};
    \node (s^+) at (4,-2) {$s^+$};

    \draw[->, blue] (s+1^+) to [bend right] node [pos=.5, above] {$[\alpha_1\alpha_2]$} (s-1)  ;
    \draw[->] (s+1) -- (s+1^+) node [pos=.5, above] {$\alpha_2^*$};
    \draw[->] (s^+) -- (s+1) node [pos=.5, above, sloped] {$\alpha_4^*$};
    \draw[->,blue] (-3.7,-1.97) -- (3.7, -1.97) node [pos=.5, above] {$[\alpha_4\alpha_3]$};
    \draw[->] (s+1) --  (s) node [pos=.5, above, sloped] {$\alpha_3^*$};
    \draw[->] (s-1) -- (s+1) node [pos=.5, above] {$\alpha_1^*$};
    \draw[->, blue] (s) -- (s-1) node [pos=.5, above, sloped] {$[\alpha_1\alpha_3]$};
    \draw[->, blue] (s+1^+) -- (s^+) node [pos=.5, right] {$[\alpha_4\alpha_2]$};
    \draw[->] (3.7,-2.09) -- (-3.7,-2.09) node [pos=.5, below] {$\alpha_5$};
    
\end{tikzpicture}.
\]
The potential $\Tilde{\mu}_{s+1} W(\mathfrak{i})$ is given by 
\[\Tilde{\mu}_{s+1} W(\mathfrak{i}) = \Tilde{W} +  W_t+\sum_{1\leq i\leq 4} \Tilde{c}_i,\]
where \[\Tilde{W} = W(\mathfrak{i}) - \sum_{1\leq i\leq 5} c_i,\]
    \[ W_t=[\alpha_1\alpha_2]\alpha_2^*\alpha_1^*+[\alpha_1\alpha_3]\alpha_3^*\alpha_1^*+[\alpha_4\alpha_3]\alpha_3^*\alpha_4^*+[\alpha_4\alpha_2]\alpha_2^*\alpha_4^*\]
and, for any $1\leq i\leq 4$, $\Tilde{c}_i$ is obtained from $c_i$ by replacing the associated subpath $\alpha_{j_1}\alpha_{j_2}$ with
$[\alpha_{j_1}\alpha_{j_2}]$. Notice that the unique 2-cycle in $\Tilde{\mu}_{s+1} W(\mathfrak{i})$ is $\Tilde{c}_3=\alpha_5[\alpha_4\alpha_3]$.
Following the proof of \cite[Lem.~4.8]{DWZ_quiv_pot_rep_I}, we obtain that there is a unitriangular automorphism $\varphi$
(see \cite[Def.~2.5]{DWZ_quiv_pot_rep_I}) of the completed path algebra of $\Tilde{\mu}_{s+1}Q(\mathfrak{i})$ sending $\Tilde{\mu}_{s+1} W(\mathfrak{i})$ to 
\[ \varphi(\Tilde{\mu}_{s+1} W(\mathfrak{i})) = \alpha_5[\alpha_4\alpha_3] + c'_5 + \Tilde{W} + (W_t-[\alpha_4\alpha_3]\alpha_3^*\alpha_4*) +\sum_{1\leq i\leq 4} \Tilde{c}_i,  \]
where $c'_5$ is the path obtained from $c_5$ by replacing $\alpha_5$ with the composition $\alpha_3^*\alpha_4^*$. 
Therefore, by \cite[Thm.~4.6]{DWZ_quiv_pot_rep_I}, after rearranging the position of its vertices, the quiver $\mu_{s+1} Q(\mathfrak{i})$ of the reduced component of $(\Tilde{\mu}_{s+1}Q(\mathfrak{i}),\Tilde{\mu}_{s+1}W(\mathfrak{i}))$, restriced to the neighbour of $s+1$, looks as 

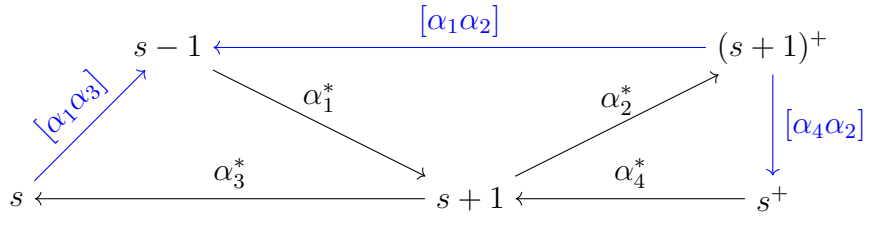
\[
\begin{tikzpicture}
    \node (s+1) at (0,-2) {$s+1$};
    \node (s+1^+) at (4,0) {$(s+1)^+$};
    \node (s-1) at (-4,0) {$s-1$};
    \node (s) at (-6,-2) {$s$};
    \node (s^+) at (4,-2) {$s^+$};

    \draw[->, blue] (s+1^+) to node [pos=.5, above] {$[\alpha_1\alpha_2]$} (s-1)  ;
    \draw[->] (s+1) -- (s+1^+) node [pos=.5, above] {$\alpha_2^*$};
    \draw[->] (s^+) -- (s+1) node [pos=.5, above, sloped] {$\alpha_4^*$};
    \draw[->] (s+1) --  (s) node [pos=.5, above, sloped] {$\alpha_3^*$};
    \draw[->] (s-1) -- (s+1) node [pos=.5, above] {$\alpha_1^*$};
    \draw[->, blue] (s) -- (s-1) node [pos=.5, above, sloped] {$[\alpha_1\alpha_3]$};
    \draw[->, blue] (s+1^+) -- (s^+) node [pos=.5, right] {$[\alpha_4\alpha_2]$};

\end{tikzpicture},
\]
and its potential is given by 
\[ \mu_{s+1} W(\mathfrak{i}) =  c'_5 + \Tilde{W} + (W_t-[\alpha_4\alpha_3]\alpha_3^*\alpha_4^*) +\sum_{1\leq i\leq 4} \Tilde{c}_i. \]

Since the quiver $\mu_{s+1} Q(\mathfrak{i})$ does not contain 2-cycles, it coincides with the classical mutation at $s+1$ of the quiver $Q(\mathfrak{i})$. Therefore, by \cite[Lem.~2.7]{HFOO_iso_quant_groth_ring_clust_alg}, it is isomorphic to the quiver $Q(\mathfrak{\jmath})$, via the index change $s\leftrightarrow s-1$. A direct check shows that this isomorphism, up to a change of arrows to adjust the signs (see \cite[Def.~2.5]{DWZ_quiv_pot_rep_I}), respects the associated potentials (they are indeed both given by the alternated sum of the minimal cycles appearing in the quiver).

\end{proof}

\begin{proof}[Proof of Proposition \ref{prop_rigid_quiver}
] If the reduced expression $\underline{w}_0$ is $\mathcal{Q}$-adpated, then, by \cite[Prop.~7.27]{KKOP_mon_cat_quant_aff_II}, the quiver   $Q(\underline{w}_0)$ is isomorphic to the quiver denoted by $\Gamma$ in \cite{HL_Clust_alg_approach_q_char} and the result follows by applying the same arguments of \cite[Prop.~4.17]{HL_Clust_alg_approach_q_char}. \\
If $\underline{w}_0$ is not $\mathcal{Q}$-adapted, following the above discussion, we can obtain it from an $\mathcal{Q}$-adapted reduced expression via a composition of braid moves and commuation moves. Since commutaion moves do not change the potential, it suffices to show that, if $\underline{w}_0=s_{\imath_1}\dots s_{\imath_{l_0}}$ is obtained by a braid move at $1< s<l_0$ from a reduced expression $\underline{w}'_0=s_{\jmath_1}\dots s_{\jmath_{l_0}}$ satisfying the statement of the proposition for any finite interval $[a,b]$, then it holds true also for $\underline{w}_0$.
Assume that $[a,b]$ is of the form $[-l_0k,l_0k]$, $k$ being a positive integer. By Lemma \ref{lem_mut_quiv_pot_preserva}, the quiver with potential $Q(\underline{w}_0)$ is obtained from the quiver with potential $Q(\underline{w}'_0)$ by mutating at the vertices $\{s+1+zl_0 | z\in \mathbb{Z}\}$ in any order.
Since by hypothesis the quiver with potential $(Q^{[-l_0k,l_0k]}(\underline{w}'_0),W^{[-l_0k,l_0k]}(\underline{w}'_0))$ is rigid and Jacobi finite, properties which are mutation-invariant for reduced quivers with potential (\cite[Cor.~6.6, Cor.~6.11]{DWZ_quiv_pot_rep_I}), the same is true also for $(Q^{[-l_0k,l_0k]}(\underline{w}_0),W^{[-l_0k,l_0k]}(\underline{w}_0))$. 
If $[a,b]$ is a general finite interval, the result follows by taking the quotient of the path algebra associated to an interval of the form $[-l_0k,l_0k]$ containing $[a,b]$ (see \cite[prop.~8.9]{DWZ_quiv_pot_rep_I}).
\end{proof}

By Proposition \ref{prop_rigid_quiver}, the quiver with potential $(Q^{[a,b]}(\underline{w}_0),W^{[a,b]}(\underline{w}_0))$ 
satisfies the hypothesis of Theorem \ref{teo_prop_(Q;W)_jacobi_finite}.

\begin{cor}
\label{cor_properties_clust_cat_Q[a,b]}
The cluster category associated to the quiver with potential 
\[
(Q^{[a,b]}(\underline{w}_0),W^{[a,b]}(\underline{w}_0))
\]
is $\mathrm{Hom}$-finite, Krull--Schmidt, 2-Calabi--Yau and it contains the canonical cluster-tilting object, whose associated quiver is $Q^{[a,b]}(\underline{w}_0)$. Moreover,  the mutation of any cluster-tilting object obtained by a sequence of mutation from the canonical one corresponds, at the level of its endoquiver, to the classical quiver mutation.
\end{cor}

\subsection{An algebra of global dimension $\leq 2$ associated to the initial seed}

We proceed now  to define an algebra of global dimension at most 2 such that the associated cluster category is triangle equivalent to the cluster category associated to the quiver with potential $(Q^{[a,b]}(\underline{w}_0),W^{[a,b]]}(\underline{w}_0))$.

For a given arrow $\alpha: a\rightarrow b$ of $Q(\underline{w}_0)$, where $\imath_a\neq \jmath_b$, we define its \emph{orbit} as the set or arrows $\{\beta: c\rightarrow d | \imath_c=\imath_a,\ \imath_d=\imath_b\}$.\\ 
 We denote by $\mathcal{I}$ the set of pairs $(\imath,m)$ such that the interval $[a,b]$ contains an $i$-box of index $\imath$ and $i$-cardinality $m$. Notice that the set $\mathcal{I}$ gives a common labelling of the vertices of all the quivers $Q(\mathfrak{C})$: if $\mathfrak{C}$ is a chain of $i$-box of range $[a,b]$ and $(\imath,m)$ is in $\mathcal{I}$, we label by $m_\imath$ the vertex of $Q(\mathfrak{C})$ associated to the $i$-box of $\mathfrak{C}$ of index $\imath$ and $i$-cardinality $m$.\\
In particular, in our graphical representation of the quiver $Q^{[a,b]}(\underline{w}_0)$ where lines are indexed by the elements of $I_{\check{\mathfrak{g}}}$, $m_\imath$ is the $m$-th vertex from the right on the line $\imath$. \\
For each pair of lines $\imath,\jmath$ such that the associated vertices are adjacent in the Dynkin diagram of $\check{\mathfrak{g}}$, remove from the quiver $Q^{[a,b]}(\underline{w}_0)$ the orbit of the second arrow from the right connecting the line $\imath$ and the line $\jmath$.
Therefore, up to exchanging the position of the lines, the resulting quiver, denoted $Q_A$, is formed by diagrams of the following types, where the dotted arrows stand for the removed arrows and blue vertices could potentially be missing, in such a way that the resulting diagram remains connected:
\begin{equation}
\label{mesh_Q_A_tipo1}
    \begin{tikzcd}[sep=tiny]
n_j\arrow[ddrrrr] &  (n-1)_{\jmath} \arrow[l] & \dots \arrow[l] &(\Tilde{n}+1)_\jmath \arrow[l] & & & & &\Tilde{n}_\jmath\arrow[lllll]\arrow[ddr] &  \\
&&&&&&&\ &&
\\
& & & & m_i\arrow[uurrrr, dotted] & (m-1)_\imath\arrow[l] & \dots \arrow[l] & (\Tilde{m}+1)_\imath \arrow[l]
& & \Tilde{m}_\imath \arrow[ll]
\end{tikzcd}
\end{equation}

\begin{equation}
\label{mesh_Q_A_tipo2}
    \begin{tikzcd}[sep=tiny]
\textcolor{blue}{n_\jmath} &  \textcolor{blue}{(n-1)_{\jmath}} \arrow[l] & \dots \arrow[l] &\textcolor{blue}{(\Tilde{n}+1)_\jmath} \arrow[l] & & & & &\Tilde{n}_\jmath\arrow[lllll]\arrow[ddr] &  \\
&&&&&&&\ &&
\\
& & & & m_\imath\arrow[uurrrr, dotted] & (m-1)_\imath\arrow[l] & \dots \arrow[l] & (\Tilde{m}+1)_\imath \arrow[l]
& & \Tilde{m}_\imath \arrow[ll]
\end{tikzcd}
\end{equation}

\begin{equation}
\label{mesh_Q_A_tipo3}
   \begin{tikzcd}[sep=tiny]
 & & & &\Tilde{n}_\jmath\arrow[ddr] &  \\
&&&\ &&
\\
 \textcolor{blue}{m_\imath} & \textcolor{blue}{(m-1)_\imath} \arrow[l] & \dots \arrow[l] & \textcolor{blue}{(\Tilde{m}+1)_\imath} \arrow[l]
& & \Tilde{m}_\imath \arrow[ll]
\end{tikzcd}
\end{equation}

\begin{equation}
\label{mesh_Q_A_tipo4}
    \begin{tikzcd}[sep=tiny]
\textcolor{blue}{n_j}\arrow[ddrrrr] &  \textcolor{blue}{(n-1)_{\jmath}} \arrow[l] & \dots \arrow[l] & \textcolor{blue}{1_\jmath} \arrow[l] & & & &  \\
&&&&&&&\ &&
\\
& & & & \textcolor{blue}{m_\imath} & \textcolor{blue}{(m-1)_\imath }\arrow[l] & \dots \arrow[l] & 1_\imath \arrow[l]
\end{tikzcd}
\end{equation}

Let $I_A$ be the ideal of the path algebra $kQ_A$ generated by the alternating sums of the paths from the vertex in the top right corner to the vertex in the bottom left corner of any diagram of type $(\ref{mesh_Q_A_tipo1})$ or $(\ref{mesh_Q_A_tipo2})$. We denote by $A$ the quotient $kQ_A/I_A$.
For any vertex $a$, we denote by $P_{a}$ the associated indecomposable projective $A$-module. For our convenience, for any $\imath \in \check{\mathfrak{g}}$, by $P_{0_\imath}$ we mean the zero module.

\begin{rem}\label{rem_WS_path}(About zero paths in the quiver with relations $(Q_A,I_A)$) In order to understand if a path is zero in the algebra $A$, replace each of its subpaths $\Tilde{n}_\jmath \rightarrow \Tilde{m}_\imath \leadsto m_\imath$ belonging to a mesh of type (\ref{mesh_Q_A_tipo1}) by the path $\Tilde{n}_\jmath \rightarrow n_\jmath \leadsto m_\imath$. We say that the resulting path is in the WS-form. 
Therefore, the path is different from 0 if and only if, in its WS-form, it does not contain any subpath $\Tilde{n}_\jmath \rightarrow \Tilde{m}_\imath \leadsto m_\imath$ belonging to a diagram of type (\ref{mesh_Q_A_tipo2}). Notice that the $WS$-form of a path is unique.
\end{rem}

\begin{prop}
  The indecomposable projective modules of the algebra $A$ are thin.
\end{prop}

\begin{proof}
It follows from the uniqueness of the $WS$-form of a path.

    By the characterization of the indecomposable projective modules, it suffices to show that, for any pair of vertices $m_\imath,m'_{\jmath}$, the space of paths from $m_\imath$ to $m'_\imath$ in $A$ is $1$-dimensional. 
    
    By construction of the quiver $Q_A$, if there is an arrow from a vertex of the line $\imath$ to a vertex of the line $\jmath$, $\imath\neq \jmath$, then there cannot be an arrow from a vertex of the line $\imath$ to a vertex of the line $\jmath$, $\imath\neq \jmath$ (they have all been removed). 
    In particular, when a path leaves a line, it cannot return to it.
    
    Therefore, if $\imath=\jmath$, the only non-zero path between two vertices $m'_\imath, m_\imath$ of the same line is
    \[m'_\imath \leftarrow (m-1)'_\imath \leftarrow \dots \leftarrow (m+1)_\imath \leftarrow m_\imath,\]
    if $m'> m$, and the idempotent $e_{m_{\imath}}$ if $m=m'$. 
    Suppose now that $\imath\neq \jmath$.

    Let $p=\alpha_N\dots\alpha_2\alpha_1$ and $p'=\alpha'_{N'}\dots\alpha'_2\alpha'_1$ ($\alpha_i,\alpha'_j\in (Q_A)_1$) be two paths from $m_\imath$ to $m'_\jmath$. Assume that they are in the $WS$-form. Let $i,i'$ be the smallest integers such that the subpaths $\alpha_N\dots \alpha_{i+1}$ and $\alpha_{N'}\dots \alpha_{i'+1}$ are equal (if they do not exist, take $i=N$ and $i'=N'$). Up to switching the notation, we can assume that $\alpha_i'$ connects two vertices on the same line $\Tilde{\imath}$, while $\alpha_i$ has its source on a different line $\Tilde{\jmath}$. Therefore, there exists an index $j'<i'$ such that $\alpha_{j'}$ goes from the line $\Tilde{\imath}$ to the line $\Tilde{\jmath}$. In particular, since the arrow $\alpha_{j'}$ is on the right with respect to the arrow $\alpha_{i}$, it must be the left descending arrow of a diagram of type (\ref{mesh_Q_A_tipo1}). This contradicts the hypothesis that $p'$ is in the $WS$-form. 
\end{proof}

To illustrate the proof of the next theorem, we give the following examples:

\begin{ex}
Consider the quiver with relations
    
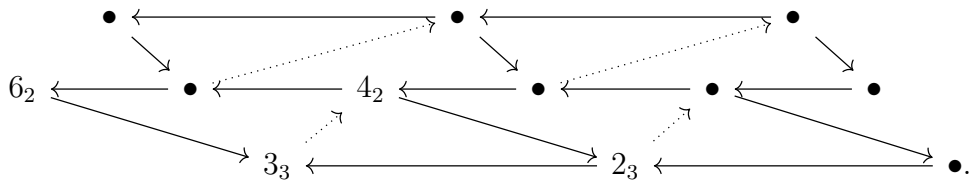
\begin{equation*}
\begin{tikzcd}[sep=small]
    & \bullet \arrow[rd] &&&& \bullet  \arrow[llll]\arrow[rd] &&&& \bullet \arrow[llll]\arrow[rd] & &  \\  
    6_2 \arrow[rrrd] && \bullet \arrow[ll]\arrow[rrru, dotted] && 4_2 \arrow[rrrd] \arrow[ll]&& \bullet \arrow[ll]\arrow[rrru, dotted]&& \bullet\arrow[rrrd]  \arrow[ll]&& \bullet \arrow[ll]&\\
    & & & 3_3 \arrow[ru, dotted] &&&& 2_3 \arrow[llll]\arrow[ru, dotted] &&&& \bullet \arrow[llll] .
\end{tikzcd}
\end{equation*}

We want to compute the projective resolution of the simple module $S_{3_3}$. 

In the following quiver, we represent in red the vertices supporting the module $P_{3_3}$.

    
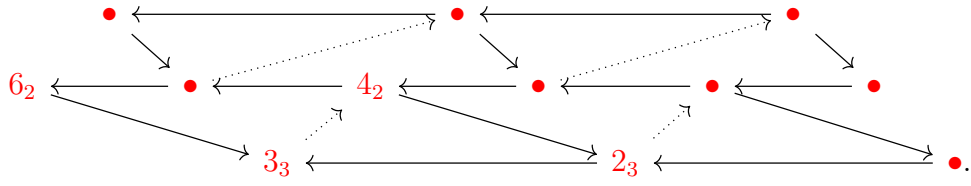
\begin{equation*}
\begin{tikzcd}[sep=small]
    & \textcolor{red}{\bullet} \arrow[rd] &&&& \textcolor{red}{\bullet}  \arrow[llll]\arrow[rd] &&&& \textcolor{red}{\bullet} \arrow[llll]\arrow[rd] & & \\  
    \textcolor{red}{6_2} \arrow[rrrd] && \textcolor{red}{\bullet} \arrow[ll]\arrow[rrru, dotted] && \textcolor{red}{4_2} \arrow[rrrd] \arrow[ll]&& \textcolor{red}{\bullet} \arrow[ll]\arrow[rrru, dotted]&& \textcolor{red}{\bullet}\arrow[rrrd]  \arrow[ll]&& \textcolor{red}{\bullet} \arrow[ll]&\\
    & & & \textcolor{red}{3_3} \arrow[ru, dotted] &&&& \textcolor{red}{2_3} \arrow[llll]\arrow[ru, dotted] &&&& \textcolor{red}{\bullet} \arrow[llll] .
\end{tikzcd}
\end{equation*}

In the following quiver, we represent in red the vertices supporting the module $P_{2_3}$:

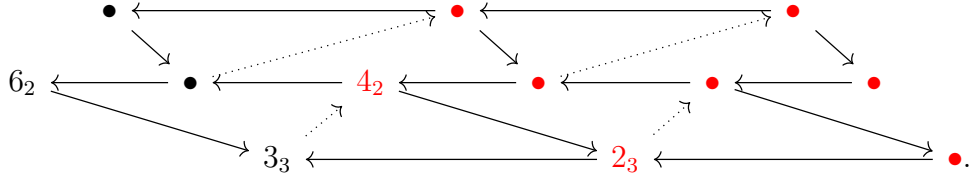
\begin{equation*}
\begin{tikzcd}[sep=small]
    & \bullet \arrow[rd] &&&& \textcolor{red}{\bullet}  \arrow[llll]\arrow[rd] &&&& \textcolor{red}{\bullet} \arrow[llll]\arrow[rd] & & \\  
    6_2 \arrow[rrrd] && \bullet \arrow[ll]\arrow[rrru, dotted] && \textcolor{red}{4_2} \arrow[rrrd] \arrow[ll]&& \textcolor{red}{\bullet} \arrow[ll]\arrow[rrru, dotted]&& \textcolor{red}{\bullet}\arrow[rrrd]  \arrow[ll]&& \textcolor{red}{\bullet} \arrow[ll]&\\
    & & & 3_3 \arrow[ru, dotted] &&&& \textcolor{red}{2_3} \arrow[llll]\arrow[ru, dotted] &&&& \textcolor{red}{\bullet} \arrow[llll] .
\end{tikzcd}
\end{equation*}

In the following quiver, we represent in red the vertices supporting the module $P_{6_2}$:

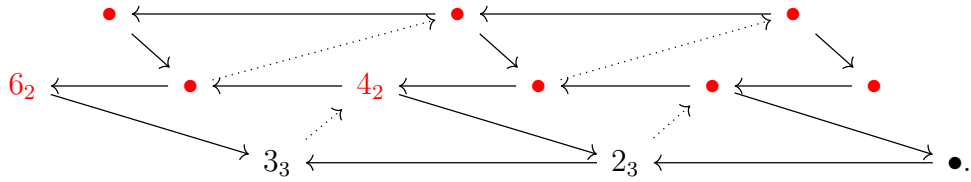
\begin{equation*}
\begin{tikzcd}[sep=small]
    & \textcolor{red}{\bullet} \arrow[rd] &&&& \textcolor{red}{\bullet}  \arrow[llll]\arrow[rd] &&&& \textcolor{red}{\bullet} \arrow[llll]\arrow[rd] & & \\  
    \textcolor{red}{6_2} \arrow[rrrd] && \textcolor{red}{\bullet} \arrow[ll]\arrow[rrru, dotted] && \textcolor{red}{4_2} \arrow[rrrd] \arrow[ll]&& \textcolor{red}{\bullet} \arrow[ll]\arrow[rrru, dotted]&& \textcolor{red}{\bullet}\arrow[rrrd]  \arrow[ll]&& \textcolor{red}{\bullet} \arrow[ll]&\\
    & & & 3_3 \arrow[ru, dotted] &&&& 2_3 \arrow[llll]\arrow[ru, dotted] &&&& \bullet \arrow[llll] .
\end{tikzcd}
\end{equation*}

In the following quiver, we represent in red the vertices supporting the module $P_{4_2}$:

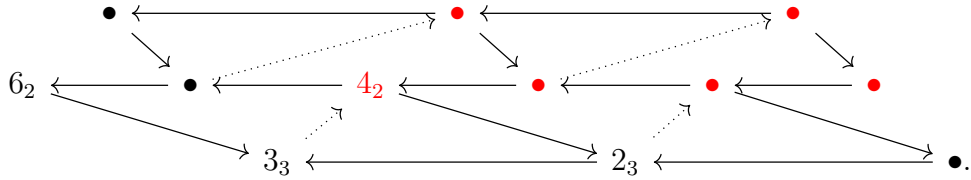
\begin{equation*}
\begin{tikzcd}[sep=small]
    & \bullet \arrow[rd] &&&& \textcolor{red}{\bullet}  \arrow[llll]\arrow[rd] &&&& \textcolor{red}{\bullet} \arrow[llll]\arrow[rd] & &  \\  
   6_2 \arrow[rrrd] && \bullet \arrow[ll]\arrow[rrru, dotted] && \textcolor{red}{4_2} \arrow[rrrd] \arrow[ll]&& \textcolor{red}{\bullet} \arrow[ll]\arrow[rrru, dotted]&& \textcolor{red}{\bullet}\arrow[rrrd]  \arrow[ll]&& \textcolor{red}{\bullet} \arrow[ll]&\\
    & & & 3_3 \arrow[ru, dotted] &&&& 2_3 \arrow[llll]\arrow[ru, dotted] &&&& \bullet \arrow[llll].
\end{tikzcd}
\end{equation*}

Therefore, we see that the sequence
\[0\rightarrow P_{4_2} \rightarrow  P_{6_2} \oplus P_{2_3} \rightarrow P_{3_3}\rightarrow S_{3_3} \]

is a minimal projective presentation of $S_{3_3}$,

\end{ex}

\begin{ex}
    Consider the quiver with relations

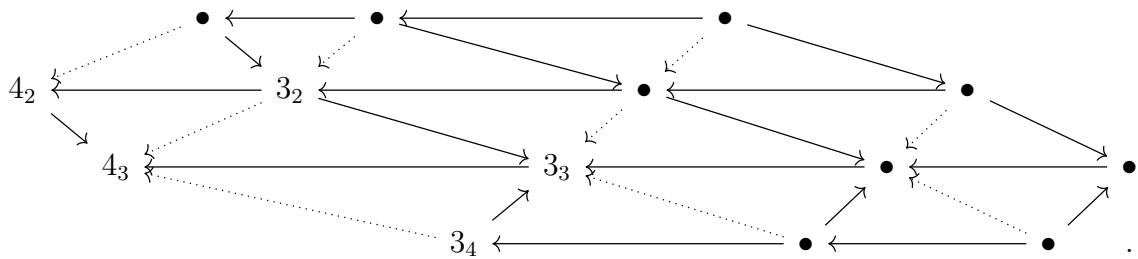
\begin{equation*}
    \begin{tikzcd}[sep=small]
         & & \bullet \arrow[dr] \arrow[dll, dotted]& & \bullet \arrow[ll] \arrow[drrr] \arrow[dl, dotted]& & & & \bullet \arrow[llll] \arrow[drrr] \arrow[dl, dotted] & & & & & \\
         4_2 \arrow[dr] & & & 3_2 \arrow[lll] \arrow[drrr] \arrow[dll, dotted] & & & & \bullet \arrow[llll] \arrow[drrr] \arrow[dl, dotted]& & & & \bullet \arrow[llll] \arrow[drr]  \arrow[dl, dotted] & & \\
         & 4_3 & & & & & 3_3 \arrow[lllll] & & & & \bullet \arrow[llll] & & &\bullet \arrow[lll] \\
         & & & & & 3_4 \arrow[ur] \arrow[ullll, dotted]& & & & \bullet \arrow[llll] \arrow[ur] \arrow[ulll, dotted] & & & \bullet \arrow[lll] \arrow[ur] \arrow[ull, dotted]&.
    \end{tikzcd}
\end{equation*}

We want to compute the projective resolution of the simple module $S_{4_3}$. 
The support of $P_{4_3}$ is

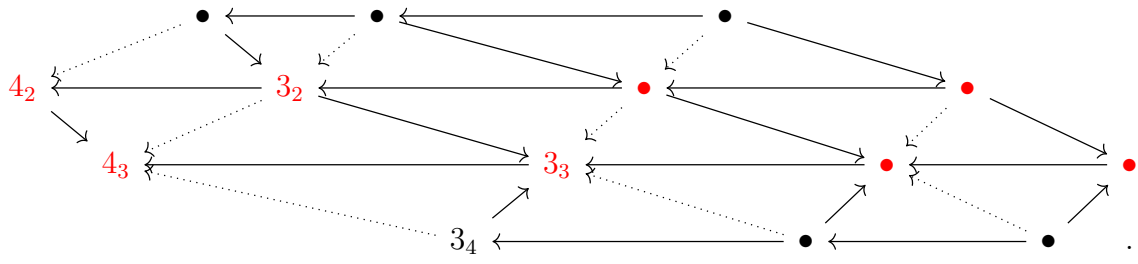
\begin{equation*}
    \begin{tikzcd}[sep=small]
         & & \bullet \arrow[dr] \arrow[dll, dotted]& & \bullet \arrow[ll] \arrow[drrr] \arrow[dl, dotted]& & & & \bullet \arrow[llll] \arrow[drrr] \arrow[dl, dotted] & & & & & \\
         \textcolor{red}{4_2} \arrow[dr] & & & \textcolor{red}{3_2} \arrow[lll] \arrow[drrr] \arrow[dll, dotted] & & & & \textcolor{red}{\bullet} \arrow[llll] \arrow[drrr] \arrow[dl, dotted]& & & & \textcolor{red}{\bullet} \arrow[llll] \arrow[drr]  \arrow[dl, dotted] & & \\
         & \textcolor{red}{4_3} & & & & & \textcolor{red}{3_3} \arrow[lllll] & & & & \textcolor{red}{\bullet} \arrow[llll] & & &\textcolor{red}{\bullet} \arrow[lll] \\
         & & & & & 3_4 \arrow[ur] \arrow[ullll, dotted]& & & & \bullet \arrow[llll] \arrow[ur] \arrow[ulll, dotted] & & & \bullet \arrow[lll] \arrow[ur] \arrow[ull, dotted]&.
    \end{tikzcd}
\end{equation*}

The support of $P_{3_3}$ is 

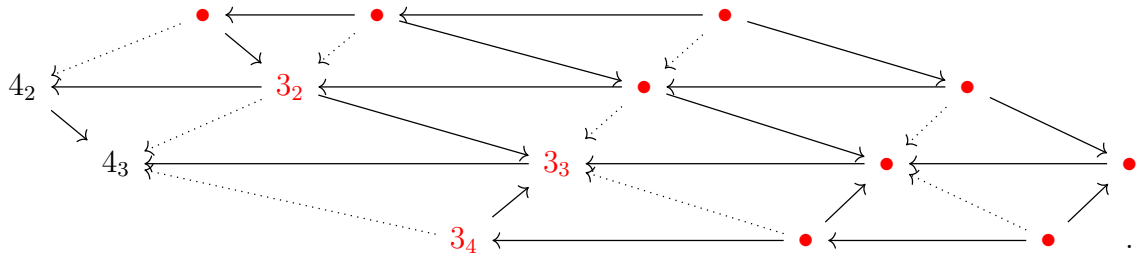
\begin{equation*}
    \begin{tikzcd}[sep=small]
         & & \textcolor{red}{\bullet} \arrow[dr] \arrow[dll, dotted]& & \textcolor{red}{\bullet} \arrow[ll] \arrow[drrr] \arrow[dl, dotted]& & & & \textcolor{red}{\bullet} \arrow[llll] \arrow[drrr] \arrow[dl, dotted] & & & & & \\
         4_2 \arrow[dr] & & & \textcolor{red}{3_2} \arrow[lll] \arrow[drrr] \arrow[dll, dotted] & & & & \textcolor{red}{\bullet} \arrow[llll] \arrow[drrr] \arrow[dl, dotted]& & & & \textcolor{red}{\bullet} \arrow[llll] \arrow[drr]  \arrow[dl, dotted] & & \\
         & 4_3 & & & & & \textcolor{red}{3_3} \arrow[lllll] & & & & \textcolor{red}{\bullet} \arrow[llll] & & &\textcolor{red}{\bullet} \arrow[lll] \\
         & & & & & \textcolor{red}{3_4} \arrow[ur] \arrow[ullll, dotted]& & & & \textcolor{red}{\bullet} \arrow[llll] \arrow[ur] \arrow[ulll, dotted] & & & \textcolor{red}{\bullet} \arrow[lll] \arrow[ur] \arrow[ull, dotted]&.
    \end{tikzcd}
\end{equation*}

The support of $P_{4_2}$ is

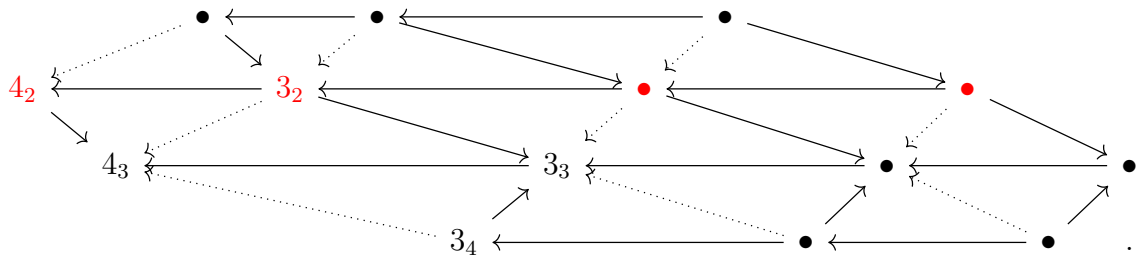
\begin{equation*}
    \begin{tikzcd}[sep=small]
         & & \bullet \arrow[dr] \arrow[dll, dotted]& & \bullet \arrow[ll] \arrow[drrr] \arrow[dl, dotted]& & & & \bullet \arrow[llll] \arrow[drrr] \arrow[dl, dotted] & & & & & \\
         \textcolor{red}{4_2} \arrow[dr] & & & \textcolor{red}{3_2} \arrow[lll] \arrow[drrr] \arrow[dll, dotted] & & & & \textcolor{red}{\bullet} \arrow[llll] \arrow[drrr] \arrow[dl, dotted]& & & & \textcolor{red}{\bullet} \arrow[llll] \arrow[drr]  \arrow[dl, dotted] & & \\
         & 4_3 & & & & & 3_3 \arrow[lllll] & & & & \bullet \arrow[llll] & & &\bullet \arrow[lll] \\
         & & & & & 3_4 \arrow[ur] \arrow[ullll, dotted]& & & & \bullet \arrow[llll] \arrow[ur] \arrow[ulll, dotted] & & & \bullet \arrow[lll] \arrow[ur] \arrow[ull, dotted]&.
    \end{tikzcd}
\end{equation*}

The support of $P_{3_2}$ is 

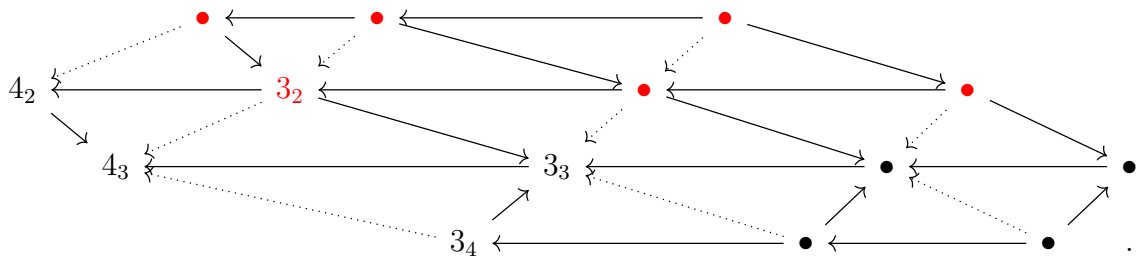
\begin{equation*}
    \begin{tikzcd}[sep=small]
         & & \textcolor{red}{\bullet} \arrow[dr] \arrow[dll, dotted]& & \textcolor{red}{\bullet} \arrow[ll] \arrow[drrr] \arrow[dl, dotted]& & & & \textcolor{red}{\bullet} \arrow[llll] \arrow[drrr] \arrow[dl, dotted] & & & & & \\
         4_2 \arrow[dr] & & & \textcolor{red}{3_2} \arrow[lll] \arrow[drrr] \arrow[dll, dotted] & & & & \textcolor{red}{\bullet} \arrow[llll] \arrow[drrr] \arrow[dl, dotted]& & & & \textcolor{red}{\bullet} \arrow[llll] \arrow[drr]  \arrow[dl, dotted] & & \\
         & 4_3 & & & & & 3_3 \arrow[lllll] & & & & \bullet \arrow[llll] & & &\bullet \arrow[lll] \\
         & & & & & 3_4 \arrow[ur] \arrow[ullll, dotted]& & & & \bullet \arrow[llll] \arrow[ur] \arrow[ulll, dotted] & & & \bullet \arrow[lll] \arrow[ur] \arrow[ull, dotted]&.
    \end{tikzcd}
\end{equation*}

The support of $P_{3_4}$ is 

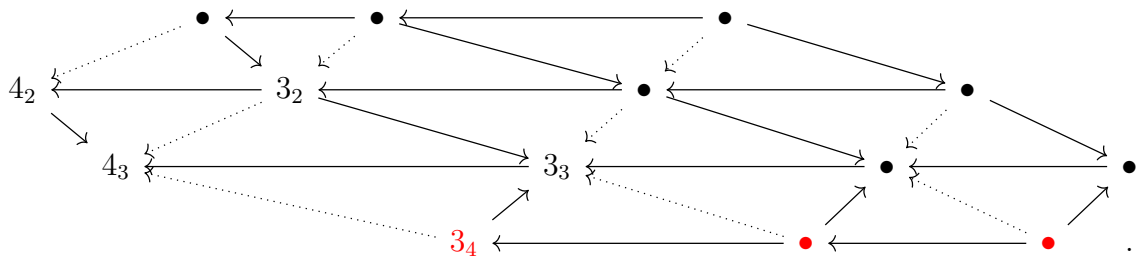
\begin{equation*}
    \begin{tikzcd}[sep=small]
         & & \bullet \arrow[dr] \arrow[dll, dotted]& & \bullet \arrow[ll] \arrow[drrr] \arrow[dl, dotted]& & & & \bullet \arrow[llll] \arrow[drrr] \arrow[dl, dotted] & & & & & \\
         4_2 \arrow[dr] & & & 3_2 \arrow[lll] \arrow[drrr] \arrow[dll, dotted] & & & & \bullet \arrow[llll] \arrow[drrr] \arrow[dl, dotted]& & & & \bullet \arrow[llll] \arrow[drr]  \arrow[dl, dotted] & & \\
         & 4_3 & & & & & 3_3 \arrow[lllll] & & & & \bullet \arrow[llll] & & &\bullet \arrow[lll] \\
         & & & & & \textcolor{red}{3_4} \arrow[ur] \arrow[ullll, dotted]& & & & \textcolor{red}{\bullet} \arrow[llll] \arrow[ur] \arrow[ulll, dotted] & & & \textcolor{red}{\bullet} \arrow[lll] \arrow[ur] \arrow[ull, dotted]&.
    \end{tikzcd}
\end{equation*}

Therefore, we see that the sequence
\[0\rightarrow P_{3_2}\oplus P_{3_4} \rightarrow  P_{4_2} \oplus P_{3_3} \rightarrow P_{4_3}\rightarrow S_{3_3} \]

is a minimal projective presentation of $S_{4_3}$,
    
\end{ex}

\begin{prop}
\label{prop_glob_dim2}
The algebra $A$ is of global dimension at most 2.
\end{prop}

\begin{proof}
    It suffices to show that the projective dimension of each simple module is less or equal to 2. Let $S_{m_\imath}$ be the simple modules associated to the vertex $m_\imath$.\\
    We write $J_1(m_\imath),\ J_2(m_\imath),\ J_3(m_\imath)$ and $J_4(m_\imath)$ for the sets of line indices $\jmath$ such that the vertex $m_\imath$ is in the bottom-left corner of a diagram of type (\ref{mesh_Q_A_tipo1}), (\ref{mesh_Q_A_tipo2}), (\ref{mesh_Q_A_tipo3}) or (\ref{mesh_Q_A_tipo4}) respectively, where the upper vertices belong to the line indexed by $\jmath$. For any $\jmath\in J_s(m_\imath)\ (s=1,2,3,4)$, if it exists, we write  $n_\jmath$ (resp. $\Tilde{n}_\jmath$) for the vertex in the top left (resp. top right) corner of the associated diagram.\\
    We claim that a minimal projective presentation of $S_{m_\imath}$ is given by
    \begin{equation}
    \label{eq_ses_proj_res_simples}
    0\rightarrow \bigoplus_{\jmath\ \in\  J_1(m_\imath)\sqcup J_2(m_\imath)} P_{\Tilde{n}_\jmath}\xrightarrow{f} P_{(m-1)_{\imath}}\bigoplus_{\jmath\ \in\  J_1(m_\imath)\sqcup J_4(m_\imath)} P_{n_\jmath}\xrightarrow{g} P_{m_\imath}\rightarrow S_{m_\imath},
    \end{equation}
    where the components of $f$ and $g$ are maps between projective modules induced by paths.\\
    We introduce also, with the denotation $K(m_\imath)$, the set of line indices $k$ such that the vertex $m_i$ is one of the internal vertices of the bottom row of a diagram of type (\ref{mesh_Q_A_tipo1}), (\ref{mesh_Q_A_tipo2}) or (\ref{mesh_Q_A_tipo3}), where the upper vertices belong to the line indexed by $k$.
    For each $k\in K(m_i)$, we denote by $\hat{n}_k$ the vertex in the top-right corner of the corresponding diagram.
    
 Considering the paths different from 0 in the quiver with relation and using Remark \ref{rem_WS_path}, we can give the following descriptions , in terms of disjoint unions, of the supports of the modules $P_{m_i}$ and $P_{(m-1)_i}$.

\begin{align*}   
    \text{Supp}(P_{m_\imath})=&A_0\sqcup (\bigsqcup_{\jmath\in J_1(m_\imath)} A^\jmath_1) \sqcup (\bigsqcup_{\jmath\in J_3(m_\imath)} A^\jmath_3 )\sqcup (\bigsqcup_{\jmath\in J_4(m_\imath)}A^\jmath_4)\sqcup (\bigsqcup_{k\in K(m_\imath)} A^k) , \text{ where } \\
    A_0=&\{m_\imath,(m-1)_\imath,\dots, 1_\imath\},\\
    A^\jmath_1 = & \text{Supp}(P_{n_\jmath}),\\
    A^\jmath_3 = & \text{Supp}(P_{\Tilde{n}_\jmath}),\\
    A^\jmath_4 =& \text{Supp}(P_{n_\jmath}),\\
    A^k = & \text{Supp}(P_{\hat{n}_k});\\
    \\ 
    \text{Supp}(P_{m_{(\imath-1)}})=& B_0\sqcup (\bigsqcup_{\jmath\in J_1(m_\imath)} B^j_1) \sqcup (\bigsqcup_{\jmath\in J_2(m_\imath)} B^\jmath_2)\sqcup (\bigsqcup_{\jmath\in J_3(m_\imath)} A^\jmath_3 )\sqcup (\bigsqcup_{k\in K(m_\imath)} A^k)  \text{ where} \\
    B_0=&\{m_{\imath-1},\dots, 1_\imath\}\\
    B^\jmath_1 = &\text{Supp}(P_{\Tilde{n}_\jmath})\\
    B^\jmath_2 = &\text{Supp}(P_{\Tilde{n}_\jmath})\\
\end{align*}

The exactness of (\ref{eq_ses_proj_res_simples}) is now a direct verification.

\end{proof}

\begin{rem}
The choice of the arrows removed from the quiver $Q(\underline{w}_0)$  is not trivial. In particular, removing the first arrow from the left between two lines can produce an algebra of projective dimension greater than 2.

For example, consider the quiver

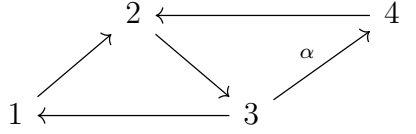
\begin{equation*}
    \begin{tikzcd}
        & 2 \arrow[dr] & & 4 \arrow[ll] \\
        1 \arrow[ur] & & 3 \arrow[ll] \arrow[ur, "\alpha"] & \ \ \ \ \ .
    \end{tikzcd}
\end{equation*}

Removing the orbit of the arrow $\alpha$, we obtain the quiver with relations:

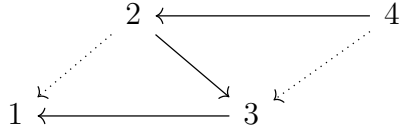
\begin{equation*}
    \begin{tikzcd}
        & 2 \arrow[dr] \arrow[dl, dotted]& & 4 \arrow[ll]\arrow[dl, dotted] \\
        1  & & 3 \arrow[ll] & \ \ \ \ \ .
    \end{tikzcd}
\end{equation*}

It can be easily showed that the global dimension of the associated path algebra is equal to 3. 
In particular, a minimal projective presentation of the simple module associated to the vertex 4 is

\[
0\rightarrow P_4 \rightarrow P_3 \rightarrow P_2 \rightarrow P_1 \rightarrow S_1.
\]

\end{rem}
Proposition \ref{prop_glob_dim2} allows us to define the generalized cluster category $\mathcal{C}_A$ of the algebra $A$.\\

\begin{prop}
The generalized cluster category $\mathcal{C}_A$ is  triangle equivalent to the cluster category associated to the quiver with potential $(Q^{[a,b]}(\underline{w}_0), W^{[a,b]}(\underline{w}_0))$. In particular, it is $\mathrm{Hom}$-finite, Krull--Schmidt, $2$-Calabi--Yau, the image of $A$ in $\mathcal{C}_A$ is the canonical cluster-tilting object whose associated quiver is $Q^{[a,b]}(\underline{w}_0)$ and the mutation of any cluster-tilting object obtained by a sequence of mutations from the canonical one corresponds, at the level of its endoquiver, to the classical quiver mutation.
\end{prop}

\begin{proof}
We want to apply Theorem \ref{teo_link_cluster_categories}. Let $S_A$ be the ideal generated by the arrows of $A$. Let $m$ and $m'$ be two vertices of $Q_A$, If $m$ and $m'$ are not, respectively,  the vertex $m_\imath$ in the bottom-left corner and the vertex $\tilde{n}_\jmath$ in the top-right corner of a diagram of type (\ref{mesh_Q_A_tipo1}) or (\ref{mesh_Q_A_tipo2}), then we have that $e_m I_A e_{m'}=e_m( I_AS_A+S_AI_A)e_{m'}$. Otherwise, let $\rho^*$ be the alternating sums of the paths from the vertex $\tilde{n}_\jmath$ to the vertex $m_\imath$ in the associated diagram. Then we have $e_{m_\imath}I_A e_{\tilde{n}_\jmath}=k\rho^*$ and $e_{m_\imath}( I_AS_A+S_AI_A)e_{\tilde{n}_\jmath}=0$. In particular, following the construction above Theorem \ref{teo_link_cluster_categories}, $\rho^*$ is a minimal relation and the quiver with potential associated to $A$ is exactly $(Q^{[a,b]}(\underline{w}_0), W^{[a,b]}(\underline{w}_0))$. Therefore, by Theorem \ref{teo_link_cluster_categories}, The generalized cluster category $\mathcal{C}_A$ is  triangle equivalent to the cluster category associated to the quiver with potential $(Q^{[a,b]}(\underline{w}_0), W^{[a,b]}(\underline{w}_0))$. The other properties in the statement of the Proposition follow from Corollary \ref{cor_properties_clust_cat_Q[a,b]}.
\end{proof}

\subsection{An additive avatar of affine determinant modules}

For $(\imath,k)\in \mathcal{I}$, let $[c,b)_\imath$ be the $i$-box of index $\imath$ and $i$-cardinality $k$ of the chain $\mathfrak{C}_-^{[a,b]}$. Recall that, if $m$ is a non-negative integer such that $(\imath,m+k)\in \mathcal{I}$ (equivalently, such that in the line $\imath$ of the quiver $Q_A$ there are at least $m+k$ vertices),
then to the $i$-box $\tau^{-m}[c,b)_\imath$ is associated the affine determinant module of Definition \ref{def_aff_det_mod}. We define its additive analogue as the $AD$-\emph{module} $AD(\imath,k,m)$ obtained as the quotient

\[ P_{m_\imath} \rightarrow P_{ (m+k)_\imath} \rightarrow AD(\imath,k,m) \rightarrow 0.\]
For any $(\imath,m)\in\mathcal{I}$, with the notation $AD(\imath,0,m)$ we mean the zero $A$-module.

When the reduced expression $\underline{w}_0$ is $\mathcal{Q}$-adapted, the module $AD(\imath,k,m)$ should be thought of as an additive analogue of a Kirillov--Reshetikhin module
over the quantum affine algebra $U_q(\mathfrak{g})$ with $k$ corresponding to the (total) degree
and $m$ to the spectral parameter.

We make a few observations about a module of the form $AD(\imath,k,m)$:
\begin{itemize}
    \item it is thin and its support $AD(\imath,k,m)$ consists of those vertices $n_\jmath$ such that in the quiver with relations $Q_A$ there is a non-zero path from $n_\jmath$ to $(m+k)_\imath$ and there are no non-zero paths from $n_\jmath$ to $m_\imath$;
    \item it is indecomposable (since $e_{(m+k)_\imath}$ generates $AD(\imath,k,m)e_{(m+k)_\imath}$ as a vector space and $AD(\imath,k,m)$ as a module);
    \item if $m=0$, then it is the projective module associated to the vertex $k_\imath$;
\end{itemize}

We say that a vertex $m_\imath$ of $Q_A$ satisfies the \emph{unbroken border} condition  if 
\begin{itemize}
   \item[(UB)]  there does not exist a vertex $m_{\imath'}'$ such that  there is a path from 
   $m_{\imath'}'$ to $m_\imath$\\ belonging to the ideal generated by the relations associated 
   to the diagrams of type (\ref{mesh_Q_A_tipo2}). 
\end{itemize}
we use the same terminology for the projective module $P_{m_{\imath}}$ and the $AD$-module $AD(\imath,k,m)$.

Thanks to Theorem \ref{teo_CY_reduction}, in our application we will often be able to assume that the vertices that we work with satisfy the (UB) condition.

\begin{prop}
\label{exchange_triangles_ses}
Let $\imath$ be in $ I_{\check{\mathfrak{g}}}$ and let $m,k$ be positive integers such that $(\imath,m+k)$ is in $\mathcal{I}$ and it satisfies the $\mathrm{(UB)}$ condition. Then, in the category of $A$-modules, the sequences
\[ 0\rightarrow AD(\imath,k,m)\xrightarrow{(\iota_1,\pi_1)} AD(\imath,k-1,m+1)\oplus AD(\imath,k+1,m)\xrightarrow {(\pi_2,\iota_2)^t} AD(\imath,k,m+1)\rightarrow 0\]
and 
\[ 0 \rightarrow P_{m_\imath}\xrightarrow{\iota_3} P_{(m+k)_\imath} \xrightarrow{\pi_3} AD(\imath,k,m) \rightarrow 0, \] 
where the maps $\iota_i$ and $\pi_i$ are natural inclusions and natural projections respectively, are exact. 
\end{prop}

\begin{proof}
A direct verification shows that $(\iota_1,\pi_1)$ is injective and that the support of its cokernel is 
\[\text{Supp}(AD(\imath,k-1,m+1)\cup \text{Supp}(AD(\imath,1,m+k) = \text{Supp}(AD(\imath,k,m+1)),\]
and that each arrow between two vertices of the support correspond to the identity map. The assumptions ensures that the relations do not play a role.
The exactness of the second sequence is verified similarly.
\end{proof}

\begin{cor}
\label{exchange_triangles}
   Let $\imath$ be in $ I_{\check{\mathfrak{g}}}$ and let $m,k$ be positive integers such that $(\imath,m+k)\in \mathcal{I}$. In the generalized cluster category of the algebra $A$, there are distinguished triangles:
    \[ AD(\imath,k,m)\xrightarrow{} AD(\imath,k-1,m+1)\oplus AD(\imath,k+1,m)\xrightarrow{} AD(\imath,k,m+1)\xrightarrow{\Delta}\]
    and
    \[ \Sigma^{-1} AD(\imath,k,m) \rightarrow P_{m_\imath}\rightarrow P_{(m+k)_\imath} \rightarrow AD(\imath,k,m). \] 
\end{cor}

Our next aim is to define, for each chain of $i$-boxes, a cluster-tilting object of the cluster category $\mathcal{C}_A$, image of a tilting module of the category of finite-dimensional $A$-modules.

\begin{lem}
\label{lem_ext_zero}
Let $R=kQ'/I$, where $Q'$ is a finite quiver and $I$ is an admissible ideal. Suppose that $N$ and $M$ are two $R$-modules such that
\begin{enumerate}
    \item M admits a minimal projective resolution
    \[ 0 \rightarrow e_bR\xrightarrow{} e_aR \xrightarrow{} M\rightarrow 0,\]
    where $a$ and $b$ are two vertices of $Q'$.
    \item The dimension of $Ne_b$ is at most $1$.
    \item There exists an element $\alpha\in R$ which generates $e_aRe_b$ as a vector space.
    \end{enumerate}
Then
\[\mathrm{Ext}^1_R(M,N) = 0 \ \text{  if and only if  } \ Ne_b=0 \ \text{  or  }\ N(\alpha) \neq 0.\]
\end{lem}

\begin{proof}
    Considering the natural isomorphism $\mathrm{Hom}_R(e_bR,e_aR) \cong e_aRe_b$, we have that, by hypothesis \emph{(3)}, up to a constant multiple, the only non-zero map $\bar{\alpha}: e_bR\rightarrow e_aR$ is the post-composition of paths with $\alpha$.
    Appling the contravariant functor $\mathrm{Hom}_R(-,N)$ to the given minimal projective resolution of $M$, we get the exact sequence
    \[ \mathrm{Hom}_R(e_aR,N) \xrightarrow{\circ\bar{\alpha}} \mathrm{Hom}_R(e_bR,N) \rightarrow \mathrm{Ext}^1_R(M,N) \rightarrow 0,\]
    from which we deduce that
    \[ \mathrm{dim}\ \mathrm{Ext}_R^1(M,N)=\mathrm{dim}\ \mathrm{Hom}_R(e_bR,N)-\mathrm{dim}\ \mathrm{Im}(\circ\bar{\alpha}).\]
    Using the commutative square

    \[
\begin{tikzcd}
  \mathrm{Hom}_R(e_aR,N) \arrow[r, "\circ\bar{\alpha}"]\arrow[d, phantom, sloped, "\simeq"]  &   \mathrm{Hom}_R(e_bR,N) \arrow[d, phantom, sloped, "\simeq"] \\
Ne_a\arrow[r, "N(\alpha)"]   &   Ne_b,
\end{tikzcd}
\]
we obtain also the equality 
 \[ \mathrm{dim}\ \mathrm{Ext}_R^1(M,N)=\mathrm{dim}\  Ne_b-\mathrm{dim}\ \mathrm{Im}(N(\alpha)).\]
The result follows by taking into account hypothesis \emph{(2)}.

\end{proof}

\begin{prop}
\label{prop_Ext_KR_vanishing}
    Let $AD(\imath,k,m)$ and $AD(\imath,k',m')$ two $AD$-modules such that $(m+k-1)_\imath$ satisfies the $\mathrm{(UB)}$ condition. Then  
    \[\mathrm{Ext}_A^1(AD(\imath,k,m), AD(\jmath,k',m'))  = 0\]    if and only if   \[AD(\jmath,k',m')e_{m_\imath} = 0 \ \text{  or  }\ AD(\jmath,k',m')e_{(m+k)_\imath} \neq 0.\]
\end{prop}

\begin{proof}
Let $\alpha$ be the unique non-zero path in $Q_A$ from $e_{m_\imath}$ to $e_{(m+k)_\imath}$
By the construction of $AD$-modules, if $AD(\jmath,k',m')e_{m_\imath} \neq 0$, then $AD(\jmath,k',m')e_{(m+k)_\imath} = 0 $ if and only if $AD(\jmath,k',m')(\alpha)=0$. The result then is an application of Lemma \ref{lem_ext_zero}.
\end{proof}

\subsection{An additive avatar of monoidal seeds}

We assign to each $i$-box contained in the interval $[a,b]$ an $AD$-module as follows.
If $\mathfrak{c}\subset [a,b]$ is an $i$-box with index $\imath$ and $i$-cardinality $k$, let $\mathfrak{c}^-$ be the $i$-box of $\mathfrak{C}^{[a,b]}_-$ with the same index and $i$-cardinality. Define 
\[T({\mathfrak{c}})=AD(\imath,k,m),\]
where $m$ is the positive integer such that 
\[ \mathfrak{c}=\tau^{-m} \mathfrak{c}^- .\]
For a chain of $i$-boxes $\mathfrak{C}$ with range $[a,b]$ we define the $A$-module $T(\mathfrak{C})$ by 

\[ T(\mathfrak{C})=\bigoplus_{\mathfrak{c}\in\mathfrak{C}} T(\mathfrak{c}). \]

Notice that

\[ T(\mathfrak{C}_{-}^{[a,b]})=\bigoplus_{(k,\imath)\in\mathcal{I}} AD(\imath,k,0) =  \bigoplus_{(k,\imath)\in\mathcal{I}} P_{k_\imath}= A, \]
that is, $T(\mathfrak{C}_{-}^{[a,b]})$ is the free $A$-module of rank 1.

\begin{teo}
\label{teo_T(C)_tilting}
Let $\mathfrak{C}$ be a chain of $i$-boxes on $[a,b]$. Then the module $T(\mathfrak{C})$ is tilting.
\end{teo}

Before proving the theorem, we state the following lemma, which gives some insight into the structure of the quiver $Q_A$.

For $\imath,\jmath \in I_{\check{\mathfrak{g}}}$ and $m_\imath \in (Q_A)_0$ a vertex on the line $\imath$, we denote by $n(\imath,\jmath,m)$ the number of vertices on the line $\jmath$ on the right of the vertex $m_\imath$ and by $\Tilde{n}(\imath,\jmath,m)$ the smallest integer such that there is a path from $m_\imath$ to $\Tilde{n}(\imath,\jmath,m)_\jmath$ (if it does not exist, we set $\Tilde{n}(\imath,\jmath,m)=0$). Notice that, for any $AD$-module $AD(\jmath,k',m')$ satisfying the (UB) condition, we have 
\[     AD(\jmath,k',m')e_{m_\imath}=0 \text{ if and only if } AD(\jmath,k',m')e_{\Tilde{n}(\imath,\jmath,m)_\jmath}=0.
\]
Keeping this notation, we have the following

 \begin{lem}
 \label{lem_structure_Q_A_2}
Assume that in $Q_A$ there is a path from the line $\imath$ to the line $\jmath$, that $n(\imath,\jmath,m)\neq 0$ and that $n(\imath,\jmath,m)_\jmath$ satisfies the (UB) condition. Then there is a path from $m_\imath$ to $n(\imath,\jmath,m)_\jmath$. In particular $n(\imath,\jmath,m) \geq \Tilde{n}(\imath,\jmath,m)$. 
 \end{lem}

\begin{proof}
We proceed by induction on the distance $d$ between the vertices $\imath$ and $\jmath$ of the Dynkin diagram of $I_{\check{\mathfrak{g}}}$. \\
If $d=1$, that is $\imath$ and $\jmath$ are adjacent, the result follows from the construction of the quiver $Q(\underline{w}_0)$, since $n(\imath,\jmath,m)_\jmath$ is the bottom left vertex of a mesh of type (\ref{mesh_Q_A_tipo1}) or (\ref{mesh_Q_A_tipo4}) where $m_\jmath$ is on the top line and is different from the top-right corner. Moreover, in this case we have $n(\imath,\jmath,m) = \Tilde{n}(\imath,\jmath,m)$.\\
If $d>1$, let $\jmath'$ be the index of the second last vertex reached by a path from $\imath$ to $\jmath$. Since enlarging the interval $[a,b]$ on the left does not change the positioning of the already existing vertices, we can assume that $n(\imath,\jmath,m)_\jmath$ is one of the vertices on the bottom line of a diagram of type (\ref{mesh_Q_A_tipo1}) or (\ref{mesh_Q_A_tipo4}), with $\jmath'$ being an index of the upper line. We discuss the case where $n(\imath,\jmath,m)_\jmath$ is one of the vertices on the bottom line of a diagram of type (\ref{mesh_Q_A_tipo1}), since the other case can be treated similarly. If $n(\imath,\jmath,m)_\jmath$ is the vertex in the bottom right corner and we denote by $s_{\jmath'}$ the vertex in the top-right corner, then $s\geq n(\imath,\jmath',m)$ and the results follows by the inductive hypothesis. Suppose now that $n(\imath,\jmath,m)_\jmath$ is any other vertex on the bottom line. If $s_\jmath \geq n(\imath,\jmath',m)$, the results follows again from induction. Otherwise, if $n(\imath,\jmath',m) > s_\jmath$, then $n(\imath,\jmath',m)_{\jmath'}$ is an other of the vertices of the top line and $n(\imath,\jmath,m)_\jmath$ is the vertex in the bottom left corner. Therefore, again by induction, there is a path from $m_\imath$ to $n(\imath,\jmath,m)_\jmath$.

\end{proof}

\begin{ex}
Let $Q_A$ be the quiver with relations

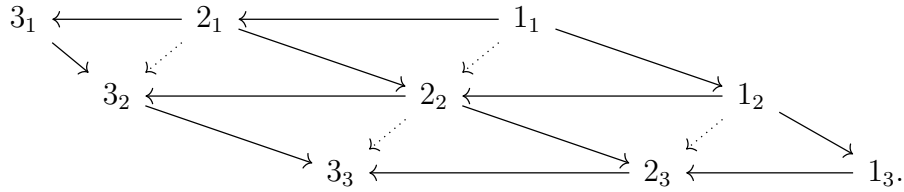
\begin{equation*}
    \begin{tikzcd}[sep=small]
        3_1 \arrow[dr] & & 2_1 \arrow[ll] \arrow[drrr] \arrow[dl, dotted]& & & & 1_1 \arrow[llll] \arrow[drrr] \arrow[dl, dotted] & & & & & \\
       & 3_2\arrow[drrr] & & & & 2_2 \arrow[llll] \arrow[drrr] \arrow[dl, dotted]& & & & 1_2\arrow[llll] \arrow[drr]  \arrow[dl, dotted] & & \\
        & & & & 3_3 & & & & 2_3 \arrow[llll] & & &1_3 \arrow[lll].
    \end{tikzcd}
\end{equation*}
Then we have the following:
\[
\begin{array}{ccc}
     n(1,2,2) = 2 & \text{and} & \Tilde{n}(1,2,2) = 2; \\
     n(1,3,2) = 3 & \text{and} & \Tilde{n}(1,3,2) = 2.
\end{array}
\]
\end{ex}

\begin{proof}[Proof of Theorem \ref{teo_T(C)_tilting}]
    We proceed by induction on the number $d$ of box moves  involved in the canonical chain transformation  $\mathfrak{C}_{-}^{[a,b]} \mapsto \mathfrak{C}.$
    If $d=0$, the object $T(\mathfrak{C})$ is the canonical tilting object of the module category of $A$.\\
    If $d>0$, assume that the results holds true up to the $(d-1)$st box move. Denote by $([c],(E_j)_{1\leq j\leq l-1})$ the rooted sequence of expansion operators associated to $\mathfrak{C}$ and by $\mathfrak{C'}$ the chain of $i$-boxes obtained after the first $(d-1)$ box moves. By assumption, the module $T(\mathfrak{C'})$
   associated to the chain $\mathfrak{C'}$ is tilting.
   If the $d$-th box move is not associated to a quiver mutation, then $T(\mathfrak{C'})\cong T(\mathfrak{C})$ and there is nothing to prove.
   Otherwise, the object $T(\mathfrak{C})$ is obtained by replacing a direct summand $AD(\imath,k,m)$ of $T(\mathfrak{C'})$ with $AD(\imath,k,m+1)$. Therefore, to show that $T(\mathfrak{C})$ is rigid, it suffices to prove that, for any direct summand $AD(\jmath,k',m')$ of $T(\mathfrak{C})$ different from $AD(\imath,k,m)$, we have
   \begin{align*}
       &(V1) \ \ \ \mathrm{Ext}_A^1(AD(\imath,k,m+1), AD(\jmath,k',m'))=0 \ \ \text{and} \\
   &(V2)\ \ \ \mathrm{Ext}_A^1(AD(\imath,k',m'), AD(\jmath,k,m+1))=0.
   \end{align*}
   We start with $(V1)$. By \ref{teo_CY_reduction}, we can assume that $AD(\jmath,k,m+1))$ satisfies the (UB) condition. Therfore, by Proposition \ref{prop_Ext_KR_vanishing}, we need to show that 

   \begin{equation}
   \label{eq_vanishing_Ext_KR_proof}
       AD(\jmath,k',m')e_{(m+1)_\imath} = 0 \ \text{  or  }\ AD(\jmath,k',m')e_{(m+k+1)_\imath} \neq 0.
    \end{equation}

   If $\imath=\jmath$, we can easily deduce the 2 following cases:
   \begin{enumerate}
       \item if $k'<k$, then $AD(\imath,k',m')e_{(m+1)_\imath} = 0$.
       \item Otherwise, $AD(\jmath,k',m')e_{(m+k+1)_\imath} \neq 0$.
   \end{enumerate}
   
   Suppose now that $\imath\neq \jmath$. We can assume that there is a path in $Q_A$ from the line $\imath$ to the line $\jmath$ (otherwise, we have that $AD(\jmath,k',m')e_{(m+1)_\imath} = 0$). 
   Since the canonical chain transformation  $\mathfrak{C}_{-}^{[a,b]} \mapsto \mathfrak{C}$ involves exactly $m+1$ flush right box moves of the $i$-box of index $\imath$ and $i$-cardinality $k$, we can deduce from Lemma \ref{lem_last_box_move_determine_other} that there are exactly $m+1$ $R$-operators associated to vertices of the line $\imath$. Moreover, If we write $s$ and $r$ for the integers
 \begin{align*}
 s&=\mathrm{min}\{ j\in [1,l-1]\ |\ E_j=R, i(E_j)=\imath\},\\
 r&=|\{j\in [s+1,l-1]\ |\ E_j=R\}|.
 \end{align*}
the, by Proposition \ref{prop_condition_m_flush_right}, we have that $r+s\geq i(m+k,\imath)-1$, where $\mathfrak{c}_{i(m+k,\imath)}^-$ is the $i$-box of $\mathfrak{C}_{-}^{[a,b]}$ of index $\imath$ and $i$-cardinality $m+k+1$ (in other words, in each of the basic chain transformation involving the line $\imath$, the associated $R$-operator is pushed far enough on the right to allow the flush right box move of the $i$-box of index $\imath$ and $i$-cardinality $k$).
 
Let $m'$ denote the non-negative integer
\[
m'=
\begin{cases}
   n(\imath,\jmath,m+1), & \text{if } k'\leq  n(\imath,\jmath,m+k+1)-n(\imath,\jmath,m+1),\\
   n(\imath,\jmath,m+k+1)-k',&\text{otherwise.}
\end{cases}
\]
In both cases, we have that $m'\leq n(\imath,\jmath,m+1)$ and $m'+k'\leq n(\imath,\jmath,m+k+1)$. From the first of these inequalities we deduce that the sequence $(E_j)_{1\leq j\leq l-1}$ contains at least $m'$ right extension operators of index $\jmath$, while, from the second, we deduce that there exists an integer $i(m'+k',\jmath)$ such that the $i$-box $\mathfrak{c}^-_{i(m'+k',\jmath)}$ had index $\jmath$ and $i$-cardinality $m'+k'$. Notice that, since $n(\imath,\jmath,m+k+1)\geq m'+k'$, then $i(m+k,\imath)-1\geq i(m'+k',\jmath)$. 
Write $s'$ for the largest integer such that 
\[|\{ j\in [s',l-1]\ |\ E_j=R, i(E_j)=\jmath\}|\geq m',\]
and let $r$ be the non-negative integer
\[r'=|\{j\in [s'+1,l-1]\ |\ E_j=R\}|.\]
Notice that 
\[r-r'=|\{j\in [s+1,s']\ |\ E_j=R\}|\leq s'-s,\]
from which we deduce
\[ s'+r'\geq s+r\geq i(m+k,\imath)-1\geq i(m'+k',\jmath).\]
 Therefore, by Proposition \ref{prop_condition_m_flush_right}, the canonical chain transformation $\mathfrak{C}_{-}^{[a,b]}\mapsto\mathfrak{C}$ entails at least \[
b(k')=\mathrm{min}(n(\imath,\jmath,m+1),n(\imath,\jmath,m+k+1)-k')\]
flush right shifts of the $i$-box of index $\jmath$ and $i$-cardinality $k'$.
We have the 2 following cases:
\begin{itemize}
    \item if $k'\leq n(\imath,\jmath,m+k+1)-n(\imath,\jmath,m+1)$, then $b(k')=n(i,j,m+1)$. Therefore, we have that $AD(\jmath,k',m')e_{n(i,j,m+1)_\jmath}= 0$, which implies $AD(\jmath,k',m')e_{\Tilde{n}(i,j,m+1)_\jmath}= 0$, which is equivakent to $AD(\jmath,k',m')e_{(m+1)_\imath}= 0$.
    \item Otherwise, if $k'>n(\imath,\jmath,m+k+1)-n(\imath,\jmath,m+1)$, then $m'\geq b(k')=n(\imath,\jmath,m+k+1)-k'$. Suppose that $AD(\jmath,k',m')e_{(m+1)_\imath} \neq 0$, which is equivalent to $AD(\jmath,k',m')e_{\Tilde{n}(\imath,\jmath,m+1)_\imath} \neq 0$. Therefore, since $AD(\jmath,k',m')e_{(m'+k')_\jmath} \neq 0$ and 
    \[m'+k'\geq k'+b(k')= n(\imath,\jmath,m+k+1) \geq \Tilde{n}(\imath,\jmath,m+k+1) \geq \Tilde{n}(\imath,\jmath,m+1),\] we deduce $AD(\jmath,k',m')e_{\Tilde{n}(\imath,\jmath,m+k+1)_\jmath} \neq 0$, which is equivalent to $AD(\jmath,k',m')e_{(m+k+1)_\imath}\neq 0$.\\
\end{itemize}

 Next, we consider $(V2)$. By definition of a canonical chain transformation, the chain $\mathfrak{C}'$ contains the $i$-box associated to the module $AD(\imath,k+1,m)$. Therefore, by inductive hypothesis, we have
    \[AD(\imath,k+1,m)e_{(m')_\jmath} = 0 \ \text{  or  }\ AD(\imath,k+1,m)e_{(m'+k')_\jmath} \neq 0.\]
    This directly implies that 
    \[AD(\imath,k,m+1)e_{(m')_\jmath} = 0 \ \text{  or  }\ AD(\imath,k,m+1)e_{(m'+k')_\jmath} \neq 0.\]

\end{proof}

\begin{cor}
Let $\mathfrak{C}$ be a chain of $i$-boxes on $[a,b]$. Then the object $T(\mathfrak{C})$ of $\mathcal{C}_A$ is cluster-tilting.
\end{cor}

\begin{proof}
    This is a particular case of Proposition \ref{prop_tilting_becomes_clust_tilt}.
\end{proof}

Recall that the set $\mathcal{I}$ of pairs $(\imath,k)$ such that the interval $[a,b]$ contains an $i$-box of index $\imath$ and $i$-cardinality $k$ forms a set of indices for the vertices of $Q_A$.

\begin{teo}
\label{teo_form_clust_tilt}
Let $\mathfrak{C}$ be a chain of $i$-boxes of range $[a,b]$.
Let $\mathcal{M}(\mathfrak{C})=\bigoplus_{(\imath,k)\in \mathcal{I}} \mathcal{M}_{k_\imath}\in \mathcal{C}_A$ be the cluster tilting object obtained from the canonical cluster tilting of $\mathcal{C}_A$ object through the mutation sequence associated to the canonical chain transformation 
$\mathfrak{C}^{[a,b]}_-\mapsto \mathfrak{C}$. \\
 For any $(\imath,k)$ in $\mathcal{I}$, let $\mathfrak{c}$ be the $i$-box of index $\imath$ and $i$-cardinality $k$ in $\mathfrak{C}$. Then we have
 \[ T(\mathfrak{c}) \cong \mathcal{M}_{m_\imath}. \]
 In particular, there is an isomorphism
\[ T(\mathfrak{C}) \cong \mathcal{M}({\mathfrak{C}}). \]
 
\end{teo}

\begin{proof}

Similarly to the proof of Theorem \ref{teo_T(C)_tilting}, We proceed by induction on the number $d$ of box moves  involved in the canonical chain transformation  $\mathfrak{C}_{-}^{[a,b]} \mapsto \mathfrak{C}.$
    If $d=0$, the object $\mathcal{M}(\mathfrak{C})$ is by definition $T(\mathfrak{C}^{[a,b]}_-)$, the canonical cluster-tilting object of the cluster category $\mathcal{C}_A$.\\
    If $d>0$, let $\mathfrak{C'}$ the chain of $i$-boxes obtained after the first $(d-1)$ box moves. If the $d$-th box move is not associated to a quiver mutation, then $M(\mathfrak{C})\cong M(\mathfrak{C'}) \cong T(\mathfrak{C'}) \cong T(\mathfrak{C})$.
    Otherwise, denoting by $k_\imath$ the vertex where the last mutations occurs and $\mathfrak{c'}$ the associated $i$-box, the object $M(\mathfrak{C})$ is obtained by replacing the direct summand $T(\mathfrak{c'})$ of $T(\mathfrak{C'})$ with a certain indecomposable object $\mathcal{M}_{k_\imath}$ of $\mathcal{C}_A$. 
    Let $m$ be the positive integer such that $T(\mathfrak{c'})=AD(\imath,k,m)$.
    Notice that,  by \cite[Prop.~7.15]{KKOP_mon_cat_quant_aff_II}, the full subquiver of the endoquiver of $T(\mathfrak{C'})$ with set vertices $k_\imath$ and those with an arrow pointing toward it, has the form

\begin{align*}
    &(k+1)_\imath\leftarrow k_{\imath}\rightarrow  (k-1)_{\imath}, &\text{ if $k>1$; }\\
    &\ \ \ \ \ \ \ \ \ \ \ \ 2_\imath \leftarrow 1_\imath, &\text{ if $k=1$}.
\end{align*}

In both cases, we can deduce from Lemma \ref{lem_exchange_triang_given_arrows} and Corollary \ref{exchange_triangles} that $\mathcal{M}_{k_\imath}$ is isomorphic to $AD(\imath,k,m+1)=T(\tau^{-1}\mathfrak{c'})$.
\end{proof}

Let $\mathfrak{C}=(\mathfrak{c}_k)_{1\leq k\leq l}$ be a chain of $i$-boxes with range $[a,b]$.
For any $(\imath,k)\in\mathcal{I}$, if $\mathfrak{c}$ is the $i$-box of $\mathfrak{C}$ associated to the vertex $k_\imath$ (that is, the $i$-box of $\mathfrak{C}$ with index $i$ and $i$-cardinality $k$), let $m(k_\imath)$ be the positive integer such that $T(\mathfrak{c})=AD(\imath,k,m(k_\imath))$.\\
By Corollary \ref{exchange_triangles}, the index of the obejct $AD(\imath,k,m(k_\imath))$ of $\mathcal{C}_A$ with respect to the canonical cluster tilting object of $\mathfrak{C}_{-}^{[a,b]}$ is 
\[
[P_{(m(k_\imath)+k)_\imath}] - [P_{m(k_\imath)_\imath}] \in K_0(\mathrm{add}(T(\mathfrak{C}_{-}^{[a,b]})).
\]

Therefore, to apply Palu's generalized mutation rule, (Theorem \ref{teo_Palu}), we define the matrix $P(\mathfrak{C})=(t_{k'_\imath,k_\jmath})$, indexed by the vertices of $Q^{[a,b]}(\widehat{\underline{w}}_0)$, by

\begin{equation}
    t_{k'_\imath,k_\jmath} =\begin{cases}
     1 & \text{if }k'_\imath=(k+m(k_\jmath))_\jmath \\
     -1 & \text{if } k'_\imath= m(k_\jmath)_\jmath \text{ and } m(k_\jmath)\neq 0,\\
     0 & \text{otherwise}.
    \end{cases}
\end{equation}

\begin{teo}[Solution of KKOP problem]
\label{teo_sol_KKOP}
The Exchange matrix $B(\mathfrak{C})$ is given by the following formula:
\[B(\mathfrak{C})=P(\mathfrak{C})^{-1}B^{[a,b]}(\underline{w}_0) P(\mathfrak{C})^{-t}.\]
\end{teo}

\begin{proof}
By Lemma \ref{lem_ord_mut} and Theorem \ref{teo_form_clust_tilt}, the endoquiver $Q(T(\mathfrak{C}))$ and the quiver $Q(\mathfrak{C})$ are both obtained from the quiver $Q(\mathfrak{C}^{[a,b]}_-)$ through the mutation sequence associated to the canonical chain transformation $\mathfrak{C}^{[a,b]}_-\mapsto \mathfrak{C}$. Therefore, it follows from Corollary \ref{cor_properties_clust_cat_Q[a,b]} that these two quivers coincide. Then a direct application of Palu's generalized mutation rule (\ref{teo_Palu}) provides a solution to the KKOP problem.
\end{proof}

\begin{rem}
\label{rem_g_vectors}
Following \cite{Dehy_Keller_combinatorics_2CY}, the indices of the  $AD$-modules correspond to the $g$-vectors of the associated cluster variables. When the reduced expression $\underline{w}_0$ is $\mathcal{Q}$-adapted, the $g$-vectors of the Kirillov--Reshetikhin modules have already been computed via other methods in \cite[Prop.~4.16]{HL_Clust_alg_approach_q_char}. See also \cite[\S 2]{HFOO_iso_quant_groth_ring_clust_alg} for other computations of $g$-vectors.
\end{rem}

The following Remark is a reformulation of \cite[Rem.~2.1]{Nakanishi_Zelev_tropical}.

\begin{rem}
Let $n$ be a positive integer. Let $B$ be a skew-symmetrizable $n\times n$-matrix and let $D$ be a \emph{skew-symmetrizer} of $B$, that is, a diagonal $n\times n$-matrix with strictly positive integer diagonal entries such that $DB=-B^TD$ (in the setting of this article $B$ is skew-symmetric and we can take $D$ as the identity matrix). Mutation equivalent matrices share the same skew-symmetrizers, cf. \cite[Prop.~4.5]{Fomin_Zelevinsky_clustalg_I}. For any $\theta=\pm 1$ and $1\leq k\leq n$, if we define the matrices $E_{B,k,\theta}=(e^k_{ij})_{1\leq i,j\leq n}$ and $F_{B,k,\theta}=(f^k_{ij})_{1\leq i,j\leq n}$ by

\[
e_{ij}^k=\begin{cases}
    \delta_{ij} & \text{if } j\neq k, \\
    -1 & \text{if } i=j=k,\\
    \mathrm{max}(0,-\theta b_{ik}) & \text{if } i\neq j=k,
\end{cases}\ \text{ and }\ f_{ij}^k=\begin{cases}
    \delta_{ij} & \text{if } j\neq k, \\
    -1 & \text{if } i=j=k,\\
    \mathrm{max}(0,\theta b_{ki}) & \text{if } i=k\neq j,
\end{cases}\
\]
we can express the mutation of the matrix $B$ at $k$ as 

\[\mu_k(B)=E_{B,k,\theta}^{-1}BF_{B,k,\theta}.\]
As pointed out in \cite[Lem.~5.5]{Keller_clusT_alg_der_cat}, 
\[F_{B,k,\theta}=D^{-1}(E_{B,k,\theta}^t)^{-1}D.\]
For a positive integer $e$ and  a sequence of indices $1\leq k_1,\dots,k_e\leq n$, let $B_s$ be the skew-symmetrizable matrix obtain from $B_0=B$ by mutating, in order, at the first $s$ indices, $1\leq s\leq e$. Then, for any $e$-tuple $(\theta_s)_{1\leq s\leq e}$ with $\theta_s=\pm 1$, we have

\[B_e = (E_{B_0,k_1,\theta_1}\dots E_{B_{e-1},k_e,\theta_e})^{-1} B D^{-1} (E_{B_0,k_1,\theta_1}\dots E_{B_{e-1},k_e,\theta_e})^{-t} D.\]
By \cite[Thm.~5.6]{Keller_clusT_alg_der_cat} (a reformulation of \cite[Prop.~1.3]{Nakanishi_Zelev_tropical}), if, for any $1\leq s\leq e$, we choose as $\theta_s$ the sign of the $c$-vector of index $k_s$ associated to $B_{s-1}$ (see \cite[Cor.~5.5]{Gross_Hacking_Keel_Kont_canonical_bases} for the sign-coherence of $c$-vectors), we have
\[ G_e= E_{B_0,k_1,\theta_1}\dots E_{B_{e-1},k_e,\theta_e},\]
where $G_e$ is the matrix of $g$-vectors associated to $B_s$ with respect to the initial matrix $B$. Therefore, in terms of the combinatorics of cluster algebras, Palu's generalized mutation rule becomes
\begin{equation}
\label{eq_ACFP}
    B_e = (G_e)^{-1}D^{-1}BG_e^{-t}D.
\end{equation} 
Relying on Remark \ref{rem_g_vectors}, equation (\ref{eq_ACFP}) provides an alternative way to adress Kashiwara--Kim--Oh--Park's problem.
\end{rem}

\begin{ex}
 Consider the pair ($\Delta=A_3,\text{Id}$), and define on it the height function $\varepsilon$ by $\varepsilon_1=0, \varepsilon_2=1, \varepsilon_3 =0$, obtaining a Q-datum for a simple Lie algebra of type $\text{A}_3$. It is associated to to the oriented quiver  $Q: \ \  \begin{tikzcd}[ampersand replacement=\&]
1 \&  2\arrow[l]\arrow[r]  \& 3 
\end{tikzcd}. $

We have that $\underline{w}_0=s_1s_3s_2s_1s_3s_2$ is an $\varepsilon$-adapted reduced expression, and the infinite sequence associated to it is
 \[ 
 \widehat{\underline{w_0}}=\dots\ 2,\ 1,\ 2,\ 3,\ \underbrace{\ 1,\ 3,\ 2,\ 3,\ 1,\ 2,\ 3,\ 1,\ 2\,}_{[-8,0]}, 1,\ 3\ \  
 \dots .\]

The chain $\mathfrak{C}^{[-8,0]}_{-}$ is formed by the $i$-boxes

 \begin{align*}
\mathfrak{c}_1&=[0]_2, &  \mathfrak{c}_2&=[-1]_1, &\mathfrak{c}_3&=[-2]_3, \\
\mathfrak{c}_4&=[-3,0]_2, &  \mathfrak{c}_5&=[-4,-1]_1, &\mathfrak{c}_6&=[-5,-2]_3,\\
\mathfrak{c}_7&=[-6,0]_2, &  \mathfrak{c}_8&=[-7,-2]_3, &\mathfrak{c}_9&=[-8,-1]_1.
\end{align*}  

The quiver associated to the monoidal seed $\mathcal{S}(\mathfrak{C}^{[-8,0]}_{-})$ is $Q^{[-8,0]}(\underline{w}_0)$:

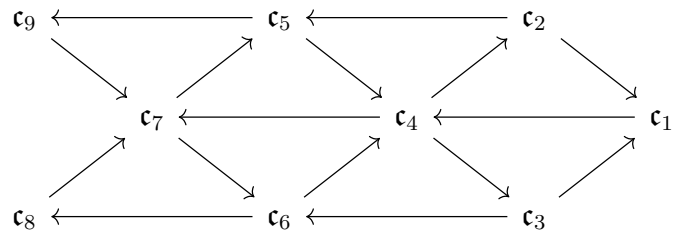
\[\begin{tikzcd}[ampersand replacement=\&]
\mathfrak{c}_9 \arrow[dr] \& \& \mathfrak{c}_5 \arrow[ll] \arrow[dr]  \& \& \mathfrak{c}_2\arrow[dr] \arrow[ll] \& \\
\& \mathfrak{c}_7 \arrow[dr] \arrow[ur] \& \& \mathfrak{c}_4\arrow[dr] \arrow[ur] \arrow[ll] \& \&\mathfrak{c}_1 \arrow[ll]\\
\mathfrak{c}_8 \arrow[ur] \& \& \mathfrak{c}_6 \arrow[ur] \arrow[ll] \& \& \mathfrak{c}_3\arrow[ur] \arrow[ll] \& \\
\end{tikzcd}
\].

An alternative chain of $i$-boxes $\mathfrak{C}'=(\mathfrak{c}_k)_{1\leq k\leq 9}$ for the interval $[-8,0]$ is given by
\begin{align*}
\mathfrak{c'}_1&=[-4]_1, & \mathfrak{c'}_2&=[-3]_2, & \mathfrak{c'}_3&=[-5]_3, \\
\mathfrak{c'}_4&=[-5,-2]_3, &\mathfrak{c'}_5&=[-6,-3]_2,
 &\mathfrak{c'}_6&=[-4,-1]_1,\\
\mathfrak{c'}_7&=[-7,-2]_3 &\mathfrak{c'}_8&=[-6,0]_2, &\mathfrak{c'}_9&=[-9,-1]_1.
\end{align*}

The algebra $k\vec{A_3}\otimes k\vec{\Delta}$ is given by the quiver with relations

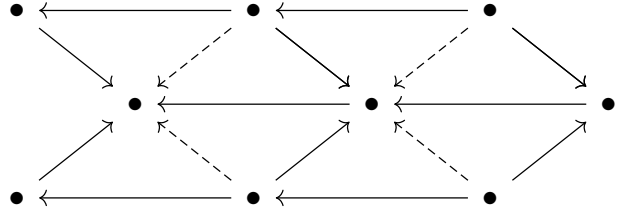
\[\begin{tikzcd}[ampersand replacement=\&]
\bullet \arrow[dr] \& \& \bullet\arrow[dl,dashed] \arrow[dr]\arrow[dr] \arrow[ll] \& \& \bullet\arrow[dl,dashed]\arrow[dr]\arrow[ll]\arrow[dr] \& \\
\& \bullet \& \& \bullet\arrow[ll] \& \&\bullet\arrow[ll]\\
\bullet \arrow[ur] \& \& \bullet\arrow[ul,dashed] \arrow[ur] \arrow[ll] \& \& \bullet\arrow[ul,dashed]\arrow[ur]\arrow[ll] \& \\
\end{tikzcd}.
\]
The $AD$-modules associated to intervals the first chain of $i$-boxes are 

 \begin{align*}
AD(2,1,0), & &AD(1,1,0), & &AD(3,1,0), \\
AD(2,2,0), & & AD(1,2,0), & &AD(3,2,0),\\
AD(2,3,0), & &AD(3,3,0), & & AD(1,3,0).
\end{align*}

The $AD$-modules associated to the intervals of the second chain of $i$-boxes are

\begin{align*}
AD(1,2,1),& & AD(2,2,1), && AD(3,2,1), \\
AD(3,2,1),& & AD(2,3,1), && AD(1,2,0),\\
AD(3,3,0)& & AD(2,3,0), && AD(1,3,0).
\end{align*}

By Palu's formula, we obtain  the quiver for the second seed, 

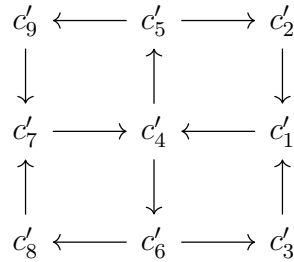
\[\begin{tikzcd}[ampersand replacement=\&]
c'_9\arrow[d] \& c'_5 \arrow[r] \arrow[l] \& c'_2\arrow[d]  \\
 c'_7 \arrow[r] \& c'_4\arrow[d]\arrow[u] \& c'_1\arrow[l]\\
c'_8 \arrow[u] \& c'_6 \arrow[r] \arrow[l] \& c'_3\arrow[u]  \\
\end{tikzcd}
\].

\end{ex}

\begin{ex}
 Consider the pair ($\Delta=A_3,\vee$), and define on it the height function $\varepsilon$ by $\varepsilon_1=1, \varepsilon_2=0, \varepsilon_3 =0$, obtaining a Q-datum for a simple Lie algebra of type $\text{B}_2$. We have that $\underline{w}_0=s_2s_1s_2s_3s_2s_1$ is an $\varepsilon$-adapted reduced expression, and the infinite sequence associated to it is
 \[ 
 \widehat{\underline{w_0}}=\dots\ 2,\ 1,\ 2,\ 3,\ \underbrace{2,\ 1,\ 2,\ 3,\ 2,\ 1,\ }_{[1,6]} 2,\ 3\ \  
 \dots .\]
\end{ex}

The chain $\mathfrak{C}^{[-3,8]}_{-}$ is formed by the $i$-boxes

 \begin{align*}
\mathfrak{c}_1&=[8]_3, &  \mathfrak{c}_2&=[7]_2, &\mathfrak{c}_3&=[6]_1, &\mathfrak{c}_4&=[3,7]_2,\\
\mathfrak{c}_5&=[4,8]_3, &  \mathfrak{c}_6&=[3,7]_2, &\mathfrak{c}_7&=[2,6]_1, &\mathfrak{c}_8&=[1,7]_2,\\
\mathfrak{c}_9&=[0,8]_3, &  \mathfrak{c}_{10}&=[-1,7]_2, &\mathfrak{c}_{11}&=[-2,6]_1, &\mathfrak{c}_{12}&=[-3,7]_2.
\end{align*} 

The quiver associated to the monoidal seed $\mathcal{S}(\mathfrak{C}^{[-3,8]}_{-})$ is $Q^{[-3,8]}(\underline{w}_0)$:

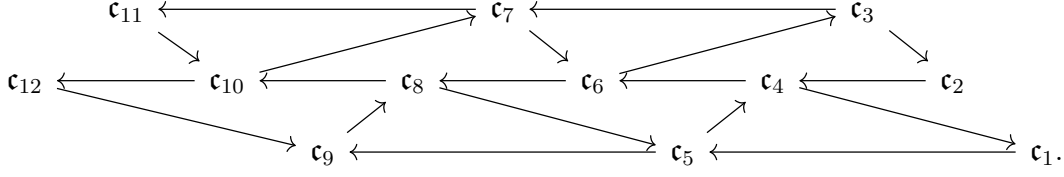
\[
\begin{tikzcd}[sep=small]
    & \mathfrak{c}_{11} \arrow[rd] &&&& \mathfrak{c}_7  \arrow[llll]\arrow[rd] &&&& \mathfrak{c}_3 \arrow[llll]\arrow[rd] & &  \\  
    \mathfrak{c}_{12} \arrow[rrrd] && \mathfrak{c} _{10}\arrow[ll]\arrow[rrru] && \mathfrak{c}_8 \arrow[rrrd] \arrow[ll]&& \mathfrak{c}_6\arrow[ll]\arrow[rrru]&& \mathfrak{c}_4\arrow[rrrd]  \arrow[ll]&& \mathfrak{c}_2 \arrow[ll]&\\
    & & & \mathfrak{c}_9 \arrow[ru] &&&& \mathfrak{c}_5 \arrow[llll]\arrow[ru] &&&& \mathfrak{c}_1 \arrow[llll] .
\end{tikzcd}
\]

An alternative chain of $i$-boxes $\mathfrak{C}'=(\mathfrak{c}_k)_{1\leq k\leq 12}$ for the interval $[-3,8]$ is given by

\begin{align*}
\mathfrak{c}'_1&=[2]_1=\tau^{-1}\mathfrak{c}_3, &  \mathfrak{c}'_2&=[3]_2=\tau^{-2}\mathfrak{c}_2, &\mathfrak{c}'_3&=[1,3]_2=\tau^{-2}\mathfrak{c}_4, &\mathfrak{c}'_4&=[4]_3=\tau^{-1}\mathfrak{c}_1,\\
\mathfrak{c}'_5&=[0,4]_3=\tau^{-1}\mathfrak{c}_5, &  \mathfrak{c}'_6&=[1,5]_2=\tau^{-1}\mathfrak{c}_6, &\mathfrak{c}'_7&=[-1,5]_2=\tau^{-1}\mathfrak{c}_8, &\mathfrak{c}'_8&=[2,6]_1=\tau^{0}\mathfrak{c}_7,\\
\mathfrak{c}'_9&=[-2,6]_1=\tau^{0}\mathfrak{c}_{11}, &  \mathfrak{c}'_{10}&=[-1,7]_2=\tau^{0}\mathfrak{c}_{10}, &\mathfrak{c}'_{11}&=[-3,7]_2=\tau^{0}\mathfrak{c}_{12}, &\mathfrak{c}'_{12}&=[0,8]_3=\tau^{0}\mathfrak{c}_9.
\end{align*} 

It corresponds to the pair $(2,(T_k)_{1\leq k\leq 11})$, where $T_k=R$ if and only if $R$ is odd.
Therefore, ordering columns and rows in such a way that $m_i<m'_j$ if $i<j$ or if $m<m'$ and $i=j$, the matrix $T$ is 

\[ T=
\setcounter{MaxMatrixCols}{20}
\begin{bmatrix}
    -1 & 0 & 0 &  0 &  0 &  0 &  0 & 0 & 0 &  0 & 0 & 0\\
    1 & 1 & 0 & 0 & 0 & 0 & 0 & 0 & 0 &  0 & 0 & 0\\
     0 &0 &1  &0  &0  &0 & 0 &0 &0  &0 &0 &0\\
    0 &0 &0  &0  &0 &-1 &-1 &0 &0  &0 &0 &0\\
    0 &0 &0 &-1 &-1  &0  &0 &0 &0  &0 &0 &0\\
    0 &0 &0  &1  &0  &0  &0 &0 &0  &0 &0 &0\\
    0 &0 &0  &0  &1  &1  &0 &0 &0  &0 &0 &0\\
    0 &0 &0  &0  &0  &0  &1 &1 &0  &0 &0 &0\\
    0 &0 &0  &0  &0  &0  &0 &0 &1  &0 &0 &0\\
    0 &0 &0  &0  &0  &0  &0 &0 &0 &-1 &-1 &0\\
    0 &0 &0  &0  &0  &0  &0 &0 &0  &1 &0 &0\\
    0 &0 &0  &0  &0  &0  &0 &0 &0  &0  &1 &1
\end{bmatrix}.\]

Therefore, by application of Theorem \ref{teo_main_KKOP}, the quiver of the monoidal seed $\mathcal{S}(\mathfrak{C}')$ is 

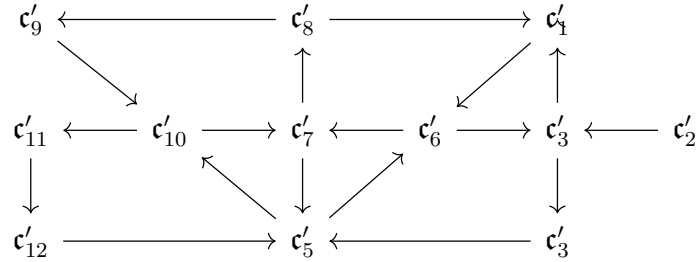
\[\begin{tikzcd}
    \mathfrak{c}'_9 \arrow[rd] & & \mathfrak{c}'_8 \arrow[ll] \arrow[rr]  & & \mathfrak{c}'_1\arrow[ld]\arrow[lr] & \\
    \mathfrak{c}'_{11}\arrow[d] & \mathfrak{c}'_{10} \arrow[l]\arrow[r]& \mathfrak{c}'_7 \arrow[d]\arrow[u]& \mathfrak{c}'_6\arrow[l]\arrow[r] &\mathfrak{c}'_3\arrow[u]\arrow[d]& \mathfrak{c}'_2 \arrow[l]\\
     \mathfrak{c}'_{12} \arrow[rr]& & \mathfrak{c}'_5\arrow[ul]\arrow[ur] & & \mathfrak{c}'_3\arrow[ll] &
\end{tikzcd}\]

\section*{Acknowledgement}
This article is part of the author's PhD thesis. The author is grateful to his supervisor Bernhard Keller for his guidance. Moreover, he is indebted to Ryo Fujita, David Hernandez, Geoffrey Janssens and Se-Jin Oh for discussions and useful comments. The author thanks the ANR CHARMS (ANR-19-CE40-0017) for financial support.


\providecommand{\bysame}{\leavevmode\hbox to3em{\hrulefill}\thinspace}
\providecommand{\MR}{\relax\ifhmode\unskip\space\fi MR }
\providecommand{\MRhref}[2]{%
  \href{http://www.ams.org/mathscinet-getitem?mr=#1}{#2}
}
\providecommand{\href}[2]{#2}

\end{document}